\UseRawInputEncoding

\documentclass[12pt]{amsart}
\usepackage{epsfig,color, amsmath}
\usepackage{esint}
\usepackage{hyperref, mathrsfs}
\usepackage{xcolor}
\usepackage{bm}
\usepackage{enumitem}
\usepackage{mathtools, comment}
  \hypersetup{colorlinks}
\hypersetup{citecolor=blue}
\hypersetup{urlcolor=blue}

\makeatother

\theoremstyle{definition}

\def\fnum{equation} 
\newtheorem{Thm}[\fnum]{Theorem}
\newtheorem{Cor}[\fnum]{Corollary}

\newtheorem{Lem}[\fnum]{Lemma}

\newtheorem{Def}[\fnum]{Definition}
\newtheorem{Exa}[\fnum]{Example}
\newtheorem{Rem}[\fnum]{Remark}
\newtheorem{Pro}[\fnum]{Proposition}
\newtheorem{Asu}[\fnum]{Assumption}

\numberwithin{equation}{section}
\renewcommand{\rm}{\normalshape} 

\newcommand{\Ker}{{\text{Ker}}}

\newcommand{\supp}{{\text{supp}}}

\newcommand{\XX}{\mathbf{X}}

  \newcommand{\R}{\ensuremath{\mathbb{R}}}

 \newcommand{\ba}{\begin{align*}}
 \newcommand{\ea}{\end{align*}}

\newcommand{\Id}{\text{Id}}
\newcommand{\iso}{\text{Iso}}
\newcommand{\aut}{\text{Aut}}
\newcommand{\loc}{\text{loc}}
\newcommand{\test}{\text{TestF}}
\DeclareMathOperator{\dr}{dt_r}
\DeclareMathOperator{\drt}{dist_{r/10}}

\newcommand{\mx}{{\text {Mx}}}

\newcommand{\rcd}{\text{RCD}^{\ast}}

\DeclareMathOperator{\RCD}{RCD}

\newcommand{\mm}{\mathfrak{m}}

\title{Margulis Lemma on RCD(K,N) spaces}
\author{Qin Deng}\address{Qin Deng, Massachusetts Institute of Technology, qindeng@mit.edu}
\author{Jaime Santos-Rodr\'iguez} \address{\parbox{\linewidth}{Jaime Santos-Rodr\'iguez, Universidad Aut\'onoma de Madrid and Durham University,\\ jaime.santos@uam.es, jaime.santos-rodriguez@durham.ac.uk}}  
\author{Sergio Zamora}\address{Sergio Zamora, Max Planck Institute for Mathematics, zamora@mpim-bonn.mpg.de}
\author{Xinrui Zhao}\address{Xinrui Zhao, Massachusetts Institute of Technology, xrzhao@mit.edu}

\begin{document}

\maketitle

\vspace{-0.5cm}

\begin{abstract}
    We extend the Margulis Lemma for manifolds with lower Ricci curvature bounds to the $\rcd (K,N)$ setting. 
    As one of our main tools, we obtain improved regularity estimates for Regular Langrangian flows on these spaces.
\end{abstract}

{
\hypersetup{linkcolor=black}
\tableofcontents
}
\section{Introduction}

The main result of this paper extends the Margulis Lemma to $\rcd (K,N)$ spaces. Recall that for a group $G$, we say an (ordered) generating set $ \beta = \{ \gamma _1 , \ldots , \gamma_n \} \subset G$ is a \textit{nilpotent basis of length} $n$ if for all $ i  , j \in \{ 1, \ldots , n \}$ one has $[\gamma _i, \gamma _j ] \in \langle \{ \gamma_{1}, \ldots , \gamma _{i-1} \} \rangle$.

\begin{Thm}\label{thm:margulis}
\rm  For each $K \in \mathbb{R}$, $N \geq 1$, there exist $\varepsilon > 0 $ and $C \in \mathbb{N} $ such that if $(X,d, \mathfrak{m} ,p)$ is a pointed $\rcd (K , N)$ space of rectifiable dimension $n$, the image of the map
\[      \pi_1(B_{\varepsilon} (p), p) \to \pi_1(X,p)                  \]
induced by the inclusion contains a subgroup of index $\leq C$  that admits a nilpotent basis of length $\leq n$.
\end{Thm}

From the work of Kapovitch--Wilking, Theorem \ref{thm:margulis} is known to hold when $X$ is a smooth Riemannian manifold \cite{KW11}.  On the other hand, Breuillard--Green--Tao proved that, after quotienting by a finite normal subgroup,  Theorem \ref{thm:margulis}  holds in more general metric spaces with nice packing properties  \cite[Corollary 11.17]{BGT12}.

The proof strategy of Theorem \ref{thm:margulis} is similar to that of Kapovitch--Wilking, including a reverse induction argument (see Theorem \ref{thm:induction-main-theorem}). Nevertheless, there are quite a few technical challenges to generalizing their arguments to a nonsmooth framework. An important tool used in \cite{KW11} is the gradient flow of smooth functions with suitable integral Hessian bounds and their associated regularity estimates. In the nonsmooth framework, these gradient flows are by necessity replaced by the Regular Lagrangian flows (RLFs) of Sobolev vector fields. Several regularity results have been obtained for RLFs in recent works \cite{BS18, BS20, BDS22}, but are not quite strong enough to give the necessary estimates, see the discussion at the end of the next subsection for more details. The main technical contribution of this paper, therefore, is to establish new regularity estimates for Regular Lagrangian flows on the $\RCD(K,N)$ spaces, which matches the best known in the smooth setting. We mention that related estimates of this type have also been employed successfully in other works to study the structure of Ricci limit spaces (see \cite{CC96, CN12, KL18}) and $\rcd (K,N)$ spaces (see  \cite{BS20, D20}). 

\subsection{Main regularity estimates on RLFs}

The following regularity result is our substitute for smoothness in the context of Regular Lagrangian flows (see \ref{def:mx} for the definition of $\mx$). 

\begin{Thm}\label{thm:essential-stability}
\rm Let $\rho > 0$, $T > 0$, $L \geq 1$, $D \geq 0$, $(X,d,\mathfrak{m})$ an $\rcd (-(N-1), N)$ space,  $b \in L^1 ( [0,T] ;  H^{1,2}_{C,s}(TX))$ a vector field with $\Vert b (t) \Vert _{\infty} \leq L$ and $\Vert \text{div}( b (t)) \Vert _{\infty} \leq D$ for all $t \in [0,T]$,  $\XX: [0,T] \times X \to X$ its RLF, and define $H : X \to \mathbb{R}$ as $H(y) : = \int_0^T \mx _{\rho} ( \vert \nabla b (t) \vert ) ( \XX_t(y) ) dt$. Then there are $\delta (D, T, N) > 0, M(D, T, N) > 0 $, such that if $x \in X$ satisfies 
\begin{equation}\label{eq:good-point-hypothesis}
    \limsup_{r \to 0} \dfrac{\mm ( \{ y \in B_r(x) \vert H(y) > \delta \} )}{ \mm (B_r(x))  }  <   \frac{1}{2},           
\end{equation}  
then there is $r_x \leq \rho / 100 $ and a representative $\tilde{\XX} : [0, T ] \times X \to X$ of the RLF to $b$ such that  for all $r \leq r_x$ the following holds:
\begin{enumerate}[label= $\bm{S}$\textbf{.\arabic*}]
    \item  There is $A_r \subset B_r(x)$ with $\mm (A_r) \geq \frac{1}{M} \mm ( B_r(x) )$ and 
    \[    \tilde{\XX} _t(A_r) \subset B_{2r}(\tilde{\XX} _t(x)) \text{ for all }t \in [0, T] .       \] \label{eq:es-st-1}
    \item  For all $t \in [0, T]$,
    \[   \frac{1}{M} \mm (B_r(x)) \leq \mm ( B_r( \tilde{\XX} _t(x) ) ) \leq M \mm ( B_r(x) ).    \] \label{eq:es-st-2}
\end{enumerate}
Moreover, $\tilde{\XX} $ can be chosen so that any point $x \in X$ satisfying \ref{eq:good-point-hypothesis} also satisfies \ref{eq:es-st-1} and \ref{eq:es-st-2} for $r$ sufficiently small (depending on $x$).
\end{Thm}

\begin{Def}
\rm Let $M (1, T, N)>0$ be given by Theorem \ref{thm:essential-stability}, 
$(X,d,\mathfrak{m})$ an $\rcd (-(N-1), N)$ space,  $b :  [0,T] \to L^2_{\loc} (TX)$ a vector field,  and $\XX: [0,T] \times X \to X$ its RLF. We say that $x \in X$ is a \textit{point of essential stability} of $\XX$ if there is $r_x > 0$ such that \ref{eq:es-st-1} and \ref{eq:es-st-2} hold for all $r \leq r_x$. 
\end{Def}

\begin{Cor}\label{cor:essential-stability}
\rm  For each  $N \geq 1$, $T \geq 0$, $D \geq 0$, $r \geq 0$, $L \geq 0$ and $\varepsilon > 0 $, there are $R \geq 1$, $\eta > 0$, such that the following holds. Let $(X,d,\mathfrak{m}, p)$ be an $\rcd (-(N-1), N)$ space,  $b \in  H^{1,2}_{C,s}(TX)$ a vector field with $\Vert b \Vert _{\infty} \leq L$, $\Vert \text{div}( b ) \Vert _{\infty} \leq D$, and $\XX: [0,T] \times X \to X$ its RLF. Assume that for all $s \in [1,R]$ one has 
\[    \fint_{B_{s}(p)}  \vert  \nabla b  \vert  ^2 d \mm  \leq \eta  .  \]
Then if $G \subset X$ denotes the set of points of essential stability of $\XX$, one has 
\[          \mathfrak (  G \cap B_r (p) ) \geq (1 - \varepsilon ) \mathfrak{m} (B_r(p)).                                    \]
\end{Cor}

We remark that for non-collapsed $\rcd (K,N)$ spaces, a version of these regularity results were obtained in \cite{BDS22} using alternative methods relying on estimates of the Green's function, which cannot be readily applied in collapsed cases. Moreover, the use of the Green's function in \cite{BDS22} resulted in the dependence of various estimates on non-structural information such as the space itself (since the Green's function naturally contains global information about $X$). This is undesirable for the application at present since we will need to consider sequences of $\RCD$ spaces and therefore cannot make use of estimates which depend on the space. Indeed this dependence can be avoided by adapting the scheme of \cite{D20}. We point out that the advantage of using the Green's function in the non-collapsed setting is that one obtains optimal infinitesimal Lipschitz estimates \cite[Theorem 1.6]{BDS22}, which does not seem to be readily obtainable using the methods employed here.

\subsection{Induction Theorem}

In this subsection we state Theorem \ref{thm:induction-main-theorem}; our main technical result from which Theorem \ref{thm:margulis} follows. Recall that for a semi-locally-simply-connected space $X$, we can identify its fundamental group $\pi_1(X)$ with the group of deck transformations of its universal cover $\tilde{X}$.

\begin{Def}
\rm Let $X$ be a semi-locally simply connected geodesic space and $\tilde{X}$ its universal cover. We say a function $f: \tilde{X} \to \tilde{X}$ is of \textit{deck type} if there is an automorphism $f_{\ast} \in \aut (\pi_1(X))$ such that for all $g \in \pi_1(X)$ and $x \in \tilde{X}$, one has $f(g(x)) = f_{\ast}(g) (f(x)).$ 
\end{Def}

\begin{Exa}
\rm If $f \in \pi_1(X)$, then it is of deck type with $f_{\ast}(g) : = f \circ g \circ f^{-1}$.
\end{Exa}

\begin{Def}
\rm For metric spaces $X, Y$, a function  $f: X \to Y$, and $r > 0 $, the \textit{distortion at scale }$r$ is defined as the map $\dr (f) : X \times X \to [0,r]$ with 
\[  \dr (f) (x_1, x_2) : = \min \{ r, \vert d_X(x_1, x_2)- d_Y(f(x_1), f(x_2))\vert \}.       \]
If $X$ is equipped with a measure $\mm $, we say that $x \in X$ is a \textit{point of essential continuity} of $f$ if there exists $r_0 > 0$ such that for all $r \leq r_0$ there is a subset $A \subset B_r(x) $ with $\mm  (A) \geq \frac{1}{2} \mm  (B_r(x))$ and $f(A) \subset B_{2r}(f(x))$.
\end{Def}

The next definition is a non-smooth version of the \textit{maps with zoom-in property} from \cite{KW11}. Although the notion is very technical, these are precisely the properties present in gradient flows of $\delta$-splittings (and as we will show, also in the RLFs of $\delta$-splittings in the $\RCD$ setting).

\begin{Def}\label{def:gas}
\rm Let $(X_i^j, d_i^j, \mm _i^j, p_i^j) $, $j \in \{ 1, 2 \}$ be two sequences of pointed $\rcd (K,N)$ spaces. We say that a sequence of measurable functions $f_i : [ X_i^1, p_i^1 ] \to  [ X_i^2, p_i^2 ]$ is \textit{good at all scales} (GS) if  there is a sequence of measurable  functions $f_i^{-1} :  [ X_i^2, p_i^2 ] \to  [ X_i^1, p_i^1 ]$ such that $f_i^{-1} \circ f_i = \Id_{X_i^1}$ almost everywhere and $f_i \circ f_i^{-1} = \Id_{X_i^2}$ almost everywhere, satisfying the following: 
\begin{enumerate}
\item $(f_i)_{\ast} (\mm  _i^1) \ll \mm _i^2$ and $(f_i^{-1})_{\ast}(\mm _i^2) \ll \mm _i^1$ for all $i$.\label{def:gas-1}
\item There is $R_0 > 0 $ and sequences $S_i^j \subset B_1(p_i^j)$ for $j \in \{ 1 , 2 \} $ with $\mm _i^j (S_i^j) \geq  \frac{1}{2} \mm _i^j (  B_1(p_i^j) ) $ and $f_i( S_i^1) \subset  B_{R_0}(p_i^2)$, $f_i^{-1}( S_i^2) \subset  B_{R_0}(p_i^1) $.\label{def:gas-2}
\item There is a sequence $\varepsilon _ i \to 0$ and sequences of subsets  $ U_i^j \subset X_i^j $ for $j \in \{ 1 , 2 \}$ such that: \label{def:gas-3}
\begin{enumerate}
\item The points of $U_i^1$ (resp. $U_i^2$) are of essential continuity of $f_i$ (resp. $f_i^{-1}$). \label{def:gas-3a}
\item $f_i$ (resp. $f_i^{-1}$) restricted to $U_i^1$ (resp. $U_i^2$) is measure preserving.\label{def:gas-3b}
\item For all $R > 0$ and $j \in \{ 1 , 2 \} $, one has 
 \[  \lim _{i \to \infty } \frac{\mm _i^j(U_i^j \cap B_R(p_i^j))}{\mm _i^j(B_R(p_i^j))} = 1.    \]   \label{def:gas-3c}
\item  For all $x_i^1 \in U_i^1,$ $ x_i^2  \in U_i^2 $,  $r \leq 1$, one has 
\begin{gather*}
     \fint _{B_r(x_i^1) ^{\times 2} } \dr (f_i) (a,b) d  (\mm _i^1 \times \mm _i^1 ) (a,b)  \leq \varepsilon_i r \\
     \fint _{B_r(x_i^2) ^{\times 2} } \dr  ( f_i^{-1} )(a,b) d  (\mm _i^2 \times \mm _i^2 )  (a,b)  \leq \varepsilon_i  r    .
\end{gather*}\label{def:gas-3d}
\end{enumerate} 
\end{enumerate}
\end{Def}
\begin{Def}
\rm  Let $\Gamma $ be a group, $G \leq \Gamma$ a subgroup admitting a nilpotent basis $\beta = \{\gamma_{1}, \ldots , \gamma _{n} \}$, and $\varphi  \in \aut (\Gamma )$. We say that $\varphi $ \textit{respects} $\beta$ if it preserves  $\langle \{ \gamma_1 , \ldots , \gamma_m \} \rangle$ for each $m$, and  acts trivially on $\langle \{ \gamma_1 , \ldots , \gamma_{m} \} \rangle / \langle \{ \gamma_{1} , \ldots , \gamma_{m-1} \} \rangle $ for each $m$.
\end{Def}

\begin{Thm}\label{thm:induction-main-theorem}
\rm  Let $(X_i, d_i, \mathbf{m}_i, p_i )$ be a sequence of pointed $\rcd \left(  -\frac{1}{i}, N \right)$ spaces of rectifiable dimension $n$ and a pointed compact metric space $(Y,y)$ of diameter $D$ for which the sequence $(X_i,p_i)$ converges in the Gromov--Hausdorff sense to $(\mathbb{R}^k \times Y, (0,y) )$. Let $\tilde{X}_i$ be the sequence of universal covers,  $\tilde{p}_i \in \tilde{X}_i$ in the preimage of $p_i$, $\Gamma_i \leq \pi_1(X_i)$ be the group generated by the elements $g \in \pi_1(X_i)$ with $d( g \tilde{p}_i , \tilde{p}_i ) \leq 2D + 1$, and for each $j \in \{1, \ldots , \ell \}$,  $f_{j,i} : [\tilde{X}_i, \tilde{p}_i]\to [\tilde{X}_i, \tilde{p}_i]$ a sequence of deck type maps with the GS property. Then for some $C>0$ and $i$ large enough, $\Gamma_i$ contains a subgroup $G_i \leq \Gamma_i$ with the following properties
\begin{itemize}
    \item $[\Gamma_i, G_i] \leq C$ .
    \item $G_i$ admits a nilpotent basis $\beta_i$ of length $\leq n-k$.
    \item $\left( f_{j,i} \right)_{\ast}^{C !}$ respects $\beta _i $ for each $j$.
\end{itemize}
\end{Thm}

\subsection{Structure of the paper}  

In Section \ref{sec:preamble}, we cover the background material we will need. In Section \ref{sec:gcc}, we prove Theorem \ref{thm:upsilon}, which provides us with subgroups $\Upsilon_i \triangleleft \Gamma _ i $ that play the role of identity connected components in the discrete groups $\Gamma_i$.

In Section \ref{sec:pmre} we prove Theorem \ref{thm:essential-stability} and Corollary \ref{cor:essential-stability}, allowing us to find points of essential stability, and in section \ref{sec:sis} we study how essential stability allows one to obtain stronger estimates. In Section \ref{sec:pgas} we prove properties of GS maps, and in Section \ref{sec:cgas} we give two ways to construct GS maps (cf. \cite[Section 3]{KW11}).

In Section \ref{sec:res} we show Theorem \ref{thm:re-scaling}, reducing Theorem \ref{thm:induction-main-theorem} to the case  $Y \neq \{ \ast \}$ (cf. \cite[Section 5]{KW11}). In Section \ref{sec:pmt} we prove Theorem \ref{thm:induction-main-theorem} and with it Theorem \ref{thm:margulis}.
 \section*{Acknowledgement}
Qin Deng and Xinrui Zhao are grateful to Prof. Tobias Colding for his interest on this project. Jaime Santos--Rodr\'iguez is supported in part by a Margarita Salas Fellowship CA1/RSUE/2021--00625, and by research grants  MTM2017-‐85934-‐C3-‐2-‐P, PID2021--124195NB--C32 from the Ministerio de Econom\'ia y Competitividad de Espa\~{na} (MINECO). Sergio Zamora holds a Postdoctoral Fellowship at the Max Planck Institute for Mathematics at Bonn, Xinrui Zhao is supported by NSF Grant DMS 1812142 and NSF Grant DMS 2104349.

\section{Preamble}\label{sec:preamble}

\subsection{Notation}

For a set $A$, we denote by $A^{\times 2} $ the set $A \times A$. If $A \subset X$, we denote by $\chi _A : X \to [0,1] $ the characteristic function of $A$. For a group $G$ and $g\in G$, we denote by $g_{\ast}  \in \aut (G) $ the map $h \mapsto ghg^{-1}$.  For metric spaces $(X, d_X)$ and $(Y, d_Y)$, we denote by $X \times Y$ the $L^2$ product. That is, for $x_1, x_2 \in X$, $y_1, y_2 \in Y$, 
\[   d_{X \times Y}((x_1, y_1), (x_2, y_2)) : =  \sqrt{  d_X(x_1,x_2)^2 + d_Y(y_1,y_2) ^2 }.        \]
We say a pointed metric measure space $(X,d, \mathfrak{m}, p)$ is \textit{normalized} if 
\[   \int_{B_1(p)}(1-d(p,\cdot ))d\mm  = 1.  \]
For $m \in  \mathbb{N}$, we denote by $\mathbb{R}^m$ the $m$-dimensional Euclidean space equipped with its usual metric, and by $\mathcal{H}^m$ the $m$-dimensional Hausdorff measure for which the metric measure space $(\mathbb{R}^m, \mathcal{H}^m, 0)$ is normalized. 

To a metric space $(X,d)$, we can adjoin a point $\ast$ at infinite distance from any point of $X$ to get a new space we denote by $X \cup \{ \ast \}$. Similarly, to any group $G$ we can adjoin an element $\ast$ whose product with any element of $G$ is defined as $\ast$, obtaining a binary operation on $G \cup \{ \ast \}$.

 We write $C(\alpha , \beta ,\gamma )$ to denote a constant $C$ that depends only on the quantities $\alpha , \beta , \gamma $.

\subsection{$\rcd (K,N)$ spaces; doubling, isometries, covers, and geodesics}

One of the most powerful tools in the study of $\rcd (K,N)$ spaces is the Bishop--Gromov inequality \cite{BS10}.

\begin{Thm}\label{thm:bg-in}
\rm (Bacher--Sturm) For each $K \in \mathbb{R}$, $N \geq 1$, $R>0$, $\lambda > 1 $ there is $C(K,N, R, \lambda )> 0 $ such that for any pointed $\rcd (K,N)$ space $(X, d, \mm ,p)$, and any $r \leq R $, one has
\[\mm (B_{\lambda r}(p)) \leq C \cdot \mm (B_r(p)) . \]
Moreover, for fixed $K , N, R$, if $\lambda \to 1$ then $C \to 1$.
\end{Thm} 

\begin{Cor}\label{cor:a-full-measure}
\rm Let $(X_i, d_i , \mm _i, p_i)$ be a sequence of pointed $\rcd (K,N)$ spaces and consider a sequence of subsets $U_i \subset X_i$. Then the following are equivalent:
\begin{itemize}
    \item For all $R > 0$, there is a sequence $\eta_i \to 0$ such that
    \[     \mm _i (U_i \cap B_R(p_i)) \geq (1 - \eta_i )\mm _i (B_R(p_i)) .                     \]
    \item  For all $R >  \delta  > 0$, there is a sequence $\varepsilon _i \to 0$ such that if $x \in B_R(p_i)$, one has
    \[      \mm _i(U_i \cap B_{\delta}(x)) \geq            (1 - \varepsilon _i ) \mm _i (B_{\delta}(x)).     \]
\end{itemize}
In either case, we say that the sequence $U_i$ has \textit{asymptotically full measure}.
\end{Cor}

\begin{Def}\label{def:mx}
\rm  Let $(X, d, \mm ) $ be an $\rcd (K,N)$ space, $R > 0 $,  and $h : X \to \mathbb{R}^{+} $ measurable. The $R$\textit{-maximal function} $\mx_R(h): X \to \R$ is defined as
\begin{equation*}
    \mx_R(h)(x)= \sup_ {0 < r \leq R} \fint_{B_r(x)} h \,  d \mm  .
\end{equation*} 
For simplicity, we denote $\mx _1$ by $\mx$. 
\end{Def}

The following well known facts follow from Theorem \ref{thm:bg-in} (see \cite[p.12]{S93} and \cite[p.6]{KW11}). 

\begin{Pro}\label{pro:max-properties}
\rm  Let $(X, d, \mm ) $ be an $\rcd (K,N)$ space,  $h : X \to \mathbb{R}^{+} $ measurable, and $R > 0 $. Then  
\begin{enumerate}
\item For all $\delta > 0$, 
\[ \mm (  \{ x \in X  \vert \mx _R(h) (x) \geq \delta \}  ) \leq \frac{C(K,N, R )}{ \delta } \int_{X} h \,  d \mm  . \]\label{pro:max-properties-1}
\item For all $\alpha > 1$, 
\[ \Vert \mx _R( h) \Vert _{\alpha} \leq C(K, N, R, \alpha ) \Vert h \Vert _{\alpha}. \]\label{pro:max-properties-2}
\item  For all $\alpha > 1$,  $s \leq R/2$,
\[   \mx_s(  \mx_s ( h )^{\alpha}) \leq C(K,N,R, \alpha ) \mx _R( h^{\alpha} ).        \]\label{pro:max-properties-3}
\end{enumerate}
\end{Pro}

For a proper metric space $X$, the topology that we use on its group of isometries $\iso(X)$ is the compact-open topology, which in this setting coincides with both the topology of pointwise convergence and the topology of uniform convergence on compact sets. This topology makes $\iso(X)$ a locally compact second countable metric group. In the case $(X,d,\mm )$ is an $\rcd (K,N)$ space,  $\iso (X)$ is a Lie group \cite{GS19, S18}.

\begin{Thm}\label{thm:gss-lie}
\rm (Sosa, Guijarro--Santos) Let $(X,d,\mm )$ be an $\rcd (K,N)$ space. Then $\iso(X)$ is a Lie group.
\end{Thm}

The $\rcd (K,N)$ condition can be checked locally (see \cite[Section 3]{EKS15}), hence if $(X,d, \mathfrak{m} )$ is an $\rcd (K,N)$ space and $\rho : \tilde{X} \to X$ is a covering space, $\tilde{X}$ admits a unique measure making it an $\rcd (K,N)$ space, and for which  $\rho$ is a local isomorphism of metric measure spaces (see \cite[Section 2.3]{MW19}). Whenever we have a covering space of an $\rcd (K,N)$ space, we assume it is equipped with such measure.  This allows one to lift estimates on maximal functions \cite[Lemma 1.6]{KW11}.

\begin{Pro}\label{pro:mx-lift}
\rm Let $(X,d, \mm ) $ be an $\rcd (K,N)$, $ \rho : ( \tilde{X}, \tilde{d}, \tilde{\mm } ) \to (X,d, \mm )$ a covering space, $x\in X$,  $\tilde{x}\in \rho ^{-1}(x)$, $f : X \to \mathbb{R}^{+}$ measurable. Then for all $r \leq R$, one has
\[    \fint_{B_{r}(\tilde{x})} (f \circ \rho) \,  d \tilde{ \mm }   \leq C(K,N, R) \fint_{B_r(x)} f \,   d \mm  .  \]
In particular, 
\[   \mx_R(f \circ \rho) \leq C(K,N, R) \cdot \mx_R( f) \circ \rho .              \]
\end{Pro}

An important topological property of $\rcd (K,N ) $  spaces is that they are semi-locally-simply-connected \cite{W22}.

\begin{Thm}\label{thm:jikang}
\rm (Wang) Let $(X,d,\mm )$ be an $\rcd (K,N)$ space. Then $X$ is semi-locally-simply-connected, so its universal cover $\tilde{X}$ is simply connected and we can identify $\pi_1(X)$ with the group of deck transformations $\tilde{X} \to \tilde{X}$.
\end{Thm}

To conclude this subsection, we note that by the Kuratowski and Ryll-Nardzewski measurable selection theorem, for any $\rcd (K,N)$ space $(X,d, \mm )$, there is a measurable map 
\[ \gamma _{\cdot , \cdot} (\cdot ) : X \times X \times [0,1] \to X  \] 
such that for all $x,y \in X$, the map $[0,1] \ni s \mapsto \gamma _{x,y} (s) $ is a constant speed geodesic from $x$ to $y$. For the rest of this paper, for each $(X,d,\mathfrak{m} ) $ we fix such a choice of $\gamma$. This allows us to state the segment inequality for $\rcd (K,N)$ spaces \cite[Theorem 3.22]{D20}.

\begin{Thm}\label{thm:segment}
\rm Let $(X,d, \mm )$ be an $\rcd (K,N)$ space,  $h : X \to \mathbb{R}^+$ measurable, $p \in X$, and $r \leq R $. Then
\begin{equation*}
\fint_{B_r(p)^{\times 2}} d(x,y) \left[ \int_0^1 h  \left(  \gamma _{x,y}(s)  \right) ds \right] d (\mm \times \mm )( x , y)  \leq r \cdot  C(K,N,R)  \fint _{B_{2r}(p)} h \,  d \mm  .
\end{equation*}
\end{Thm}

We will also need the following variation of the Lebesgue Differentiation Theorem.

\begin{Def}
\rm   Let $(X,d,\mm )$ be a metric measure space. We say that a family of measures $\mathcal{V} $ on $X$ has \textit{bounded eccentricity} if there are $M \geq 1 \geq  \eta > 0$ such that $ \nu \leq  M \mm $ for all $\nu \in \mathcal{V}$, and a map $ \theta : \mathcal{V} \to X$  such that for all $\nu \in \mathcal{V}$ there is $r(\nu ) > 0$ with  $ \supp ( \nu ) \subset B_{r} (\theta ( \nu ) ) $ and  $  \nu (B_r(\theta (\nu ))) \geq \eta \,  \mm (B_r(\theta ( \nu )))  $. We then say that a net $\nu_i \in \mathcal{V}$ converges to $x \in X$ if $\theta (\nu_i ) = x$ for all large $i$ and $r(\nu_i) \to 0$. 
\end{Def}

\begin{Lem}\label{lem:leb}
Let $(X,d,\mm )$ be an $\rcd (K,N) $ space,  $f\in L^1(\mm)$, and $\mathcal{V}$ a family of measures of bounded eccentricity. Then for $\mm$-almost every $x \in X$ we have 
\begin{align*}
    f(x)=\lim_{\nu\to x }\frac{1}{\nu(X)}\int_X f d\nu.
\end{align*}
\end{Lem}
\begin{proof}
For $\alpha > 0 $, define
\[ E_\alpha = \bigg \{ x \in X :\limsup_{ \nu \to x}  \frac{1}{\nu(X)} \bigg|\int_Xf(y) -f(x)\, d\nu \bigg|  > 2\alpha \bigg\}. \]
Given $\varepsilon >0$,  pick a continuous function $g \in L^1(\mm ) $ with
\begin{equation}\label{eq:g-approx}
    \|f-g\|_{L^1(\mm)}\leq \varepsilon.
\end{equation}
For $\nu \in \mathcal{V}$ with $r(\nu ) \leq 1$ and $\theta (\nu ) = x$ we have
\begin{equation}\label{eq:leb}
\begin{split}
   & \frac{1}{\nu(X)} \big \vert \int_X ( f(y)-f(x) ) d \nu (y) \big \vert  \\
    & \leq   \frac{1}{\nu(X)} \big \vert \int_X ( f(y)-g(y) ) d \nu (y) \big \vert +  \frac{1}{\nu(X)} \big \vert \int_X (g(y)-g(x)) d \nu (y) \big \vert + \big \vert g(x)-f(x) \big \vert . 
\end{split}
\end{equation}
Since $g $ is continuous, for all $x \in X$ we have
\begin{equation}\label{eq:leb-con}
    \lim_{\nu \to x}    \frac{1}{\nu(X)} \big \vert \int_X (g(y)-g(x)) d \nu (y) \big \vert = 0. 
\end{equation}             
To deal with the first summand, we compute
\begin{equation}\label{eq:leb-mx}
\begin{split}
     & \frac{1}{\nu(X)} \big \vert \int_X ( f(y)-g(y) ) d \nu (y) \big \vert \\
     & \leq  \frac{M}{ \eta \, \mm (B_r(x) ) }  \int_{B_r(x)} \vert f(y) - g(y) \vert d \mm   \leq \frac{M}{\eta} \mx \big( \vert  f-g \vert \big) (x).
\end{split}
\end{equation}
Combining \ref{eq:leb}, \ref{eq:leb-con}, and \ref{eq:leb-mx}, we get
\begin{equation*}
\begin{split}
    E_{\alpha} \subset \left\{  \mx \big( \vert f - g \vert \big) \geq \frac{ \eta \,  \alpha }{M} \right\} \cup \{ \vert f - g \vert \geq \alpha \}.
\end{split}
\end{equation*}
Then from \ref{eq:g-approx} and Proposition \ref{pro:max-properties}(\ref{pro:max-properties-1}) we obtain
\[    \mm (E_{\alpha}) \leq \frac{ C(K,N)M  }{\eta \alpha} \varepsilon.   \]
Since $\varepsilon $ was arbitrary we get $\mm (E_{\alpha}) = 0$, and hence the result.
\end{proof}

\subsection{Gromov--Hausdorff topology}

\begin{Def}\label{def:pmgh}
\rm Let $(X_i,p_i)$ be a sequence of pointed proper metric spaces. We say that it \textit{converges in the Gromov--Hausdorff sense} to a proper pointed metric space $(X,p)$ if there is a sequence of functions $\varphi _i : X_i \to X \cup \{ \ast \} $ with $\varphi _i(p_i ) \to p $  such that for each $R>0$,
\begin{gather*}
 \varphi_i^{-1} (B_{R}(p)) \subset B_{2R}(p_i) \text{ for } i \text{ large enough}, \\
\lim_{i \to \infty } \sup_{x_1,x_2 \in B_{2R}(p_i)} \vert d(\varphi _i(x_1),\varphi _i(x_2)) - d(x_1,x_2) \vert =0  , \\
  \lim_{i \to \infty} \sup_{y \in B_R(p) } \inf_{x \in B_{2R}(p_i)}  d(\varphi  _i (x),y)  = 0. 
\end{gather*}
If in addition to that, $(X_i,d_i,\mm _i)$, $(X,d,\mm )$ are metric measure spaces, the maps $\varphi _i$ are Borel measurable, and
\[  \int_X f \cdot d((\varphi _i)_{\ast}\mm _i) \to \int_X f \cdot d \mm      \]
for all $f: X \to \mathbb{R}$ bounded continuous with compact support, then we say that $(X_i,d_i, \mm _i,p_i)$ converges to $(X,d,\mm ,p)$ in the \textit{measured Gromov--Hausdorff sense}.
\end{Def}

\begin{Rem}
\rm Whenever we say that a sequence of spaces $X_i$ converges in the  Gromov--Hausdorff sense to some space $X$, we implicitly assume the existence of the maps $\varphi _i$, called \textit{Gromov--Hausdorff approximations} satisfying the above conditions, and if a sequence $x_i \in X_i$ is such that $\varphi _i (x_i) \to x \in X$, by an abuse of notation we say that $x_i$ \textit{converges} to $x$. 
\end{Rem}

The topology induced by this convergence is also given by a metric \cite{G81}.

\begin{Thm}
\rm (Gromov) There is a metric $d_{GH}$ in the class of pointed proper metric spaces modulo pointed isometry with the property that a sequence $(X_i, p_i) $ converges to a space $(X,p)$ in the Gromov--Hausdorff sense if and only if $d_{GH}((X_i, p_i), (X,p)) \to 0$. 
\end{Thm}

\begin{Rem}\label{rem:100k}
\rm The only property we will need about this metric is that if $(Y,y)$ is a pointed compact geodesic space for which 
\[  d_{GH}((\mathbb{R}^k \times Y, (0,y)), (\mathbb{R}^k, 0)) \leq  1/100   \text{ for some } k \in \mathbb{N}, \]
then diam$(Y) \leq 1/10$.
\end{Rem}

One of the main features of the class of $\rcd (K,N)$ spaces is the compactness property \cite{G81, BS10}. Notice that for any pointed $\rcd (K,N)$ space $(X,d, \mm, p)$,  there is a unique $c> 0$ for which $(X,d,c\mm ,p)$ is normalized.

\begin{Thm}\label{thm:compactness}
\rm (Gromov) If $(X_i,d_i, \mm _i,p_i)$ is a sequence of pointed  $\rcd (K,N)$ spaces, then one can find a subsequence for which $(X_i, p_i)$ converges in the Gromov--Hausdorff sense to some pointed proper geodesic space $(X,p)$.
\end{Thm}

\begin{Thm}\label{thm:compactness-2}
\rm (Bacher--Sturm) The class of pointed normalized $\rcd (K,N) $ spaces is closed under measured Gromov--Hausdorff convergence. Moreover, if $(X_i,d_i, \mm_i, p_i)$ is a sequence of $\rcd (K- \varepsilon_i,N)$ spaces such that $\varepsilon_i \to 0$ and $(X_i, p_i)$ converges in the Gromov--Hausdorff sense to a pointed proper metric space $(X,p)$, then $X$ admits a measure $\mm $ that makes it a normalized $\rcd (K,N)$ space, and after passing to a subsequence, there are  $c_i > 0$ for which $(X_i, d_i, c_i\mm_i, p_i)$ converges in the measured Gromov--Hausdorff sense to $(X,d,\mm, p)$.
\end{Thm}

\begin{Def}
\rm Let $(X,d,\mm )$ be an $\rcd (K,N)$ space and $m \in \mathbb{N}$. We say that $p \in X$ is an $m$\textit{-regular} point if for each $\lambda_i \to \infty$, the sequence $(\lambda_i X, p)$ converges in the Gromov--Hausdorff sense to $(\mathbb{R}^m, 0)$.
\end{Def}
 
Mondino--Naber showed that the set of regular points in an $\rcd (K,N)$ space has full measure \cite{MN19}. This result was refined by Bru\'e--Semola who showed that most points have the same local dimension \cite{BS20}.

\begin{Thm}\label{thm:bs-dimension}
\rm (Bru\'e--Semola) Let $(X,d,\mm )$ be an $\rcd (K,N)$ space. Then there is a unique $m \in \mathbb{N} \cap [0, N]$ such that the set of $m$-regular points in $X$ has full measure. This number $m$ is called the \textit{rectifiable dimension }of $X$.
\end{Thm}

The Cheeger--Gromoll splitting theorem was extended by Gigli to this setting \cite{G14}.

\begin{Thm}\label{thm:gigli}
\rm (Gigli) Let $(X,d,\mm )$ be an $\rcd (0, N)$ space of rectifiable dimension $n$ and assume the metric space $(X,d)$ contains an isometric copy of $\mathbb{R}^{m}$, then there is $c>0$ and an $\rcd (0, N-m)$ space $(Y, d^Y, \nu )$ of rectifiable dimension $n-m$ such that $(X,d,c \mm )$ is isomorphic to the product $(\mathbb{R}^m\times Y, d ^{\mathbb{R}^m} \times d^Y, \mathcal{H}^m \otimes \nu)$. In particular $m \leq n$, and if $m =n$ then $Y$ is a point.
\end{Thm}

\begin{Cor}\label{thm:gigli-corollary-2}
\rm Let $(X_i, d_i, \mm _i, p_i)$ be a sequence of pointed normalized $\rcd (- \delta _i , N)$ spaces with $\delta_i \to 0$. If $(X_i, p_i)$ converges in the Gromov--Hausdorff sense to $(\mathbb{R}^k, 0)$, then $(X_i, d_i, \mm _i, p_i)$ converges to $(\mathbb{R}^k, d^{\mathbb{R}^k}, \mathcal{H}^k, 0)$ in the measured Gromov--Hausdorff sense as well.
\end{Cor}

\begin{Cor}\label{cor:split-cover}
\rm Let $(\tilde{Y}, d, \mm )$ be an $\rcd  (0,N)$ space of rectifiable dimension $n$ for which $ \tilde{Y} / \iso (\tilde{Y})$ is compact.  Then there are $m \leq n$ and a compact metric space $Z $ for which $\tilde{Y}$ is isometric to the product $\mathbb{R}^m\times Z$.
\end{Cor}

The rectifiable dimension is lower semi-continuous  \cite{K19}.

\begin{Thm}\label{thm:dim-semicont}
\rm (Kitabeppu) Let $(X_i, d_i, \mm _i, p_i)$ be a sequence of pointed $\rcd (K,N) $ spaces of rectifiable dimension $m$. Assume $(X_i, p_i)$ converges in the Gromov--Hausdorff sense to $(X,p)$. If $\mm$ is a measure on $X$ that makes it an $\rcd (K,N)$ space, then $(X,d,\mm)$ has rectifiable dimension at most $m$.
\end{Thm}

\subsection{Equivariant Gromov--Hausdorff convergence}

In the setting of Gromov--Hausdorff convergence, there is a notion of convergence of group actions  \cite[Section 3]{FY92}. For a pointed proper metric space $(X,p)$, we equip its isometry group $\iso (X)$ with the metric $d_0^{p}$ given by
\begin{equation}\label{d0}
d_{0}^{p} (h_1, h_2) : = \inf_{r > 0 } \left\{ \frac{1}{r} + \sup_{x \in B_r(p)} d(h_1x, h_2x)  \right\}          
\end{equation}
for $h_1, h_2 \in \iso (X)$. It is easy to see that this metric is left invariant, induces the compact-open topology, and makes $\iso (X)$ a proper metric space.

Recall that if a sequence of pointed proper metric spaces $(X_i, p_i)$ converges in the Gromov--Hausdorff sense to the pointed proper metric space $(X,p)$,  one has Gromov--Hausdorff approximations $\varphi _i : X_i \to X \cup \{ \ast \} $.

\begin{Def}\label{def:equivariant}
\rm Consider a sequence of pointed proper metric spaces $(X_i,p_i)$ that converges in the Gromov--Hausdorff sense to a pointed proper metric space $(X,p)$, a sequence of closed groups of isometries $\Gamma_i \leq \iso(X_i)$, and a  closed group $\Gamma \leq \iso (X)$. Equip $\Gamma_i $ with the metric $d_0^{p_i}$ and $\Gamma$ with the metric $d_0^p$. We say that the sequence $\Gamma_i$ \textit{converges equivariantly} to $\Gamma$ if there is a sequence of Gromov--Hausdorff approximations $\psi_i : \Gamma _ i \to \Gamma  \cup \{ \ast \} $ such that for each $R > 0 $ one has 
\[  \lim_{i \to \infty} \sup_{g \in B_R(\Id_{X_i}) } \sup_{ x \in B_R(p_i)}  d( \varphi  _i( g x ), \psi_i(g) ( \varphi _i x ) ) = 0       . \]
\end{Def}

Isometry groups of proper spaces satisfy a compactness property \cite[Proposition 3.6]{FY92}.

\begin{Thm}\label{thm:equivariant-compactness}
\rm (Fukaya--Yamaguchi) Let $(Y_i,q_i) $ be a sequence of proper metric spaces that converges in the Gromov--Hausdorff sense to a proper space $(Y,q)$, and take a sequence $\Gamma_i \leq \iso(Y_i)$ of closed groups of isometries. Then, after taking a subsequence,  $\Gamma_{i}$ converges equivariantly  to a closed group $\Gamma \leq \iso(Y)$, and the sequence $(Y_i/\Gamma_i, [q_i])$ converges in the Gromov--Hausdorff sense to $(Y/\Gamma , [q])$. Moreover, if $\rho_i : Y_i \to Y_i / \Gamma _ i$, $\rho : Y \to  Y / \Gamma $ are the projections, there are $\delta_i \to 0$, $R_i \to \infty$, and  Gromov--Hausdorff approximations $\tilde{\varphi}_i: Y_i \to Y  \cup \{ \ast \} $, $\varphi_i : Y_i / \Gamma _i \to Y / \Gamma  \cup \{ \ast \} $ such that for all $x \in B_{R_i}(q_i)$ one has 
\begin{equation}\label{eq:fy}
     d( \varphi_i (\rho_i(x)) , \rho ( \tilde{\varphi}_i(x)  )   ) \leq \delta_i    .          
\end{equation}  
\end{Thm}

As a consequence of Theorems \ref{thm:gigli} and \ref{thm:equivariant-compactness}, one gets the following well-known result.

\begin{Pro}\label{pro:gigli-corollary}
\rm  For each $i \in \mathbb{N}$, let $(X_i, d_i, \mm _i, p_i)  $ be a  pointed $\rcd (- \frac{1}{i} ,N)$ space of rectifiable dimension $n$.  Assume $(X_i, p_i)$ converges in the Gromov--Hausdorff sense to $(X, p)$, there is a sequence of closed groups of isometries $\Gamma_i \leq \iso (X_i)$ that converges equivariantly to $\Gamma \leq \iso(X)$, and the sequence of pointed metric spaces $(X_i/\Gamma_i , [p_i])$ converges in the Gromov--Hausdorff sense to $(\mathbb{R}^k \times Y,(0,q))$ for some pointed proper metric space $(Y, q)$. 

Then there is a pointed metric space $(\tilde{Y}, \tilde{q})$ for which  $X$ is isomorphic to the product $\mathbb{R}^k\times {\tilde{Y}}$,  
the $\Gamma$-action respects the splitting $\mathbb{R}^k \times \tilde{Y}$, and acts trivially on the first factor. In particular, if $k = n$, then $\tilde{Y}$ is a point.
\end{Pro}

\begin{proof}  By Theorem \ref{thm:equivariant-compactness},  $X/\Gamma  = \mathbb{R}^k \times Y$, and one can use the submetry $\rho : X \to  X/\Gamma $ to lift the lines of $\mathbb{R}^k$ to lines in $X$ passing through $p$. By Theorem \ref{thm:compactness-2}, $X$ admits a measure that makes it an $\rcd (0, N)$ space, so by  \ref{thm:gigli} and \ref{thm:dim-semicont},  we get the desired splitting $X = \mathbb{R}^k \times \tilde{Y}$ with the property that $\rho ( x , \tilde{q} )=(x,q)$ for all $x \in \mathbb{R}^k$.

 Now we show that the action of $\Gamma $ respects  the $\tilde{Y}$-fibers. Let $g \in \Gamma$ and assume $g(x_1,\tilde{q} ) = (x_2,y)$ for some $x_1,x_2 \in \mathbb{R}^k$, $y \in \tilde{Y}$. Then for all $t \geq 1$, one has
\begin{eqnarray*}
t \vert x_1-x_2 \vert & = & d( \rho  ( x_1+ t(x_2-x_1),\tilde{q}) , \rho  ( x_1,\tilde{q}) ) \\
& = & d(  \rho   ( x_1+ t(x_2-x_1),\tilde{q}) ,  \rho  (x_2,y)  )        \\
& \leq & d^X (  ( x_1 + t (x_2-x_1),\tilde{q}) , (x_2,y) )\\
& = & \sqrt{    \vert  (t-1) (x_2-x_1) \vert ^2  + d^{\tilde{Y}}(\tilde{q},y)^2 }.
\end{eqnarray*} 
As $t \to \infty$, this is only possible if $x_1=x_2$. This shows that $g(x , \tilde{q}) = (x, y)$ for some $y \in \tilde{Y}$ independent of $x \in \mathbb{R}^k$. Now assume $g(x_1, z) = (x_2, z^{\prime})$ for some $z,z^{\prime} \in \tilde{Y}$. Then 
\begin{eqnarray*}
t ^2 \vert x_1-x_2 \vert ^2 + d ^{\tilde{Y}}(\tilde{q},z)^2 &  = & d^X (( x_1+ t(x_2-x_1),\tilde{q}) ,   ( x_1,z) )^2 \\
& = & d^X( ( x_1+ t(x_2-x_1),y) , (x_2, z^{\prime}) )^2        \\
& = &  \vert  (t-1) (x_2-x_1) \vert ^2  + d^{\tilde{Y}}(y,z^{\prime})^2.
\end{eqnarray*} 
This is only possible if $x_1 = x_2$, showing that $\Gamma$ acts trivially on the $\mathbb{R}^k$-factor.
\end{proof}

\subsection{$\delta$-splittings}

Let us recall some results on $\delta$-splittings. For proof and detailed discussions see for example \cite[section 3.1]{BPS23}.

\begin{Lem}\label{lem:convergence-to-delta}
\rm Let $(X_i, d_i, \mm _i, p_i) $  be a sequence of $ \rcd (- \frac{1}{i} , N)$ spaces for which $(X_i,p_i)$  converges in the Gromov--Hausdorff sense to $(\mathbb{R}^k\times Y , (0, y))$ for some metric space $ (Y,y) $. Then for any sequence of Gromov--Hausdorff approximations $\varphi _i : X_i \to \mathbb{R}^k \times Y \cup \{ \ast \} $, there are sequences $\delta_i \to 0$, $R_i \to \infty$, and a sequence of $L(N)$-Lipschitz functions $h^i \in H^{1,2}(X_i; \mathbb{R}^k)$ such that
\begin{itemize}
\item $\nabla h^i $ is divergence free in $B_{R_i}(p_i)$.
\item For all $r \in [1, R_i]$, one has
\[   \fint _{B_r(p_i)} \left[  \sum_{j_1, j_2 =1}^k \vert \langle \nabla h_{j_1}^i, \nabla h_{j_2}^i \rangle - \delta_{j_1, j_2} \vert + \sum_{j=1}^k  \vert  \nabla \nabla h^i_j  \vert  ^2   \right] d \mm _i  \leq \delta_i^2.          \]
\item  For all $x \in B_{R_i}(p_i) $ one has 
\begin{equation}  \label{eq:d-split-1}
\vert h^i (x) - \pi ( \varphi _i x) \vert \leq \delta_i , 
\end{equation}
where $\pi : \mathbb{R}^k \times Y\to \mathbb{R}^k$ is the projection.    
\end{itemize}
\end{Lem}

\begin{Lem}\label{lem:delta-to-convergence}
\rm Let $(X_i, d_i, \mm _i, p_i) $ be a sequence of $ \rcd (-\frac{1}{i} , N)$. Assume there are sequences $\delta_i \to 0$, $R_i \to \infty$, and a sequence of $L$-Lipschitz functions $h^i \in H^{1,2}(X_i; \mathbb{R}^k)$ with $h^i (p_i) = 0$ for all $i$ and such that
\begin{itemize}
    \item $\nabla h^i $ is divergence free in $B_{R_i}(p_i)$.
    \item For all $r \in [1, R_i]$, one has
    \[   \fint _{B_r(p_i)} \left[  \sum_{j_1, j_2 =1 }^k \vert \langle \nabla h_{j_1}^i, \nabla h_{j_2}^i \rangle - \delta_{j_1, j_2 } \vert + \sum_{j=1}^k  \vert  \nabla \nabla h^i_j  \vert ^2   \right] d \mm _i  \leq \delta_i^2.          \]
\end{itemize}
Then, after taking a subsequence, there is a metric space $ (Y,y) $ and a sequence of Gromov--Hausdorff approximations $ \varphi _i : X_i  \to  \mathbb{R}^k \times Y  \cup \{ \ast \}  $ for which 
\begin{equation}\label{eq:d-split-2}
    \sup_{x \in B_{R_i}(p_i)} \vert h^i (x) - \pi (  \varphi _i x) \vert \to 0 \text{ as }i \to \infty ,     
\end{equation}   
where $\pi : \mathbb{R}^k \times Y\to \mathbb{R}^k$ is the projection. 
\end{Lem}

\begin{Rem}
\rm In the literature, Lemmas \ref{lem:convergence-to-delta} and  \ref{lem:delta-to-convergence} are often stated without Equations \ref{eq:d-split-1} and \ref{eq:d-split-2}. However, these equations follow from how the functions $h_i$ (resp. $\varphi_i$) are constructed in the proof of Lemma \ref{lem:convergence-to-delta} (resp. \ref{lem:delta-to-convergence}). Similarly, the maps $h_i$ are usually only defined on balls around $p_i$ with radii going to infinity, but thanks to the existence of good cut-off functions \cite[Lemma 3.1]{MN19}, we can assume they are fully defined on the  spaces $X_i$.
\end{Rem}

\subsection{Regular Lagrangian Flows}

In $\rcd (K,N)$ spaces, there exist flows of certain Sobolev vector fields. 



\begin{Def}\label{def:rlf}
Let $(X,d, \mm)$ be an $\rcd (K,N)$ space, $T > 0$, and $V : [0,T] \to L^2_{\loc }(TX) $ a time-dependent  vector field.  A Borel map $ \XX :[0,T] \times X \to X$ is called a \textit{Regular Lagrangian flow (RLF)} to $V$ if the following holds:
	\begin{enumerate}[label= $\bm{R}$\textbf{.\arabic*}]
	\item $ \XX _0(x) = x$ and $[0,T] \ni t \mapsto  \XX _t(x)$ is continuous for every $x \in X$. \label{def:rlf-1}
	\item For every $f \in \test (X)$ and $\mm$-a.e. $x \in X$, $t \mapsto f( \XX _t(x))$ is in $W^{1,1}([0,T])$ and 
		\begin{equation}\label{eq:rlf-def}
		 	\frac{d}{dt}f( \XX _t(x))= df(V(t))( \XX _t(x)) \; \; \text{ for a.e. } t \in [0,T].  
		\end{equation} \label{def:rlf-2}
		\item There exists a constant $C(V)$ so that $( \XX _t)_{*}m \leq C\mm $ for all $t$ in $[0,T]$. \label{def:rlf-3}
	\end{enumerate}
\end{Def}

For sufficiently regular vector fields, RLFs satisfy an existence and uniqueness property \cite{AT14}.

\begin{Thm}\label{thm:RLF-existence}
Let $(X,d,\mm)$ be an $\rcd (K,N)$ space, and assume $V \in L^1([0,T],$ $L^2(TX))$ satisfies $V(t)\in D(\text{div})$ for a.e. $t \in [0,T]$ with
\begin{equation*}
	\text{div}(V(\cdot)) \in L^1 \left( [0,T],L^2(\mm ) \right) \; \;  \;(\text{div}(V(\cdot )))^- \in L^1\left( [0,T], L^{\infty}(\mm )\right)  \; \; \; \nabla V( \cdot )  \in L^1 \left( [0,T],L^2(T^{\otimes 2}X)\right) .
\end{equation*}
Then there exists a unique (up to $\mm$-a.e. equality) RLF $\XX : [0,T] \times X \to X$ for $V$ satisfying
\begin{equation}\label{eq:rlf-measure-0}
	( \XX _t)_{*}(\mm ) \leq \exp 
 \left( \int_{0}^{t} \|\text{div}(V(s))^-\|_{L^{\infty}(\mm )}\,ds \right) \mm
\end{equation}
for every $t \in [0,T]$. 
\end{Thm} 
\begin{Rem}
\rm The estimate \ref{eq:rlf-measure-0} can be localized for any $S \in \mathcal{B}(X)$ as
\begin{equation}\label{eq:volume-distortion-00}
\left( \XX_t  \right)_{\ast} \left( 
 \mm \vert _S \right) \leq \exp \left(  \int_0^t     \|\text{div}(V(s))^-\|_{L^{\infty}\left( (\XX_s)_{\ast} (\mm \vert _S) \right) }\,ds \right) \mm  .        
\end{equation}      
This is obtained from \cite[4-22]{AT14} by choosing $\beta (z) : = z^p $ for $p \to \infty$. 
\end{Rem}

\begin{Rem}\label{rem:sob-to-lip}
\rm From \ref{def:rlf-2}, we get that if $\Vert V (t) \Vert_{\infty} \leq L$ for all $t \in [0,T] $ and some $L > 0 $, then for $\mathfrak{m}$-a.e. $x \in X$, the map 
\begin{equation}\label{eq:sob-to-lip-flow}
    [0, T] \ni t \mapsto \XX_t(x) \text{  is }L\text{-Lipschitz}.
\end{equation} 
Thus, after modifying $\XX$ on a set of measure zero, we can always assume \ref{eq:sob-to-lip-flow} holds for all $x \in X$ (see \cite[Theorem A.4]{GT21}).   
\end{Rem}

For nice vector fields, there is a reverse flow \cite[Proposition 3.12]{D20}.

\begin{Pro}
\rm Let $(X,d, \mm ) $, $V$, $\XX$, be as in Theorem \ref{thm:RLF-existence}, and define $\overline{V} : [0,T] \to L^2(TX)$ as 
\[  \overline{V} (t) (x) : =  - V (T-t)(x) \]
 for each $t \in [0,T], x \in X. $ Then there is a map $\overline{\XX} : [0,T] \times X \to X$ which is a RLF for $\overline{V}$ and for $\mm$-a.e. $x \in X$  one has 
\[   \overline{\XX} _t ( \XX_T  (x) ) = \XX_{T-t}(x) \hspace{0.2cm} \text{ for all } \hspace{0.2cm} t \in [0,T].             \]
\end{Pro}
\begin{Rem}
    If $\Vert \text{div} (V (t)) \Vert _{\infty} \leq D $ for all $t \in [0,T]$ and some $D > 0$, \ref{eq:rlf-measure-0} implies 
\begin{equation}\label{eq:volume-distortion}
    e^{- D T} \mm \leq (\XX_t(\cdot))_{*}(\mm) \leq e^{ D  T} \mm \; \; \text{for all } t \in [0,T].
\end{equation}
\end{Rem}

The integral first variation formula extends to $\rcd (K,N)$ spaces \cite[Corollary 4.2]{BDS22}.

\begin{Thm}\label{thm:first-variation}
   Let $r > 0 $, $(X,d, \mm )$ an $\rcd (K,N) $ space, and $V$ a time-dependent vector field satisfying the conditions of Theorem \ref{thm:RLF-existence}. Set 
   \begin{equation}\label{eq:ddrt}
       \begin{split}
       \dr : [0, T] \times X \times X \to [0,r] \\
       \dr (t) : =  \sup_{s \in [0,t]} \dr ( \XX_s)
       \end{split}
   \end{equation}
    Let $S_1, S_2$ be Borel subsets of $X$ with finite positive measure, and define 
    \begin{equation}\label{eq:gamma(t)}
        \Gamma (t) : = \{ (a,b) \in S_1 \times S_2 \vert \dr (t) (a,b) < r \} .
    \end{equation} 
    Then the map $t \mapsto \int_{S_1 \times S_2} \dr (t)(x,y) d(\mm \times \mm)(x,y)$ is Lipschitz on $[0,T]$ and for a.e. $t \in [0,T]$ one has
    \begin{align*}
         &\frac{d}{dt} \int_{S_1 \times S_2} dt_{r}(t)(x,y) \, d(\mm \times \mm)(x,y)\\
        \leq &\int_{0}^{1} \int_{\Gamma (t)} d(\XX_{t}(x), \XX_{t}(y))|\nabla V (t)|(\gamma_{\XX_{t}(x), \XX_{t}(y)}(s)) \, d(\mm \times \mm)(x,y) \, ds.
    \end{align*}
\end{Thm}

\begin{Rem}
Although \cite[Corollary 4.2]{BDS22} was stated only for the non-collapsed case (i.e. $\mm = \mathscr{H}^N$), its proof follows that of \cite[Proposition 3.27]{D20} (see also \cite[Proposition 4.1]{BDS22} for additional comments) and in particular works without the non-collapsed assumption. 
\end{Rem}

\subsection{Group norms}

Let $(X,p)$ be a pointed proper geodesic space and $\Gamma \leq \iso(X)$ a closed group of isometries. The \textit{norm} $\Vert \cdot \Vert_p : \Gamma \to \mathbb{R}$ associated to $p$ is defined as $\Vert g \Vert_p : = d(gp,p)$.  We denote as $\mathcal{G}(\Gamma, X ,  p,r)$ the subgroup of $\Gamma$ generated by the elements of norm $\Vert \cdot \Vert_p \leq r$. The \textit{norm spectrum} $\sigma (\Gamma )$ is defined as the set of $r \geq 0 $ for which  $\mathcal{G}(\Gamma ,X, p,r) \neq \mathcal{G}(\Gamma ,X, p , r- \varepsilon )$ for all $\varepsilon > 0 $.  Notice that we always have $0 \in \sigma (\Gamma)$.  If we want redundancy we sometimes write $\sigma(\Gamma , X,p)$ to denote the spectrum of the action of $\Gamma$ on the pointed space $(X,p)$. 

\begin{Pro}\label{pro:bounded-generation}
\rm If $\Gamma$ is equipped with the metric $d_0^p$ from \ref{d0}, and $\Gamma = \mathcal{G}(\Gamma , X, p, D)$ for some $D>0$, then $\Gamma = \langle B_{D + 2 \sqrt{2} + \varepsilon }( \Id_X )   \rangle $ for all $\varepsilon > 0$.
\end{Pro}
\begin{proof}
From  \ref{d0} with $r = 1/\sqrt{2}$, for all $g \in \Gamma$ one gets
\begin{equation}\label{norm-dist}
  \Vert g \Vert _p \leq d_0^p(g , I_X) \leq \Vert g \Vert _p + 2 \sqrt{2}   . 
\end{equation}
Then $ \{ g \in \Gamma \vert \Vert g \Vert _p \leq D  \} \subset B_{D + 2 \sqrt{2} + \varepsilon }( \Id_X ) \text{ for all }\varepsilon  >  0.  $
\end{proof}
This spectrum is closely related to the covering spectrum introduced by Sormani--Wei in \cite{SW04}, and it also satisfies a continuity property \cite[Proposition 47]{SZ23}.

\begin{Pro}\label{pro:spec-cont}
\rm Let $(X_i,p_i)$ be a sequence of pointed proper metric spaces that converges in the Gromov--Hausdorff sense to $(X,p)$ and consider a sequence of closed isometry groups $\Gamma _i \leq \iso (X_i)$ that converges equivariantly to a closed group $\Gamma \leq \iso (X)$. Then for any convergent sequence $r_i \in \sigma (\Gamma _ i ) $, we have $ \left[ \lim_{i \to \infty} r_i \right] \in \sigma (\Gamma ) $. 
\end{Pro}

\begin{Rem}
\rm It is possible that an element in $\sigma (\Gamma)$ is not a limit of elements in $\sigma (\Gamma_i)$, so this spectrum is not necessarily continuous with respect to equivariant convergence (see \cite[Example 1]{KW11}).
\end{Rem}

\begin{Pro}
\rm For any $a > 0$, one has $\mathcal{G}(\Gamma , X, p, a) = \mathcal{G}(\Gamma , X, p, a + \varepsilon )$  for $\varepsilon > 0 $ small enough.
\end{Pro}

\begin{proof}
Assuming the proposition fails, there is a sequence of elements $g_i $ not in $\mathcal{G}(\Gamma , X,p, a)$ with $\Vert g_i \Vert_p \to a$. As the sequence $\Vert g_i \Vert _p$ is bounded, after taking a subsequence we can assume $g_i \to g $ for some $g \in \Gamma$ with $\Vert g \Vert _p = a$. Then for large enough $i$, $\Vert g^{-1}g_i \Vert _p < a$, so $g_ i = (g) (g^{-1}g_i) \in \mathcal{G}(\Gamma , X , p, a) $, which is a contradiction.  
\end{proof}

\begin{Cor}\label{cor:empty-spec-1}
\rm For any $[a,b] \subset (0, \infty)$, the following are equivalent:
\begin{itemize}
\item $\sigma (\Gamma)\cap (a, b] = \emptyset$.
\item $\mathcal{G} (\Gamma , X , p, a) = \mathcal{G}(\Gamma , X , p, b)$.
\end{itemize}
\end{Cor}
It is well known that when a group action is co-compact, the spectrum is bounded \cite[Proposition 5.28]{G81}.

\begin{Lem}\label{lem:co-compact-spectrum}
\rm Let $(X, p)$ be a pointed proper geodesic space and $\Gamma \leq \iso (X)$ a closed group of isometries. Then $ r \leq 2  \cdot$  diam$(X/ \Gamma)  $ for all $r \in \sigma(\Gamma , X , p ).$
\end{Lem} 

To prove Theorem \ref{thm:induction-main-theorem}, one needs to control the number of generators of the groups $\Gamma_i$ \cite[Theorem 80]{SZ23}.

\begin{Lem}\label{lem:very-finite-generated}
\rm  Let $(X,d, \mm , p ) $ be a pointed $ \rcd (K,N)$ space, and $\Gamma \leq Iso (X)$ a discrete group of measure preserving isometries with $\Gamma  = \mathcal{G}(\Gamma , X , p, D)$. Then $\Gamma$ can be generated by at most $C(K,N, D)$ elements. 
\end{Lem}

\subsection{Group Theory}

In this section we cover basic group theory results needed later. Proofs of Propositions \ref{pro:pass-to-subgroup} and \ref{pro:finite-extension-abelian} below can be found in \cite[Section 4]{FY92}.

\begin{Pro}\label{pro:pass-to-subgroup}
\rm Let $G$ be a group and $H \leq G$ a subgroup of index $[G:H] \leq M$.
\begin{enumerate}
\item There is a normal subgroup $H^{\prime} \triangleleft G$ with $H^{\prime} \leq H$ and $[G:H^{\prime}] \leq C(M)$.\label{normal-always}
\item If $G$ is generated by $k$ elements, there is a characteristic subgroup $H^{\prime \prime} \triangleleft G$ with $H^{\prime \prime } \leq H$ and $[G : H^{\prime \prime } ]\leq C(M,k)$.\label{characteristic-always}
\end{enumerate}
\end{Pro}

\begin{Rem}\label{rem:main-characteristic}
\rm By Lemma \ref{lem:very-finite-generated} and Proposition \ref{pro:pass-to-subgroup}.\ref{characteristic-always}, whenever Theorem \ref{thm:induction-main-theorem} holds, we may assume the subgroups $G_i  \triangleleft \Gamma _ i $  are characteristic.
\end{Rem}

\begin{Pro}\label{pro:finite-extension-abelian}
\rm Let $A$ be an abelian group generated by $m$ elements, and $\varphi :  G \to A $ a surjective morphism with finite kernel. Then $G$ contains a finite index abelian subgroup generated by $m$ elements. 
\end{Pro}

\begin{Pro}\label{pro:powers}
\rm Let $G$ be a group,  $H \triangleleft G$ a normal subgroup, $a, b \in G$ such that $[a,b ] \in H$, and $H_0 \triangleleft H$ a characteristic subgroup of $H$ with $[H : H_0] \leq M$. Then for all $C \geq 2M$ one has $[a^{C !}, b ] \in H_0$.
\end{Pro}

\begin{proof} In the group $G / H_0$, set $\alpha  : = a H_0$ and $\beta : = bH_0$. Then $\alpha \beta \alpha ^{-1} = \beta h$ for some $h \in H/ H_0$. A direct computation shows that $\alpha^k \beta \alpha^{-k} = \beta (h)(\alpha h \alpha ^{-1}) \cdots ( \alpha^{k-1} h \alpha ^{-k+1})$. As $H/H_0$ is normal in $G/H_0$ and $\vert H / H_0 \vert \leq M$, one gets that $\alpha ^{M!} h \alpha^{-M!} = h$, so
\begin{eqnarray*}
\alpha ^{M! M} \beta \alpha^{-M! M} & = & \beta \prod_{j=0}^{M!M-1} (\alpha ^j h \alpha^{-j}) \\
& = & \beta \left( \prod_{j=0}^{M!-1} (\alpha ^j h \alpha ^{-j} )    \right)^{M}\\
& = & \beta.
\end{eqnarray*} 
If $C \geq 2M$, then $C!$ is a multiple of $ M!M $, and 
\begin{eqnarray*}
[\alpha ^{C!}, \beta ] & = &  \alpha^{C!}\beta \alpha ^{-C!}\beta^{-1}\\
& = &  \alpha^{M!M} (  \ldots ( \alpha^{M!M}  \beta    \alpha^{-M!M} )  \ldots  \alpha^{-M!M})  \beta^{-1}\\
& = & \beta \beta^{-1}\\
& = & e_{G/H_0}
\end{eqnarray*} 
This shows that $[a^{C!}, b ] \in H_0$.
\end{proof}

\begin{Pro}\label{pro:replace}
\rm Let $\Gamma$ be a group, $G \triangleleft \Gamma$ a characteristic subgroup admitting a nilpotent basis, $g \in \Gamma$,  $\varphi  \in \aut (\Gamma )$, and $C\in 2 \mathbb{Z} $. If $ [\Gamma : G ] \leq C /2 $, then the nilpotent basis in $G$ is preserved by $\varphi  ^{C!}$ if and only if it is preserved by $(\varphi  \circ g_{\ast} ) ^{C!}$. 
\end{Pro}

\begin{proof}
First we observe that for any $k \in \mathbb{N}$ we have
\begin{equation}\label{eq:replacement}
\begin{split}
(\varphi  \circ g_{\ast} ) ^{k} & =  (\varphi  \circ g_{\ast} \circ  \varphi ^{-1})(\varphi ^2 \circ  g_{\ast} \circ  \varphi ^{-2}) \cdots (\varphi  ^k \circ  g_{\ast}  \circ \varphi  ^{-k} ) \varphi ^k \\ 
& =  (\varphi  (g))_{\ast} ( \varphi ^2 (g))_{\ast} \cdots (\varphi  ^k (g))_{\ast} \varphi  ^k \\
& =  (  \varphi  (g) \varphi  ^2 (g) \cdots \varphi  ^k (g)  )_{\ast} \varphi ^k.
\end{split}
\end{equation}
On the other hand, as $G$ is characteristic in $\Gamma$,  the group $G_{\ast} : = \{ x_{\ast} : \Gamma \to \Gamma \vert x \in G \}$ is normal in $\aut(\Gamma )$, so one has $y_{\ast}^C G_{\ast} = G_{\ast}$ for all $y \in \Gamma$. Also, notice that  $\varphi ^{(C/2)!}(g ) G = g G$ in $\Gamma / G$,  so  $(\varphi ^{(C/2)!}(g ))_{\ast} G_{\ast} = g_{\ast} G_{\ast}$ in $\aut (\Gamma )/ G_{\ast}$. Thus if $\ell = (C-1)!/(C/2)!$, using \ref{eq:replacement} we have
\begin{equation*}
\begin{split}
(\varphi  \circ g_{\ast} ) ^{C!} G_{\ast} & =  \left( (  \varphi  (g) \varphi  ^2 (g) \cdots \varphi  ^{(C/2)!} (g)  )_{\ast} \right)^{C \ell } \varphi ^{C!} G_{\ast} \\
& =   \varphi ^{C!} G_{\ast}.     
\end{split}
\end{equation*}
This implies that $(\varphi  \circ g_{\ast} ) ^{C!}$ and $\varphi ^{C!}$ differ only by an element of $G_{\ast}$, which clearly respects the nilpotent basis in $G$. 
\end{proof}

We will also need the following version of Bieberbach Theorem \cite[Section 4]{FY92}.

\begin{Thm}\label{thm:fyb}
\rm (Fukaya--Yamaguchi) Let $G \leq \iso(\mathbb{R}^m)$ be a closed group of isometries and $G _0 \leq G$ its identity connected component. Then $G / G_0$ contains a finite index abelian subgroup generated by at most $m$ elements.
\end{Thm}

\begin{Cor}\label{cor:fybc}
\rm Let $Z$ be a compact metric space, $\Gamma \leq \iso ( \mathbb{R}^m \times Z )$ a closed group of isometries and $\Gamma_0 \leq \Gamma$ its identity connected component. If $\iso (Z)$ is a Lie group, then $\Gamma / \Gamma_0$ contains a finite index abelian subgroup generated by at most $m$ elements.
\end{Cor}

\begin{proof}

 Notice that for each $(x,z) \in \mathbb{R}^m \times Z$, the $\mathbb{R}^m$-fiber passing through $(x,z)$ can be characterized as the union of the images of all infinite geodesics passing through $(x,z)$. This implies that $\iso(\mathbb{R}^m \times Z)$ respects the splitting $\mathbb{R}^m \times Z$ and decomposes as $\iso(\mathbb{R}^m \times  Z) = \iso(\mathbb{R}^m) \times \iso(Z)$. Let $G \leq \iso (\mathbb{R}^m )$ be the image of $\Gamma $ under the projection  $ \pi : \iso(\mathbb{R}^m \times Z) \to \iso(\mathbb{R}^m)$. As $\Gamma$ is closed and $\iso(Z) = \Ker (\pi) $ is compact,  $G$ is closed in $\iso (\mathbb{R}^m)$.  
 
 We claim that $ \pi ( \Gamma _ 0 ) = G_0$. Assuming the contrary, as $G_0$ is connected, there would be a sequence $x_i \in G_0 \backslash \pi (\Gamma _0)$ with $x_i \to e_{G_0}$. Pick $g_i \in \Gamma$ with $\pi (g_i) = x_i$.  Since $\iso (Z)$ is compact, after passing to a subsequence we can assume $g_i \to g_{\infty}$ for some $g_{\infty} \in \Ker ( \pi )$. Then  $g_{\infty}^{-1} g_i \to e_{\Gamma}$, so for $i$ large enough one has $g_{\infty}^{-1} g_i \in \Gamma_0$. This would mean that $\pi ( g_{\infty}^{-1} g_i ) = \pi (g_i)  = x_i \in \pi (\Gamma_0)$, which is a contradiction. 

Let $H : = \pi^{-1}(G_0) \cap \Gamma $. We claim that $[H : \Gamma_0 ] < \infty $. Otherwise, there would be a sequence $h_i \in H \cap \Ker(\pi)$ with $h_i^{-1}h_j \in H \backslash \Gamma_0$ for all $i \neq j$. As $\Ker(\pi )$ is compact, after taking a subsequence we can assume $h_i \to h_{\infty}$  for some $h_{\infty} \in \Ker(\pi )$. This would mean that for $i,j$ large enough, one has $h_i^{-1}h_j \in \Gamma_0$, which is a contradiction. 
 
 The above implies that $\Gamma / \Gamma_0$ is a finite extension of $ (\Gamma / \Gamma_0) / (H / \Gamma_0 ) = \Gamma / H \cong G / G_0$, so the result follows from  Theorem \ref{thm:fyb} and Proposition \ref{pro:finite-extension-abelian}.
\end{proof}

\section{Groups of connected components}\label{sec:gcc}

The goal of this section is to prove the following result (cf. \cite[Theorem 3.10]{FY92} and \cite[Lemma 58]{SZ23}). The groups $\Upsilon_i$ play the role of ``connected component of the identity'' in the groups $\Gamma_i$.

\begin{Thm}\label{thm:upsilon}
\rm  Let $(X_i, p_i)$ be a sequence of proper geodesic spaces that converges in the Gromov--Hausdorff sense to a space $(X,p)$, and $\Gamma_i \leq \iso (X_i)$ a sequence of closed groups of isometries that converges equivariantly to a closed group $\Gamma \leq \iso (X)$. Assume
\begin{itemize}
\item $\Gamma _ i = \mathcal{G}(\Gamma _i , X_i , p_i, D)$ for some $D > \sqrt{2}$.
\item $\Gamma_0$, the connected component of the identity of $\Gamma$, is open.
\item $\Gamma / \Gamma_0 $ is finitely presented.  
\end{itemize}
Then there are subgroups $\Upsilon_i \triangleleft \Gamma_i$ such that 
\begin{itemize}
\item $\Upsilon_i$ is normal in $\Gamma_i$ for $i$ large enough.
\item For any $R > 0$, $\Upsilon_i  = \langle \psi _i ^{-1}(B_R(\Id_X ) \cap \Gamma _0 )  \rangle $ for $i$ large enough.
\item For $i$ large enough, there are surjective morphisms $\Gamma / \Gamma_0 \to \Gamma_i / \Upsilon _ i$. 
\end{itemize} 
Where $\psi_i : \Gamma_i \to \Gamma \cup \{ \ast \} $ are the Gromov Hausdorff approximations given by Definition \ref{def:equivariant}.
\end{Thm}

\begin{proof}
Let $r > 0 $ be such that $B_{2r}(\Id_X) \subset \Gamma_0$. First we show that for any fixed $R \geq r $ and  $\delta \in (0, r]$, the subgroup of $\Gamma_i$ generated by $\psi_i^{-1}(B_{\delta}(\Id_X  ))$ in $\Gamma_i$ coincides with $   \langle \psi _i ^{-1}(B_R(\Id_X ) \cap \Gamma _0 )  \rangle  $ for large enough $i$. To see this, first take a collection $y_1, \ldots , y_{n} \in B_R(\Id_X ) \cap \Gamma_0$ with 
\[ B_R(\Id_X ) \cap \Gamma_0 \subset \bigcup_{j=1}^n B_{\delta /10}(y_j) .  \]
By connectedness,  for each $j  \in \{1 , \ldots , n \} $ we can construct a sequence $e= z_{j,0}, \ldots , z_{j,k_j}=y_j $ in $\Gamma$ with  $d(z_{j,\ell -1},z_{j , \ell}) \leq \delta /10$ for each $\ell \in \{ 1, \ldots ,k_j \} $. Since all $  z_{j, \ell}  $ are contained in a compact subset of $\Gamma_0$,  if $i$ is large enough, for any element $x \in \psi_i^{-1}(B_R(\Id_X)\cap \Gamma_0)$ we can find $y_j $ with $d(y_j,\psi_i(x)) \leq \delta /10$, and $ e = x_0 , \ldots , x_{k_j} = x $  in $ \Gamma _i $ with $d(z_{j,\ell} , \psi_i(x_{\ell})) \leq \delta / 10$ for each $\ell$. This allows us to write $x = (x_1)(x_1^{-1}x_2 )\cdots (x_{k_j -1}^{-1}x_{k_j})$ as a product of $k_j $ elements in $\psi_i^{-1}(B_{\delta}(\Id_X ))$, proving our claim. Set $\Upsilon_i$ to be the subgroup of $\Gamma_i$ generated by $\psi_i^{-1}(B_r(\Id _X))$.

Choose $\delta > 0 $ small enough so that for all $g \in B_{3D}(\Id_X ) $, $h \in B_{\delta}( \Id_X)$ one has $ghg^{-1} \in B_{r/2}(\Id_X)$.  Then for large enough $i$, the conjugate of an element in  $\psi_i^{-1}(B_{\delta}(\Id_X ))$ by an element in $\psi_i^{-1}(B_{3D}(\Id_X ))$ lies in $\psi_i^{-1}(B_r(\Id_X ))$. By Proposition \ref{pro:bounded-generation},  $\psi_i^{-1}(B_{3D}(\Id_X))$ generates $\Gamma_i$ and $\psi_i^{-1}(B_{\delta}(\Id_X ))$ generates $\Upsilon_i$ for large enough $i$, implying that $\Upsilon_ i$ is normal in $\Gamma_i$.

Let $S_0 = \{ \overline{s}_1, \ldots , \overline{s}_k  \} \subset \Gamma / \Gamma_0$ be a finite symmetric generating set containing all connected components intersecting $B_{4D}(\Id_X ) $, $ S = \{s_1, \ldots , s_k \} \subset \Gamma $ a set of representatives, and for each $j \in \{ 1, \ldots ,  k\} $, pick a sequence $g_i^j \in \Gamma_i$ with $\psi_i (g_i^j) \to s_j$. Then define $h_i^{\prime} : S_0 \to \Gamma_i / \Upsilon_i$ as  $h_i^{\prime} ( \overline{s}_j)  : = g_i^j\Upsilon_i \in \Gamma_i / \Upsilon_i $. It is easy to check that $h_i^{\prime}(s_j)$ does not depend on the choices of the representatives $s_j$ nor the sequences $g_i^j$ for $i$ large enough. 

By hypothesis, $\Gamma / \Gamma_0$ admits a presentation $\langle S_0, W \rangle$ with $W$ a finite set of words. For $\overline{s}_{i_1} \ldots \overline{s}_{i_{\ell}} \in W$, one has $d_0^p(\psi_i(g_i^{i_1})\cdots \psi_i(g_i^{i_{\ell}}), \psi_i( g_i^{i_1} \cdots g_i^{i_{\ell}} ) ) < r $ for $i$ large enough, and hence $  \psi_i( g_i^{i_1} \cdots g_i^{i_{\ell}} )  \in \Gamma_0 $. This means, again for $i$ large enough, that $g_i^{i_1} \cdots g_i^{i_{\ell}} \in \Upsilon_i $, and $h_i^{\prime}(s_{i_1})\cdots h_i^{\prime}(s_{i_{\ell}})  =  g_i^{i_1} \cdots g_i^{i_{\ell}} \Upsilon_i = \Upsilon_i \in \Gamma _i / \Upsilon_i  $.  As there are only finitely many words in $W$, the functions $h_i^{\prime} :S_0 \to \Gamma_i / \Upsilon_i$ extend to group morphisms $h_i : \Gamma / \Gamma_0 \to \Gamma_i / \Upsilon_i$. As $S_0$ intersects each connected component in $B_{4D}(\Id_X)$ and $\Gamma_i$ is generated by $B_{3D}(\Id_{X_i})$, the maps $h_i$ are surjective.
\end{proof}
The following result deals with the base of induction in the proof of Theorem \ref{thm:induction-main-theorem}.
\begin{Lem}\label{lem:no-collapse-small-subgroup}
\rm Let $(X_i, d_i, \mm _i, p_i)$ be a sequence of $\rcd (K,N)$ spaces of rectifiable dimension $n$, and  $\Gamma_i \leq \iso (X_i)$ a sequence of closed groups of isometries. Assume the sequence $(X_i, p_i)$ converges in the Gromov--Hausdorff sense to a pointed metric space $(X, p)$, $X$ admits a measure that makes it an $\rcd (K,N)$ space of rectifiable dimension $n$, there is $D> 0 $ such that $\Gamma_i = \mathcal{G}(\Gamma_i ,X_i , p_i , D)$ for all $i$, and $\Gamma_i$ converges equivariantly to the trivial group. Then the groups $\Gamma_i$ are trivial for $i$ large enough.
\end{Lem}
\begin{proof}
Clearly, we can assume $D > \sqrt{2}$.  Let $\Upsilon_i \leq \Gamma_i$ be the subgroups given by Theorem \ref{thm:upsilon}. Then  $\Gamma _ i = \Upsilon _i = \langle \psi_i^{-1} ( \Id_X) \rangle  $ for $i$ large enough. From the definition of equivariant convergence, it is easy to see that the $\psi_i$-preimage of an open compact subgroup of $\Gamma$ is a subgroup in $\Gamma_i$ for $i$ large enough, hence $\Gamma_i = \psi_i^{-1} (\Id_X)$. This means that $\Gamma_i$ are small subgroups in the sense of \cite[Definition 66 and Remark 75]{SZ23}, so by  \cite[Theorem 93]{SZ23} the result follows.  
\end{proof}

\section{Proof of main regularity estimates on RLFs}\label{sec:pmre}

\begin{proof}[Proof of Theorem \ref{thm:essential-stability}:]
Let $r_x  \leq \rho / 100 $ be such that for all $r \leq r_x$ one has
\[    \mathfrak{m} (  \{  y \in B_r(x) \vert H(y) \leq  \delta \} ) \geq \frac{1}{2} \mathfrak{m} ( B_r(x)  ).          \]
Fix $r \leq r_x$. We claim that if $\delta$ is small enough,  there is $x_r \in B_{r}(x) \cap \{ H \leq  \delta \} $ and a constant $C_0 (N) > 1 $ which is independent of $r$ such that
\begin{enumerate}[label= $\bm{S _r}$\textbf{.\arabic*}]
    \item   There is $B_{r}'(x_r) \subseteq B_{r}(x_r)$ such that $\mm(B_{r}'(x_r)) \geq (1 - \sqrt{\delta }) \mm(B_r(x_r))$ and 
    \[   \XX_t(B_{r}'(x_r)) \subseteq B_{2r}(\XX_t(x_r)) \text{ for all } t \in [0, T].  \]\label{eq:sr-1}
    \item  For all $t \in [0,T]$, 
    \[ \frac{1}{C_0}\mm(B_{r}(x_r)) \leq \mm(B_r(\XX_t(x_r)) \leq C_0\mm(B_{r}(x_r)).\]\label{eq:sr-2}
\end{enumerate}
The argument is based on an induction on time, following the scheme of \cite[Section 5]{D20}. Let $I_0 = [0,\frac{r}{10L}]$ and fix some $x_{r,0} \in B_{r}(x)$ with $H (x_{r,0}) \leq \delta$. By Remark \ref{rem:sob-to-lip}, for any $y\in B_{r}(x)$ and any $t \in I_0$, we have
\begin{align}\label{distance bound 1}
\begin{split}
    d(\XX_{t}(y), \XX_{t}(x_{r,0})) &\leq d(\XX_{t}(y), y) + d(y, x) + d(x, x_{r,0}) + d(x_{r,0}, \XX_{t}(x_{r,0}))\\
    &\leq 2L\frac{r}{10L}+2r < 3r.
\end{split}
\end{align}
Define $dt_{r} (t)$ as in \ref{eq:ddrt}, and denote $S_1 = B_{r}(x) \cap \{H \leq  \delta \}$, $S_2 = B_{2r}(x) $, and $\Gamma (t)$ as in \ref{eq:gamma(t)}. By Theorem \ref{thm:first-variation}, we obtain
\begin{align}\label{eq:ide-1}
\begin{split}
    & \int_{S_1 \times S_2} dt_{r}\big( \frac{r}{10L} \big) (x,y)\, d(\mm \times \mm)(x,y)\\
    = & \,\int_{I_0} \frac{d}{dt} \int_{S_1 \times S_2} dt_{r}(t)(x,y) \, d(\mm \times \mm)(x,y)\,dt\\
    \leq & \,\int_{I_0} \int_{0}^{1} \int_{\Gamma (t)} d(\XX_{t}(x), \XX_{t}(y))|\nabla b (t) |(\gamma_{\XX_{t}(x), \XX_{t}(y)}(s)) \, d(\mm \times \mm)(x,y) \, ds \, dt.
\end{split}
\end{align}
Using \eqref{eq:volume-distortion} and a change of variables, for any $t \in I_0$ we have
\begin{align*}
    & \int_{0}^{1} \int_{\Gamma (t)} d(\XX_{t}(x), \XX_{t}(y))|\nabla b(t)|(\gamma_{\XX_{t}(x), \XX_{t}(y)}(s)) \, d(\mm \times \mm)(x,y) \, ds \\
    \leq & \, e^{D T} \int_{0}^{1}\int_{\XX_t(\Gamma (t))} d(x, y)|\nabla b (t) |(\gamma_{x,y}(s)) d(\mm \times \mm)(x,y)\, ds.
\end{align*}
Furthermore,
\begin{align*}
    & e^{D T} \int_{0}^{1}\int_{ \XX_t(\Gamma (t))} d(x, y)|\nabla b (t) |(\gamma_{x,y}(s)) d(\mm \times \mm)(x,y)\, ds\\
    \leq & \, e^{D T}\int_{0}^{1} \int_{B_{6r}(\XX_{t}(x_{r,0}))^{ \times 2}} d(x, y)|\nabla b (t) |(\gamma_{x,y}(s)) d(\mm \times \mm)(x,y)\, ds\\
    \leq & \, e^{D T}C(N)r \mm(B_{6r}(\XX_{t}(x_{r,0})))^2 \fint_{B_{12r}(\XX_{t}(x_{r,0}))}|\nabla b (t) |\, d\mm\\
    \leq & \, e^{ D T}C(N)r\mm(B_r(x))^2 \fint_{B_{12r}(\XX_{t}(x_{r,0}))}|\nabla b (t) |\, d\mm, 
\end{align*}
where we used \eqref{distance bound 1} for the second line, Theorem \ref{thm:segment} for the third line, and Theorem \ref{thm:bg-in} for the fourth line. 
Combining the above estimates starting from \eqref{eq:ide-1}, we obtain
\begin{align*}
    & \int_{S_1 \times S_2} dt_{r} \big( \frac{r}{10L}\big) (x,y)\, d(\mm \times \mm)(x,y)\\
    \leq & \, e^{ D T}C(N)r\mm(B_r(x))^2 \int_{I_0} \fint_{B_{12r}(\XX_{t}(x_{r,0}))}|\nabla b (t) |\, d\mm \, dt\\
    \leq & \, e^{ D T}C(N)r\mm(B_r(x))^2 \int_{I_0} \mx_{\rho} (|\nabla b (t)|)(\XX_{t}(x_{r,0})) \, dt\\
    \leq & \, e^{ D T}C(N)r\mm(B_{r}(x))^2 H(x_{r,0}) \\
    \leq & \, e^{ D T}C(N)r\mm(B_{r}(x))^2 \delta .
\end{align*}
By Chebyshev's inequality, there is $x_{r,1} \in B_{r}(x) \cap \{H < K\}$ with 
\begin{equation*}
    \int_{B_{2r}(x)} dt_{r}\big( \frac{r}{10L}\big) (x_{r,1},y) d\mm(y) \leq e^{ D T}C(N)r\mm(B_{r}(x)) \delta .\\
\end{equation*}
Another instance of Chebyshev's inequality, along with 
\begin{equation*}
    \mm(B_{r}(x_{r,1})) \geq \frac{1}{C(N)} \mm(B_{r}(x)) ,
\end{equation*}
 shows that here is $B_{r,1}(x_{r,1}) \subseteq B_{r}(x_{r,1})$ with $ \mm(B_{r,1}(x_{r,1})) \geq (1 - \sqrt{ \delta }) \mm(B_{r}(x_{r,1}))$ and
    \begin{equation*}
        dt_{r}\big( \frac{r}{10L} \big) (x_{r,1}, z) \leq e^{D T}C(N)r \sqrt{\delta} \text{ for all } z \in B_{r,1}(x_{r,1}).
    \end{equation*}
 Hence if  $e^{ D T}C(N)\sqrt{ \delta} < 1$, then by the definition of $dt_{r} (t) $, for all $t \in I_0$ and $z \in B_{r,1}(x_{r,1})$ we have
\begin{align*}
    d(\XX_{t}(x_{r,1}), \XX_{t}(z)) &< d(x,z)+e^{D T}C(N)r\sqrt{ \delta } \\
    &< 2r,
\end{align*}
 so $\XX_{t}( B_{r,1}(x_{r,1}) ) \subset B_{2r}(\XX_{t}(x_{r,1}))$ for all $t \in I_0$. As 
\begin{align}B_{\frac{1}{2}r}(x_{r,1})\subset B_r(\XX_{t}(x_{r,1}))\subset B_{\frac{3}{2}r}(x_{r,1})  \text{ for all }  t \in I_0 , \end{align} 
from Theorem \ref{thm:bg-in} we have for all $t \in I_0$,
\begin{align}
   \frac{1}{C}\mm(B_r(x_{r,1}))\leq \mm(B_r(\XX_{t}(x_{r,1})))\leq C \mm(B_r(x_{r,1})) .
\end{align}
The argument above establishes \ref{eq:sr-1} and \ref{eq:sr-2} up to time $\frac{r}{10L}$. Now we show we can establish the same estimate up to time $T$ provided $\delta$ is small enough. 

 Let $k \in \mathbb{N}$ with $k < \lceil  10T L / r \rceil$,  and assume there is $x_{r,k}\in B_r(x)$ such that 
\begin{enumerate}[label=$\bm{S_{r,k}}$\textbf{.\arabic*}]
    \item There exists $B_{r}'(x_{r,k}) \subseteq B_{r}(x_{r,k})$ with $\mm(B_{r}'(x_{r,k}))) \geq (1 - \sqrt{\delta }) \mm(B_{r}(x_{r,k})))$ and  
    \begin{equation*}
            \XX ( B_{r}^{\prime} (x_{r,k})) \subset B_{2r}(x_{r,k}) \text{ for all } t \in \big[ 0,\frac{kr}{10L} \big].
    \end{equation*} \label{eq:e-s-ind-hyp-1}
    \item  For all $t \in [ 0, \frac{kr}{10L}  ]$, 
    \[     
   \frac{1}{C_0}\mm(B_r(x_{r,k}))\leq \mm(B_r(\XX_{t}(x_{r,k})))\leq C_0 \mm(B_r(x_{r,k}))   .               \]\label{eq:e-s-ind-hyp-2}
\end{enumerate}  
Set $t_k : = \min \big\{ \frac{(k+1)r}{10L}  ,  T \big\}$ and  $I_k : = [ 0, t_k ]$. 
From \ref{eq:e-s-ind-hyp-1} and Remark \ref{rem:sob-to-lip},  for all $t \in I_k$ we have
\begin{equation}\label{eq:sausage-r-induction}
\XX_{t}(B_{r}'(x_{r,k})) \subset B_{2r + \frac{1}{2}r}(\XX_{t}(x_{r,k})).
\end{equation}
Let $S_1 : = B_{r}^{\prime}(x_{r,k})$, $S_2 : = B_{2r}(x) $, and $\Gamma (t)$ be given by \ref{eq:gamma(t)}. By Theorem \ref{thm:first-variation},
\begin{align}\label{eq:ide-2}
\begin{split}
    & \int_{S_1 \times S_2} dt_{r} ( t_k  ) (x,y)\, d(\mm \times \mm)(x,y)\\
    = & \,\int_{I_k} \frac{d}{dt} \int_{S_1 \times S_2} dt_{r}(t)(x,y) \, d(\mm \times \mm)(x,y)  \, dt \\
    \leq & \,\int_{I_k} \int_{0}^{1} \int_{\Gamma (t)} d(\XX_{t}(x), \XX_{t}(y))|\nabla b (t) |(\gamma_{\XX_{t}(x), \XX_{t}(y)}(s)) \, d(\mm \times \mm)(x,y) \, ds \, dt\\ 
    \leq &\,\int_{I_k}  e^{D T} \int_{0}^{1}\int_{ \XX_t(\Gamma (t))} d(x, y)|\nabla b (t) |(\gamma_{x,y}(s)) d(\mm \times \mm)(x,y)\, ds \, dt \\ 
    \leq &\,\int_{I_k}   e^{D T}\int_{0}^{1} \int_{B_{6r}(\XX_{t}(x_{r,k})) ^{\times 2 }} d(x, y)|\nabla b (t) |(\gamma_{x,y}(s)) d(\mm \times \mm)(x,y)\, ds  \, dt \\ 
    \leq &\,\int_{I_k}   e^{D T}C(N)r \mm(B_{12r}(\XX_{t}(x_{r,k})))^2 \fint_{B_{6r}(\XX_{t}(x_{r,k}))}|\nabla b (t) |\, d\mm  \, dt \\ 
    \leq &\,\int_{I_k}   e^{D T} C_0^2C(N)r\mm(B_r(x))^2 \fint_{B_{12r}(\XX_{t}(x_{r,k}))}|\nabla b (t) |\, d\mm  \, dt .
\end{split}
\end{align}
where we used \ref{eq:volume-distortion}, \eqref{eq:sausage-r-induction}, Theorem \ref{thm:segment}, and Theorem \ref{thm:bg-in} with \eqref{eq:sausage-r-induction}.
From the above estimates we obtain
\begin{align*}
    & \int_{S_1 \times S_2} dt_{r} (t_k) (x,y)\, d(\mm \times \mm)(x,y)\\
    \leq & \, e^{D T}C_0^2C(N)r\mm(B_r(x))^2 \int_{I_k} \fint_{B_{12r}(\XX_{t}(x_{r,k}))}|\nabla b (t)|\, d\mm \, dt\\
    \leq & \, e^{D T}C_0^2C(N)r\mm(B_r(x))^2 \int_{I_k} \mx _{\rho}(|\nabla b (t)|)(\XX_{t}(x_{r,k})) \, dt\\
    \leq & \, e^{D T}C_0^2C(N)r\mm(B_{r}(x))^2 H(x_{r,k}) \\
    \leq & \, e^{D T}C_0^2C(N)r\mm(B_{r}(x))^2 \delta .
\end{align*}
By Chebyshev's inequality, there is some $x' \in B_{r}'(x_{r,k})$ with
\begin{equation*}
    \int_{B_{2r}(x)} dt_{r} (t_k) (x',y) d\mm(y) \leq e^{ D T}C_0^2C(N)r\mm(B_{r}(x)) \delta.\\
\end{equation*}
Thus there are $C(D, T, N) > 1 $ and  $B_k\subset B_{2r}(x)$ with 
\begin{align}
dt_{r} (t_k) (x',y) &\leq C r \sqrt{\delta }  \text{ for all }  y\in B_k, \label{eq:b-k-1} \\
\mm(B_k) &\geq (1-\sqrt{\delta }/2)\mm(B_{2r}(x)).\label{eq:b-k-2}
\end{align}
As $x'\in B_r'(x_{r,k})$, from \ref{eq:sausage-r-induction} we have $\XX_{t}(x')\in B_{2r + \frac{1}{2}r}(\XX_{t}(x_{r,k}))$ for all  $t\in I_k$, provided $C \sqrt{\delta } < 1$. Thus for all $t \in I_k$, we also have 
\begin{equation}\label{eq:sausage-r-induction-2}
\XX_{t}(B_k) \subset  B_{7r}(\XX_{t}(x_{r,k})).    
\end{equation}
Define $S_1=B_k$, $S_2=B_{r}(x)\cap\{H<\delta \}$, and $\Gamma (t)$ as in \ref{eq:gamma(t)}. Similar as before, we have
\begin{align*}
    & \int_{S_1 \times S_2} dt_{r} (t_k) (x,y)\, d(\mm \times \mm)(x,y)\\
    \leq & \, e^{D T}C_0^2C(N)r\mm(B_r(x))^2 \int_{I_k} \fint_{B_{20r}(\XX_{t}(x_{r,k}))}|\nabla b (t)|\, d\mm \, dt\\
    \leq & \, e^{D T}C_0^2C(N)r\mm(B_r(x))^2 \int_{I_k} \mx _{\rho}(|\nabla b|)(\XX_{t}(x_{r,k})) \, dt\\
    \leq & e^{D T}C_0^2C(N)r\mm(B_{r}(x))^2 \delta.
\end{align*}
This there is $x_{r,k+1}\in B_{r}(x)\cap\{H<\delta \}$ such that
\begin{align*}
     \int_{B_k} dt_{r} (t_k) (x_{r,k+1},y)\, d\mm(y)  \leq  e^{ D T}C_0^2C(N)r\mm(B_{r}(x))\delta.
\end{align*}
From \ref{eq:b-k-2} and Theorem \ref{thm:bg-in}, there are $C(D,T, N)>1$ and  $B_r'(x_{r,k+1})\subset B_r(x_{r,k+1})$ with 
\begin{align*}
dt_{r}  (t_k) (x_{r,k+1},y) &\leq C\sqrt{\delta}r \text{ for all } y\in B_r'(x_{r,k+1}),\\
\mm(B_r'(x_{r,k+1}))&\geq (1-\sqrt{\delta}) \mm(B_{r}(x_{r,k+1})).
\end{align*}
Thus if $ C \sqrt{\delta } < 1$, for all $t \in I_k$ we have 
\[\XX_t(B_r'(x_{r,k+1})) \subseteq B_{2r}(\XX_t(x_{r,k+1})).\] 
Also, from Theorem \ref{thm:bg-in} and \ref{eq:volume-distortion} we have for some $C(D,T, N) > 1$,
\begin{align}
    \notag\mm(B_r(\XX_t(x_{r,k+1})))&\geq \frac{1}{C}\mm(B_{2r}(\XX_t(x_{r,k+1})))\geq \frac{1}{C}\mm(\XX_t(B_r'(x_{r,k+1})) )\\ &\geq \frac{1}{C}\mm(B_r(x_{r,k+1})).\label{lowerbound}
\end{align}
To obtain the other direction of the volume estimate corresponding to \ref{eq:sr-2}, we consider the reversal of the flow. Fix $t \in I_k$, define $\overline{b} \in L^1 ([0,t]; H^{1,2}_{C,s}(TX))$ as
\[   \overline{b}(s) : = - b(t-s) \text{ for all }s \in [0,t] ,   \]
and let $\overline{\XX} : [0,t] \times X \to X$ be its RLF. Define $ dt_{r}' : [0,t] \times X \times X \to [0,r]$ as 
\begin{equation*}
    dt_{r}'(s)(x,y)=\sup_{0 \leq u \leq s} \dr (\overline{\XX}_u), 
\end{equation*}
 $S_1=\XX_t(B_r'(x_{r,k+1}))$, and $S_2=B_r(\XX_t(x_{r,k+1}))$. Similar as before we have
\begin{align*}
    & \int_{S_1 \times S_2} dt_{r}'(t )(x,y)\, d(\mm \times \mm)(x,y)\\
    \leq & \, e^{ D T}C_0^2C(N)r\mm(B_r(x_{r,k+1}))^2 \int_{0}^{t} \fint_{B_{20r}(\XX_{t-s}( x_{r,k+1}))}|\nabla \overline{b}(s) |\, d\mm \, ds\\
    \leq & \, e^{D T}C_0^2C(N)r\mm(B_r(x))^2 \int_{0}^{t} \mx _{\rho}(|\nabla b (s) |)(\XX_{t-s}(x_{r,k+1})) \, ds\\
    \leq & e^{D T}C_0^2C(N)r\mm(B_{r}(x))^2 \delta.
\end{align*}
Thus we have $x''\in \XX_t(B_r'(x_{r,k+1}))$ with
\begin{align*}
     \int_{B_r(\XX_t(x_{r,k+1}))} dt_{r}'( t )(x'',y)\, d\mm (y)   \leq  e^{D T}C_0^2C(N)r\mm(B_{r}(x))\delta.
\end{align*}
Hence there are $C(D,T,N)>0$ and $A'\subset B_r(\XX_t(x_{r,k+1}))$ such that 
\begin{align*}
dt_{r}'(t)(x'',y) &\leq Cr \sqrt{\delta} \text{ for all } y\in B_r(\XX_t(x_{r,k+1})),\\
\mm(A')&\geq ( 1- \sqrt{\delta})\mm(B_r(\XX_t(x_{r,k+1}))).
\end{align*}
Thus we have for some $C(D,T,N)>1$,
\begin{align}
\notag\mm(B_r(x_{r,k+1}))&\geq \frac{1}{C}\mm(B_{2r}(x_{r,k+1}))\geq \frac{1}{C}\mm(\overline{\XX}_{t}(A')) )\\
&\geq \frac{1}{C}\mm(B_r(\XX_t(x_{r,k+1}))).\label{upperbound}
\end{align}
Combining with \eqref{lowerbound} we have the desired volume bound, concluding the induction step and establishing \ref{eq:sr-1} and \ref{eq:sr-2} for $r \leq r_x$.

Take 
\begin{align*}
    x_r \in B_r(x) \cap \{ H \leq \delta \}, \, & \, B_r^{\prime}(x_r) \subset B_r(x_r),\\
 x_{r/2}\in B_{r/2}(x) \cap \{ H \leq \delta \} , \, & \,  B_{r/2}^{\prime} (x_{r/2}) \subset B_{r/2}(x_{r/2})   
\end{align*}
  given by the claim. That is, they satisfy \ref{eq:sr-1} and \ref{eq:sr-2} with $r$ and $r/2$ respectively. We claim that 
\begin{equation}\label{eq:second-claim}
 d( \XX_t (  x_r), \XX_t ( x_{r/2} ) ) \leq 20r \text{ for all } t\in [0,T].
\end{equation}
Take  $S_1=B_r(x)$ and $S_2=B_{r/2}^{\prime}(x_{r/2})$. Arguing as before, we can find  $x^{\prime}\in B_{r/2}^{\prime}(x_{r/2})$ with
\begin{align*}
     \int_{B_r(x)} dt_{r}(T)(x^{\prime},y)\, d\mm (y) \leq  e^{D T}C(N)r\mm(B_{r}(x))\delta,
\end{align*}
and $B_r^{\prime \prime}(x)\subset B_r(x)$ such that 
\begin{align}
dt_{r}(T)(x^{\prime},y) & < r   \text{ for all } y\in B_r^{\prime \prime }(x), \label{eq:b-prime-prime} \\ 
\mm(B_r^{\prime \prime }(x))&\geq ( 1-\sqrt{\delta}) \mm(B_{r}(x)). \label{eq:b-prime-prime-m}
\end{align}
Hence for all $t \in [0, T] $, $y \in B_r^{\prime }(x)$, using \ref{eq:b-prime-prime} we have
\begin{equation}\label{eq:b-prime-prime-y}
\begin{split}
    d ( \XX _t (x_{r/2}), \XX _ t (y)  )  & \leq d(  \XX _t (x_{r/2}), \XX_t(x^{\prime})  ) + d(  \XX _t (x^{\prime})  , \XX_t (y) ) \\
    & \leq  r + r + d( x^{\prime}, y ) \leq 4r.          
\end{split}
\end{equation}
In a similar fashion, one can find a subset $B^{\prime \prime \prime }_r(x) \subset B_r(x) $ with 
\begin{gather}
\textbf{X}_t(B_r^{\prime \prime \prime}(x)) \subset B_{10r} ( \XX_t(x_r) )  \label{eq:b-prime-prime-prime-d} \\
    \mathfrak{m} ( B_r^{\prime \prime \prime }(x) ) \geq (1 - \sqrt{\delta }) \mm ( B_r(x)).\label{eq:b-prime-prime-prime-m}
\end{gather}                
From \ref{eq:b-prime-prime-m} and \ref{eq:b-prime-prime-prime-m}, one can find $z \in B_r^{\prime \prime } (x) \cap B_r^{\prime \prime \prime } (x) $.  Then from \ref{eq:b-prime-prime-y} and \ref{eq:b-prime-prime-prime-d} applied to $z$, we conclude \ref{eq:second-claim}. 

Notice that by iterated applications of \ref{eq:second-claim},  for $r_1,r_2 \leq r_x$ and $t \in [0,T]$, one gets
\begin{equation}\label{eq:adjusted-well-defined}
    d(  \XX_t(x_{r_1}), \XX_t(x_{r_2})  ) \leq 100 \max \{ r_1, r_2 \}.
\end{equation}
Now let us adjust the Regular Lagrangian Flow. Let $S \subset X$ denote the set of $x$ satisfying \ref{eq:good-point-hypothesis} and for each $x \in S$ define $r_k \leq  \rho / 100 $ such that for all $r \leq  r_x $ one has
\[    \dfrac{\mm ( \{ H > \delta \} \cap B_r(x) )}{\mm (B_r(x))} \leq \frac{1}{2} .                       \]
Using $x_r$ from our first claim, by Kuratowski-Ryll-Nardzewski measurable selection theorem we can take a measurable choice of $x_r$ to define a measurable map $\tilde{x}:\mathbb{R}^+\times X\to X$ as
\[    \tilde{x} (r,x) := \begin{cases}
    x_r  & \text{ if }\,  x \in S, \, r \leq r_x \\
    x & \text{ else. }
\end{cases}          \]
Then let us define an adjusted flow $\tilde{\XX}:  [0,T] \times X \to X$ as\begin{align}
    \tilde{\XX}(x,t)=\lim _{r\to 0}\XX(\tilde{x}(r,x),t).
\end{align}
By \ref{eq:adjusted-well-defined}, the limit exists and satisfies \ref{eq:es-st-1} and \ref{eq:es-st-2} for all $x \in S$ and $r \leq r_x $. Now we need to verify that $\tilde{\XX}$ is also a Regular Lagrangian Flow.

\ref{def:rlf-1}  holds as \ref{eq:sob-to-lip-flow} passes to the limit trajectories.  Given $x\in S$, $r \leq r_x$, choose a set $A_r(x)\subset B_r(x)$ satisfying \ref{eq:es-st-1}, and consider the probability measure 
\begin{align}
    \mu_{r,x}=\frac{\chi_{A_r(x)}}{\mm(A_r(x))}\mm.
\end{align}
From  the definition of $\XX$, for all $f\in \test  (X)$ we have 
\begin{equation}\label{eq:mu-derivative}
    \frac{d}{dt}\int_X f d\mu_{r,x}(t)=\int_X df ( b ( t )) \, d\mu_{r,x}(t).
\end{equation} 
Also notice that by \ref{eq:es-st-1}, for all $t \in [0,T]$ we have 
\begin{equation}\label{eq:delta-convergence}
    \supp ( \mu _{r,x} (t)  )  \subset B_{2r} (\tilde{\XX}_t(x))  ,
\end{equation}   
and by \ref{eq:es-st-2}, the map $\mu _{r,x} (t) \xmapsto{\theta} \tilde{\XX}_t (x) $ makes 
\[  \{ \mu_{r,x}(t) \vert x \in S , r \leq r_x , t \in [0,T] \} \]
a family of bounded eccentricity.  Let $f\in \test (X)$ and $t \in [0,T]$.  From Lemma \ref{lem:leb} we have for $\mm$-a.e. $x\in S$,
\begin{align*}
    df (b(t) ) (\tilde{\XX}_t(x))  = \lim_{r\to 0}\int_X  df (b(t)) d \mu_{r,x}(t).
\end{align*}
Hence, for all $t_0, t_1 \in [0,T]$ and $\mm$-a.e. $x \in S$,
\begin{equation*}
\begin{split}
    \int_{t_0}^{t_1}   df (b(t) ) (\tilde{\XX}_t(x)) dt       &= \lim_{r\to 0}\int_{t_0}^{t_1}\int_X  df (b(t) ) d \mu_{r,x}(t) \\
    & =\lim_{r\to 0}\int_Xf d\mu_{r,x}(t_1)-\int_Xf d \mu_{r,x}(t_0)\\
    &=f(\tilde{\XX}_{t_1}(x))-f(\tilde{\XX}_{t_0}(x)),
\end{split}
\end{equation*}
where we used Dominated Convergence on the first two lines, \ref{eq:mu-derivative} on the second one, and \ref{eq:delta-convergence} on the third one. This implies $\tilde{\XX}$ satisfies \ref{def:rlf-2} for $\mm$-almost all $x \in S$. Since $\tilde{\XX}$ and $\XX$ coincide on $X \backslash S$, and $\XX$ satisfies \ref{def:rlf-2}, we deduce  $\tilde{\XX}$ satisfies \ref{def:rlf-2} as well. 

Let $S_k : = \{ x \in S \vert r_x \geq \frac{1}{k }\}$. We claim there is $\mathfrak{C} > 0 $ such that for all $r \leq \frac{1}{k}$, $y \in X$, we have
\begin{equation}\label{eq:y-bound}
    \int_{S_k} \frac{\chi _{\XX_t(A_r(x))} (y) }{\mm ( A_r (x)  )} \text{d} \mm (x) \leq \mathfrak{C} .
\end{equation}
This follows from the computation
\begin{equation*}
\begin{split}
    \int_{S_k} \frac{\chi _{\XX_t(A_r(x))} (y) }{\mm ( A_r (x)  )} d \mm (x) & \leq  M^2 \int_{S_k} \frac{\chi _{B_{2r}(y)} ( \XX_t (x) )}{ \mm ( B_r(\XX _t (x)) ) } d \mm (x)   \\
    & \leq M^2 \int_{B_{2r}(y)} \frac{1}{\mm (B_r(z) )} d ( (\XX_t)_{\ast} (\mm ) ) (z)\\
    & \leq M^2 C (b) \int_{B_{2r}(y)} \frac{1}{\mm ( B_r(z)  ) } d \mm (z) \\
    & \leq M^2 C (b) C(N), 
\end{split}
\end{equation*}
where the first line follows from \ref{eq:es-st-1} and \ref{eq:es-st-2}, the second from a change of variables, the third from  \ref{def:rlf-3} applied to $\XX$, and the fourth from Theorem \ref{thm:bg-in}. Then, given any $0\leq f\in \test (X)$, we have
\begin{equation*}
\begin{split}
\int_X f d ( (\tilde{ \XX} _t)_{\ast} (\mm \vert _{S_k}) ) & = \int_{S_k} (f \circ \tilde{\XX}_t) d \mm \\
&  \leq \lim _{r \to 0}  \int_{S_k} \int_X f(y) d \mu _{r,x} (t) (y) d \mm (x)          \\
&  = \lim _{r \to 0 } \int_X f(y) \int_{S_k} \frac{ \chi _{X_t(A_r(x))}(y)}{\mm (A_r(x))} d\mm  (x) d ((\XX_t)_{\ast} \mm) (y)\\
& \leq \mathfrak{C} \int_X f (y) d((\XX_t)_{\ast}\mm ) (y) \\
& \leq C \mathfrak{C} \int_X f d \mm ,
\end{split}
\end{equation*}
where we used \ref{eq:delta-convergence} on the second line, Tonelli's Theorem on the third line, \ref{eq:y-bound} on the fourth, and \ref{def:rlf-3} applied to $\XX$ on the fifth. This implies that $(\tilde{\XX})_{\ast}\mm \vert_{S_k} \leq \tilde{ \mathfrak{C}} \mm$ for some $\tilde{\mathfrak{C}}>0$ independent of $k$. Hence
\begin{equation*}
\begin{split}
( \tilde{\XX}_t )_{\ast} \mm & = (\tilde{\XX}_t)_{\ast} (\mm \vert _{X \backslash S}) +  (\tilde{\XX}_t)_{\ast} (\mm \vert _{S}) \\
& = (\XX_t)_{\ast} (\mm \vert _{X \backslash S}) + \lim_{k \to \infty}  (\tilde{\XX}_t)_{\ast} (\mm \vert _{ S_k})\\
& \leq ( C + \mathfrak{C}) \mm ,
\end{split}
\end{equation*}
establishing \ref{def:rlf-3} for $\tilde{\XX}$, so we conclude $\tilde{\XX}$  is an RLF for $b$.
\end{proof}

\begin{proof}[Proof of Corollary \ref{cor:essential-stability}:]
By Proposition \ref{pro:max-properties}(\ref{pro:max-properties-2}), for all $s \in [1, R-1]$, one has 
\[   \fint_{B_{s}(p)} \mx (  \vert  \nabla b  \vert  )^2 d \mm  \leq C(N) \eta .                     \]
 Define $H : X \to \mathbb{R}$ as $H(x) : = \int_0^T \mx  (  \vert \nabla b  \vert  ) ( \XX_t(x) ) dt$. Then by the Cauchy-Schwarz inequality and \ref{eq:volume-distortion} one gets
\begin{equation}\label{eq:h-corollary}
\begin{split}
& \left[ \fint_{B_r(p)}  H(x) d\mathfrak{m}(x) \right] ^2 \\
& \leq \fint_{B_r(p)} \left[  \int_0^T \mx (  \vert  \nabla b  \vert  )  (\XX _t (x)) dt  \right] ^2 d\mathfrak{m}(x) \\
& \leq \, T  \fint_{B_r(p)} \int_0^T \mx (  \vert  \nabla b  \vert  )  ^2 (\XX _t (x)) dt d\mathfrak{m}(x)\\
& \leq \, T  e^{2DT} \int_0^T \fint_{ \XX _t (B_r (x)) }    \mx (  \vert  \nabla b  \vert  )^2 (\XX _t (x))  d\mathfrak{m}(x) dt \\
& \leq \,  C(D,T, N) \int_0^T \fint_{ B_{r + LT } (x) }    \mx ( \vert \nabla b \vert )^2 (\XX _t (x))  d\mathfrak{m}(x) dt \\
& \leq C(D,T,N) \eta. 
\end{split}
\end{equation}
Let $\delta ( D,T, N) > 0 $ be given by Theorem \ref{thm:essential-stability}. From \ref{eq:h-corollary}, there is $C(D,T, N) > 1 $ such that
\begin{equation}\label{eq:h-corollary-2}
    \mm ( \{ H \leq \delta \}  \cap B_r(p)  ) \geq ( 1 - C \sqrt{\eta}  ) \,  \mm ( B_r(p) ) .     
\end{equation}   
By Theorem  \ref{thm:essential-stability}, $G$ contains the density points of $\{ H \leq \delta \}$, so the result follows from \ref{eq:h-corollary-2} provided $\eta \leq \varepsilon ^2 /C^2 $.
\end{proof}

\section{Self-improving stability}\label{sec:sis}

In this section, we use essential stability to obtain geometric control on RLFs from integral control on the covariant derivative of the corresponding vector fields. 
\begin{Lem}\label{lem:boot-strap}
\rm For each $N \geq 1$, $M \geq 1$,  there are  $\lambda (N) \geq 4$, $\varepsilon (N,M) > 0$, such that the following holds. Let $(X, d , \mm ) $ be an $ \rcd (-(N-1),N)$ space, $x \in X$, $V \in L^1 ( [0,T]; H^{1,2}_{C,s}(TX))$ a divergence free vector field, and $\XX: [0,T] \times X \to X$ the RLF of $V$. Assume 
\[     \int_0^T \mx _4 ( \vert \nabla V (t) \vert ) ( \XX_t(x)) dt \leq \varepsilon ,              \]
for some $r \leq 1/\lambda $ one has 
\begin{equation}\label{eq:measure-control-bootstrap-lambda}
    \frac{1}{M} \mm (B_r(x)) \leq \mm  (B_r(\XX_t(x))) \leq  M  \mm (B_r(x)) \text{ for all }t \in [0,T],  
\end{equation}     
and there is $ S_r \subset B_r(x)$ with   
\begin{gather*}
\mm  (S_r) \geq \frac{1}{M} \mm (B_r(x)),  \\
\XX_t(S_r) \subset B_{2r}(\XX_t(x)) \text{ for all }t \in [0,T].        
\end{gather*}
Then 
\begin{equation}\label{eq:boot-strap-measure}
\frac{1}{2} \mm (B_{\lambda r }(x)) \leq \mm  (B_{\lambda r }(\XX_t(x))) \leq  2 \mm (B_{\lambda r}(x)) \text{ for all }t \in [0,T],
\end{equation}
and there is $A_{\lambda r} \subset B_{\lambda r}(x) $ with 
\begin{equation}\label{eq:better-stability-1}
\mathbf{X}_t(A_{\lambda r}) \subset B_{(\lambda + 4 ) r } ( \XX_t(x)) \text{ for all }t \in [0,T],            
\end{equation}     
\begin{equation}\label{eq:better-stability-2}
 \mathfrak {m}  (  \XX_t (A_{\lambda r }) \cap B_{\lambda r} (  \XX_t (x) )  ) \geq \frac{9}{10} \mm  (B_{\lambda r} (\XX_t(x))) \text{ for all } t \in [0,T].      
\end{equation}
\end{Lem}

\begin{proof} 
Pick $\lambda (N) \geq 5 $ such that 
\begin{equation}\label{eq:lambda-bg}
\mm  (B_{(\lambda + 5) s } (y) ) \leq \frac{101}{100} \mm  (B_{\lambda s} (y) )               
\end{equation}
for all $y \in X$, $s \leq 10 / \lambda $. With this choice of $\lambda$, \ref{eq:boot-strap-measure} will follow from \ref{eq:better-stability-1} and \ref{eq:better-stability-2}.  Let $\dr (t)$ be given by \ref{eq:ddrt} and 
\[\Gamma (t) : = \{ (a,b) \in S_r \times B_{\lambda r}(x) \vert \dr (t) (a,b) < r \}.\]
Notice that for each $t \in [0, T]$, $(a,b)  \in \Gamma (t)$, one has $d(\XX_t(a), \XX_t(x)) \leq 2r$, and 
\begin{equation}\label{eq:sausage}  
\begin{split}
d( \XX_t(b), \XX_t(x) ) & \leq d(\XX_t(b), \XX_t(a)) + d(\XX_t(a), \XX_t(x)) \\
& \leq d(a,b) + r + 2 r\\
& \leq  \lambda r + 4r.
\end{split}
\end{equation}
Then
\begin{equation}\label{eq:as-usual-1}
\begin{split}
  &\fint_{S_r \times B_{ \lambda r}(x)}  \dr(T)  d ( \mm \times \mm )   \\
  & \leq \dfrac{ M }{ \mm (B_{ r}(x))^2 }   \int_0^T  \int_{\Gamma(t)}  d( \XX_t(a),\XX_t(b)) \left[ \int_0^1 \vert \nabla V (t) \vert (\gamma_{\XX_t (a), \XX_t (b)}(s)) ds \right] d (\mm \times \mm )(a,b)dt \\
  & \leq \dfrac{  M   }{ \mm (B_{ r}(x))^2 }   \int_0^T  \int_{ \XX_t (\Gamma(t))}  d( a,b) \left[ \int_0^1 \vert \nabla V (t) \vert (\gamma_{a,b}(s)) ds \right] d (\mm \times \mm )(a,b)dt \\
  & \leq  \dfrac{  M }{ \mm (B_{ r}(x))^2 }   \int_0^T  \int_{ B_{(\lambda + 4 )r} (\XX_t(x))^{\times 2} }  d( a,b) \left[ \int_0^1 \vert \nabla V (t) \vert (\gamma_{a,b}(s)) ds \right] d (\mm \times \mm )(a,b)dt \\
  & \leq C(N) \cdot M^3 \cdot r \cdot \int_0^T  \fint _{B_{(2 \lambda + 8)r  } ( \XX_t(x))} \vert \nabla V (t) \vert (y) d\mm (y) dt \\
  & \leq C(N) \cdot M ^3  \cdot r \cdot  \int_0^T \mx _4 (\vert \nabla V (t) \vert )(\XX_t(x)) dt \leq  C(N) M ^3 \varepsilon r,
 \end{split}
\end{equation}
where the first inequality follows from Theorem \ref{thm:first-variation}, the second one from Tonelli Theorem, the third one from \ref{eq:sausage}, and the fourth one from Theorem \ref{thm:segment} and \ref{eq:measure-control-bootstrap-lambda}. Hence there is $y \in S_r$ with 
\begin{equation}\label{eq:many-dont-escape}
     \fint_{B_{ \lambda r}(x)} \dr (T) (y,z) d\mm (z) \leq C(N) M ^3 \varepsilon r  .
\end{equation}   
We can then define $A_{ \lambda r} : = \{ z \in B_{\lambda r}(x) \vert \dr (T) (y,z)  < r \} $, which by \ref{eq:sausage} satisfies \ref{eq:better-stability-1}, and by \ref{eq:many-dont-escape} satisfies  
\begin{equation} \label{eq:points-dont-escape}
    \mm (A_{\lambda r}) \geq (1- C(N)M^3 \varepsilon ) \mm  (B_{\lambda r }(x))  \geq \frac{99}{100} \mm (B_{\lambda r} (x)), 
\end{equation}  
provided $\varepsilon$ is small enough. To verify \ref{eq:better-stability-2}, fix $t \in [0,T]$, consider the vector field $\overline{V}  \in L^1( [0, t], H^{1,2}_{C,s} ( 
 TX ) )$ given by $\overline{V} ( s ) : = - V (t-s)$, and $\overline{\XX} : [0,t] \times X  \to X$ its RLF. Also set 
 \begin{gather*}
  \overline{\dr}(\cdot) (\cdot, \cdot ) : [0,t] \times X \times X \to \mathbb{R} \\
    \overline{\dr} (s):= \sup_{u \in [0,s]} \dr ( \overline{X}_u ),      
 \end{gather*} 
and define $\overline{\Gamma}  (s) : = \{ (a,b) \in \XX_t(S_r) \times B_{\lambda r}(\XX_t(x)) \vert \overline{\dr} (s) (a,b) < r \}$. Then for all $s \in [0,t]$ and $(a,b) \in \overline{\Gamma}  (s)$, one has $d ( \overline{\XX}_s(a) , \XX_{t-s} (x) ) \leq 2 r $, and  
\begin{equation}\label{eq:sausage-2}
\begin{split}
    d( \overline{\XX}_s(b), \XX_{t-s}(x) ) & \leq   d( \overline{\XX}_s(b), \overline{\XX}_s(a)  )  +  d( \overline{\XX}_s(a), \XX_{t-s}(x) )       \\
    &   \leq   d( a,b) + r + 2r          \\
    &   \leq   \lambda r + 5 r. 
\end{split}
\end{equation}                        
As in \ref{eq:as-usual-1}, using \ref{eq:sausage-2} instead of \ref{eq:sausage} we get
\begin{equation*}
\begin{split}
  &\fint_{\XX_t(S_r) \times B_{\lambda r}(\XX_t(x))}  \overline{\dr} (t)  d(\mm \times  \mm )  \\
  & \leq \dfrac{ 2  M  }{ \mm (B_{r}( x))^2 }   \int_0^t  \int_{\overline{\Gamma}  (s)}  d(\overline{ \XX}_s(a),\overline{\XX}_s(b)) \left[ \int_0^1 \vert \nabla \overline{ V} ( s) \vert (\gamma_{\overline{\XX}_s (a), \overline{\XX}_s (b)}(u)) du \right] d (\mm \times \mm )(a,b)ds \\
  & \leq  \dfrac{ 2  M  }{ \mm (B_{r}(  x))^2 } \int_0^t  \int_{ \overline{\XX}_s (\overline{\Gamma}  (s))}  d( a,b) \left[ \int_0^1 \vert \nabla \overline{V} ( s) \vert (\gamma_{a,b}(u)) du \right] d (\mm \times \mm )(a,b)ds \\
  & \leq  \dfrac{ 2  M  }{ \mm (B_{ r}( x))^2 }   \int_0^t  \int_{ B_{(\lambda + 5 )r} (\XX_{t-s}(x))^{\times 2} }  d( a,b) \left[ \int_0^1 \vert \nabla \overline{V} ( s) \vert (\gamma_{a,b}(u)) du \right] d (\mm \times \mm )(a,b)ds \\
  & \leq C(N)  M^3  r  \int_0^t  \fint _{B_{(2 \lambda + 10)r  } ( \XX_{t-s}(x))} \vert \nabla \overline{V} ( s) \vert (y) d\mm (y) ds \\
  & \leq C(N)  M ^ 3  r  \int_0^t \mx _4 (\vert \nabla V ( t-s)\vert )(\XX_{t-s}(x)) ds \leq  C(N) M^3 \varepsilon r.
  \end{split}
  \end{equation*}
  Pick $y_t \in \XX_t(S_r)$ such that 
  \[   \fint_{B_{\lambda r} (\XX_t(x)) } \overline{\dr} (t) (y_t, z) d\mm (z) \leq C(N) M^3 \varepsilon r.    \]
  Then the set $  A^t_{\lambda r} : =  \{ z \in  B_{\lambda r} (\XX_t(x)) \vert \overline{\dr}(t) (y_t, z) < r   \}  $ satisfies 
  \[\mm  (A^t_{\lambda r}) \geq (1 - C(N) M^3 \varepsilon ) \mm  (B_{\lambda r }(\XX_t(x))).\]
  Since $\overline{X}_t(A^t_{\lambda r}) \subset B_{(\lambda + 5)r} (x) $ and $\overline{X}_t$ is measure preserving, using \ref{eq:lambda-bg} we have 
  \begin{equation}\label{eq:boot-strap-measure-2}
  \mm  (B_{\lambda r}(x) ) \geq \dfrac{98}{100} \mm  (B_{\lambda r} (\XX_t(x))) 
  \end{equation}
  provided $\varepsilon$ is small enough. We conclude 
\begin{equation*}
\begin{split}
&\mm   ( \XX_t(A_{\lambda r}) \cap B_{\lambda r} (\XX_t (x))  )  \\
&\geq \mm  ( \XX_t (A_{\lambda r} )) - \mm  ( B_{(\lambda + 4 ) r}(\XX_t(x)) \backslash B_{\lambda r } (\XX_t(x)) )  \\
&\geq  \mm  (A_{\lambda r} ) - \frac{1}{100} \mm  ( B_{\lambda r} (\XX_t(x)) )\\
&\geq \frac{99}{100} \mm  ( B_{\lambda r}(x) )  - \frac{1}{100} \mm  ( B_{\lambda r} (\XX_t(x)) ) \\
&\geq \frac{9}{10} \mm  ( B_{\lambda r} (\XX_t(x)) ),
\end{split}
\end{equation*}
where we used \ref{eq:better-stability-1} on the first inequality,  \ref{eq:lambda-bg} on the second, \ref{eq:points-dont-escape} on the third, and \ref{eq:boot-strap-measure-2} on the fourth.
\end{proof}

\begin{Cor}\label{cor:e-s-better}
\rm Let $(X, d , \mm )$ be an $ \rcd (-(N-1),N)$ space, $x \in X$, $V \in L^1 ( [0,T] ;  H^{1,2}_{C,s}(TX))$ a divergence free vector field, and $\XX: [0,T] \times X \to X$ the RLF of $V$. Assume $x$ is a point of essential stability of $\XX$ and 
\[   \int_0^T \mx _4 ( \vert \nabla V (t) \vert ) (\XX_t(x)) dt \leq \varepsilon.         \]
If $\varepsilon$ is small enough, depending only on $N$, then for all $r \leq 1$, $t \in [0, T]$, one has 
    \[     \frac{1}{2} \mm (B_r(x)) \leq \mm  (B_r(\XX_t(x))) \leq  2 \mm (B_r(x)).   \]
\end{Cor}

\begin{proof}   
By the definition of essential stability, there is $M(N) > 0 $ for which the hypotheses of  Lemma \ref{lem:boot-strap} hold for small enough $r\leq 1$. By Lemma \ref{lem:boot-strap}, if they hold for a certain $r$, then they hold for $\lambda r$ so we can apply Lemma \ref{lem:boot-strap} repeatedly, and \ref{eq:boot-strap-measure} is valid for all $r \leq 1/\lambda$.  
\end{proof}

\begin{Pro}\label{pro:e-s-better-2}
\rm There is $C_0(N) > 0$ such that under the conditions of Corollary \ref{cor:e-s-better}, for all $r \leq 1$ there is $A_r \subset B_r(x)$ such that 
\begin{gather}
\mm (A_r) \geq (1 - C_0 \varepsilon ) \mm (B_r(x))\label{e-s-better-2-1} \\
\XX_t(A_r) \subset B_{2r} ( \XX_t(x)) \text{ for all }t \in [0,T].  \label{e-s-better-2-2}
\end{gather}
\end{Pro}

\begin{proof}   
By the definition of essential stability, for $r\leq 1$ sufficiently small, 
there is $S_{r/10} \subset B_{r/10}(x) $ with $\mm (S_{r/10}) \geq \frac{1}{M(N)} \mm  (B_{r/10}(x)) $ and $ \XX_t (S_{r/10}) \subset B_{r/5} (\XX_t (x))$ for all $t \in [0,T]$.  Set 
\begin{gather*}
    \drt (\cdot) (\cdot, \cdot ) : [0,T] \times X \times X \to [0, r/10] \\
      \drt (t) : =  \sup_{s \in [0,t]}   \drt (\XX_s ) ,  
\end{gather*}  
and $\Gamma(t) : = \{ (a,b) \in S_{r/10} \times B_{r}(x) \vert \drt (t) (a,b) < r / 10 \}$. Then for $(y,z) \in \Gamma (t)$, for each $t \in [0,T]$ we have
\begin{equation}\label{eq:medium-scale-control}
\begin{split}
d(   \XX_t (b), \XX_t(x) ) & \leq d ( \XX_t (b), \XX_t(a)) + d (\XX_t (a), \XX_t(x) ) \\
& \leq  d(a,b) + r/10 + r/5 \\
& \leq r/2 + r/10 + r/10 + r/5 < r ,
\end{split}
\end{equation}  
so as in \ref{eq:as-usual-1}, using Corollary \ref{cor:e-s-better} one gets
\begin{equation}\label{eq:as-usual-2}
\begin{split}
  &\fint_{S_{r/10} \times B_{r}(x)}  \drt(T)  d(\mm \times \mm ) \\
  &  \leq C(N)  \cdot r \cdot  \int_0^T \mx _4 (\vert \nabla V (t) \vert )(\XX_t(x)) dt \leq  C(N)  \varepsilon r.
 \end{split}
\end{equation}
 Then there is $y \in S_{r/10}$ with 
\begin{equation}\label{eq:stability-medium-scales}
     \fint_{B_{r}(x)} \drt (T) (y,z) d \mm (z) \leq C(N) \varepsilon r.
\end{equation}   
We can then define $A_{r} : = \{ z \in B_{r}(x) \vert \drt (T) (y,z)  < r/10 \} $, which  by \ref{eq:stability-medium-scales} satisfies \ref{e-s-better-2-1} and by \ref{eq:medium-scale-control} also \ref{e-s-better-2-2}. The above analysis shows that \ref{e-s-better-2-1} and \ref{e-s-better-2-2} hold for all $r$ small enough. An identical argument (using $A_{r/10}$ instead of $S_{r/10}$) shows that if \ref{e-s-better-2-1} and \ref{e-s-better-2-2} hold for some $r /10 \leq 1/10$, then they hold for $r$.  
\end{proof}
\begin{Pro}\label{pro:e-s-better-3}
\rm There is $C_0 (N) > 0 $ such that under the conditions of Corollary \ref{cor:e-s-better},  for all $r \leq 1$, one has 
    \[    \fint_{B_{r}(x)^{\times 2}} \dr  ( \XX_T )  d ( \mm \times \mm  ) \leq  C_0 \varepsilon r.          \]
\end{Pro}

\begin{proof} 
For $A_r \subset B_r(x)$ given by Proposition \ref{pro:e-s-better-2}, a computation analogous to  \ref{eq:as-usual-1} yields
 \[ \fint_{A_{r}^{\times 2}}  \dr( \XX_{T} )    d ( \mm \times \mm  )  \leq   C(N)  \varepsilon r.\] 
Combining this with \ref{e-s-better-2-1}, we get the result.
\end{proof}

\begin{Def}\label{def:weak-e-s}
\rm Let $(X,d,\mm )$ be an $\rcd (-(N-1), N)$ space, $V_1, \ldots , V_k \in L^1 ( [0, 1] ; H^{1,2}_{C,s}(TX))$ divergence free vector fields,  $\XX^j : [ 0, 1 ] \times X \to X $ their RLFs, and $ x_1 , \ldots , x_k   \in X$ such that
\begin{itemize}
    \item $x_j$ is a point of essential stability of $\XX^j$ for each $j \in \{ 1, \ldots , k \}$.
    \item $X_1^j (x_j) = x_{j+1}$ for each $j \in \{ 1, \ldots , k-1 \} $.
\end{itemize}
If $V \in L^1 ( [0,1] ; H^{1,2}_{C,s}(TX))$ is given by
\[
    V(t)  : = k \cdot  V_j ( kt - j + 1 ) \text{ for }t \in \left[ \frac{j-1}{k}, \frac{j}{k}  \right] ,
\]   
and $\XX : [0, 1] \times X \to X$ is its RLF, then we say $x_1$ is a point of \textit{weak essential stability} of $\XX$.
\end{Def}

\begin{Pro}\label{pro:connect}
\rm  Under the conditions of Definition \ref{def:weak-e-s}, there is $\eta (N) > 0$ such that if 
\[   \int_0^1  \mx ( \vert \nabla V (t) \vert ) ( \XX _t( x_1) ) dt = \sum_{j=1}^k \int_{0}^1 \mx ( \vert \nabla V_j (t) \vert ) (\XX^j_t(x_j) ) dt \leq \eta ,  \]
then  $x_1$ is a point of essential stability of $\XX$. 
\end{Pro}

\begin{proof}
By induction we can assume $k = 2$. By Corollary \ref{cor:e-s-better}, and Proposition \ref{pro:e-s-better-2}, we can apply  Lemma \ref{lem:boot-strap} to both $X^1$ and $X^2$, provided $\eta$ is small enough. By \ref{eq:boot-strap-measure}, for each $r \leq 1$, $t \in [0,1]$, we get
\[   \frac{1}{4} \mm  (B_r(x_1)) \leq  \mm  (B_r( \XX_t(x_1)))   \leq 4  \mm  (B_r(x_1))  .           \]
Let $\overline{V}_1 \in L^1 ( [0,1] ; H^{1,2}_{C,s}(TX) )$ be given by $\overline{V} _1(t) : = - V_1(1-t) $ and let $\overline{\XX}^1 : [0,1] \times X \to X$ be its RLF. Again by Lemma \ref{lem:boot-strap}, for $r$ small enough there are sets $A^1_r , A_r^2 \subset B_r(x_2) $ such that 
\begin{gather*}
\overline{\XX}_t^1 (A_r^1) \subset B_{2r} (\overline{\XX}_t^1 (x_2)) , \text{ }\XX_t^2 (A_r^2) \subset B_{2r} ( \XX_t^2 (x_2)) \\
 \mm  ( \overline{\XX}_t^1 (  A^1_r ) \cap B_r ( \overline{\XX}_t^1 (x_2) ) )   \geq \frac{9}{10} \mm  ( B_r ( \overline{\XX}_t^1 (x_2) ) )   \\
 \mm  ( \XX_t^2 (  A^2_r ) \cap B_r ( \XX^2_t (x_2) ) )   \geq \frac{9}{10} \mm  ( B_r ( \XX^2_t (x_2) ) )
\end{gather*}
 for all $t \in [0,1] $. Then $A_r : =  B_r(x_1)  \cap \overline{\XX}_1^1 ( A_r ^1 \cap A_r^2 )$ satisfies 
 \begin{equation*} 
      \XX_t(A_r) \subset B_{2r} ( \XX_t(x) ) \text{ for all }t \in [0,1], 
 \end{equation*}  
 and using \ref{eq:boot-strap-measure} we conclude 
\begin{equation*}
\begin{split}
     \mathfrak {m} (A_r ) & \geq   \mm  ( B_r(x_1) \cap \overline{\XX}_1^1 ( A_r ^1 )  - \mm ( A_r^1 \backslash A_r^2  )   \\
     & \geq \frac{9}{10} \mm  ( B_r(x_1) ) - \frac{1}{10} \mm  (B_r(x_2))\\
     & \geq \left[  \frac{9}{10} - \frac{1}{5}   \right] \mm  ( B_r(x_1) ) \\
     & \geq \frac{1}{2} \mm  (  B_r(x_1)).
\end{split}
\end{equation*}
\end{proof}

\section{Properties of GS maps}\label{sec:pgas}

In this section we prove the main properties of GS maps; they converge weakly to an isometry (Lemma  \ref{gas:limit}), have the zoom-in property (Proposition \ref{gas:zoom}), and can be concatenated (Proposition \ref{gas:composition}).

\begin{Rem}\label{rem:good-good-sets}
\rm From the absolute continuity condition of Definition \ref{def:gas}, we can assume that for all $x \in U_i^1$, $y \in U_i^2$, we have $(f_i^{-1})^{-1}(x) = \{ f_i(x) \}$ and $(f_i)^{-1}(y) =  \{ f_i^{-1}(y)\}$.
\end{Rem}

\begin{Def}
\rm  For $j \in \{ 1, 2 \} $, let $(X_i^j, d_i^j , \mm _i^j,p_i^j) $,  be a sequence of pointed  $\rcd (K,N)$ spaces for which $(X_i^j,p_i^j)$  converges to $(X^j,p^j)$  in the Gromov--Hausdorff sense. We say a sequence of maps $f_i : X_i^1 \to X_i^2$ converges weakly to $f_{\infty} : X^1 \to X^2$ if there is a sequence of subsets $U_i \subset X_i^1$ with asymptotically full measure such that  
\begin{equation}\label{eq:weak-limit}
 \lim_{i \to \infty } \sup_{x \in U_i} d(  \varphi _i ^2 f_i (x) , f _{\infty}\varphi _i ^1 (x)  ) = 0 ,             
\end{equation}
where $\varphi _i^j : X_i^j \to X^j \cup \{ \ast \} $ are Gromov--Hausdorff approximations for $j \in \{ 1, 2 \}$. 
\end{Def}

\begin{Lem}\label{gas:limit}
\rm For $j \in \{ 1, 2 \} $, let $(X_i^j, d_i^j , \mm _i^j,p_i^j) $,  be a sequence of pointed  $\rcd (K,N)$ spaces for which $(X_i^j,p_i^j)$  converges to $(X^j,p^j)$  in the Gromov--Hausdorff sense.  If $f_i : [X_i^1, p_i^1] \to [X_i^2, p_i^2] $ is a sequence of GS maps, then, after taking a subsequence, $f_i$ converges weakly to an isometry $f_{\infty}: X^1 \to X^2$.
\end{Lem}

\begin{proof}
Let $R_0 > 0 $, $\varepsilon_i \to 0$, and $S_i^j, U_i ^j \subset X_i^j$ be given by Definition \ref{def:gas} (see Remark \ref{rem:good-good-sets}). \\

\noindent \textbf{Step 1:} For $i$ large enough, $x \in U_i^1$, $r \leq 1$, there is $A \subset B_r(x )$ such that 
\[ f_i(A) \subset B_{2r}(f_i (x)) ,\text{ and } \mm _i^1 (A) \geq \frac{1}{2}\mm _i^1 (B_r( x )) .  \]
From the definition of essential continuity, the statement holds for $r$ small enough (depending on $x$). We now see that if $i$ is large enough, and there is $A_0 \subset B_{r/10}(x )$ such that $f_i(A_0) \subset B_{r/5}(f_i (x))$, and $ \mm _i^1 (A_0) \geq \frac{1}{2}\mm _i^1 (B_{r/5}( x )) $, then there is $A \subset B_r(x )$ such that $f_i(A) \subset B_{2r}(f_i (x))$, and $ \mm _i^1 (A) \geq \frac{1}{2}\mm _i^1 (B_r( x)).$ Since there is $C (K,N) > 0$ such that $   \mm _i^1 ( B_r(x) ) \leq C \mm _i^1 (A_0), $ we have
\[      \fint _{B_r(x_i^1)\times A_0 } \dr (f_i)  d(\mm_i^1 \times \mm_i^1 ) \leq  C r \varepsilon_i .  \]
Hence if $A : = \{ y \in B_r(x) \vert d(fy,fx ) < 2r \} $, one gets 
\[  r \cdot \frac{  \mm _i^1( B_r(x) \backslash  A) }{ \mm _i^1( B_r(x)) }   =   \frac{\int_{ (B_r(x) \backslash A) \times A_0 } \dr (f_i) d(\mm_i^1 \times \mm_i^1 ) }{ \mm _i^1( B_r(x)) \cdot \mm _i^1 ( A_0) }   \leq         Cr \varepsilon_i, \]
implying that $\mm _i^1 (A ) \geq \frac{1}{2} \mm _i^1( B_r(x) )$ provided $\varepsilon_i \leq 1/2C$.\\

\noindent \textbf{Step 2:} For all distinct $x_i, y_i \in U_i^1$ with $d(x_i , y_i) \leq 1/2$, one has  
\[ \limsup_{i \to \infty } \dfrac{d(f_i x_i , f_iy_i)}{d(x_i, y_i)} \leq 1.    \]
Set $r_i : = d(x_i , y_i)$ and assume, after taking a subsequence, that $d( f_ix_i , f_iy_i ) \geq (1 + \delta ) r_i$ for some $\delta > 0 $ and all $i$. By Step 1, there are subsets $A_i \subset B_{\delta r_i /10}(x_i )$, $ B_i \subset B_{\delta r_i /10}(y_i )$ with $f_i(A_i) \subset B_{\delta r_i /5}(f_ix_i )$, $f_i(B_i) \subset B_{\delta r_i /5}(f_iy_i )$, $\mm _i^1 (A_i) \geq \frac{1}{2} \mm _i^1 ( B_{\delta r_i /10}(x_i) )$, and $\mm _i^1 (B_i) \geq \frac{1}{2} \mm _i^1 ( B_{\delta r_i /10}(y_i ) )$. Since there is $C (K, N) > 0 $ such that 
\[   \mm _i^1 ( B_{2r_i}(x_i) ) \leq C \cdot \min \{ \mm _i^1 (A_i), \mm _i^1 (B_i)   \},                \]
one has
\[  \frac{\delta r_i }{ 10 \cdot C^2}  \leq  \dfrac{\int _{A_i \times B_i } \dr (f_i) d(\mm_i^1 \times \mm_i^1 ) }{ \mm _i^1 (B_{2r_i}(x_i ))^2 }  \leq  \fint _{B_{2r_i}(x_i )^{\times 2}} \dr  (f) d(\mm_i^1 \times \mm_i^1 ) \leq  2 r_i \varepsilon_i , \]
which is impossible as $\varepsilon _i \to 0$.\\

\noindent \textbf{Step 3:} For $R>0$ and distinct $x_i, y_i \in U_i^1$ with $d(x_i , p_i^1), d(y_i , p_i^1) \leq R$, one has  
\[ \limsup_{i \to \infty } \dfrac{d(f_i x_i , f_iy_i)}{d(x_i, y_i)} \leq 1.    \]
By Step 2, we can assume $d(x_i, y_i) \geq \frac{1}{2}$ for all $i$.  For each $i$, choose a sequence $x_i = z_i^0 , \ldots , z_i^k  = y_i \in X_i^1$ with $d(z_i^{j-1}, z_i^j) \leq \frac{1}{3}$ for each $j \in \{ 1, \ldots , k \}$, $d(x_i , y_i) = \sum_{j=1}^k d(z_i^{j-1}, z_i^j)$, and $k  = \lfloor 10 R \rfloor$. For each $i \in \mathbb{N}$ and $j \in \{ 1, \ldots , k\}$, let $w_i ^j \in U_i^1$ be such that 
\[d( w_i^j , z_i^j  )\leq 2 \cdot \inf \{ d(w, z_i^j  ) \vert w \in U_i^1 \}.\]
As the sets $U_i^1$ have asymptotically full measure, $\sup_j d(w_i^j , z_i^j) \to 0$ as $i \to \infty$, and the claim follows from Step 2 applied to pairs $(w_i^{j-1}, w_i^j)$.\\

\noindent \textbf{Step 4:} For $R > 0 $ and $x_i \in U_i^1$ with $d(x_i , p_i^1) \leq R$, one has  
\[ \limsup_{i \to \infty } d(f_i x_i , p_i^2 ) \leq R_0 + R + 1.    \]
As the sets $U_i^1$ have asymptotically full measure, for $i$ large enough one can pick $y_ i \in U_i^1 \cap  S_i^1 $. Then the result follows from Step 3 and the fact that $d(x_i , y_i ) \leq R+1$ for all $i$.\\

\noindent \textbf{Step 5:} For $R >0$, $x_i \in U_i^1$ with $d(x_i , p_i^1) \leq R$, and $\delta > 0 $, for large enough $i$ there is 
\[y _i \in B_{\delta}(x_i ) \cap U_i^1 \cap f_i^{-1}(U_i^2). \]
Without loss of generality assume $\delta < 1/2$. As the sets $U_i^1$ have asymptotically full measure, the sets $A_i : =  B_{\delta}(x_i ) \cap U_i^1 $ satisfy $\mm _i^1 (A_i) \geq \frac{1}{2} \mm _i^1 (B_{\delta}(x_i ))  $ for $i$ large enough. Assuming the claim fails, one has from Step 2, after taking a subsequence, that $f _ i (A_i ) \subset B_{2 \delta}( f_ix_i  ) \backslash U_i^2$ for all $i$. As $f_i$ restricted to $A_i$ is measure preserving, and the sets $U_i^2$ have asymptotically full measure, this means that 
\begin{equation}\label{eq:step-5}
  \frac{\mm _i^1 (B_{\delta}(x_i )) )   }{ \mm _i^2(B_{2 \delta}( f_ix_i  ))  }  \to 0 \text{ as }i \to \infty .    
\end{equation}  
From Step 4 we know that $B_{2 \delta}(f_ix_i) \subset B_{R_0 + R + 2} (p_i)$ for large $i$, and from the Bishop--Gromov inequality, there is $C ( K, N , R_0 , R , \delta ) > 0 $ such that 
\begin{itemize}
\item $  \mm _i^2 ( B_{R_0 + R + 2}(p_i^2 ) ) \leq C \cdot \mm _i^2 (S_i^2 \cap U_i^2) $ for large enough $i$.
\item $ \mm _i^1( B_{R_0}( p_i^1  ) )\leq C \cdot  \mm _i^1 (  B_{\delta}(  x_i ) ) $.
\end{itemize}
 Combining this with the fact that $f_i^{-1} ( S_i^2 \cap U_i^2) \subset B_{R_0}(p_i^1 )$, we get that 
\[    \frac{ \mm _i^2(B_{2 \delta}( f_ix_i  ))  }{\mm _i^1 (B_{\delta}(x_i )) )   }    \leq  C^2         \]
for $i$ large enough, contradicting \ref{eq:step-5}.\\

\noindent \textbf{Step 6:} For $R > 0 $ and distinct $x_i, y_i \in U_i^1$ with $d(x_i , p_i^1), d(y_i , p_i^1) \leq R$, one has  
\[ \lim_{i \to \infty } \vert d( f_ix_i , f_iy_i  ) - d(x_i, y_i ) \vert  =0   .    \]
From Step 3, one gets 
\[     \limsup_{i \to \infty } ( d( f_ix_i , f_iy_i  ) - d(x_i, y_i ) ) \leq 0.         \]
By Step 5, there are sequences $w_i , z_i \in U_i^1 \cap h_i (U_i^2)$ with $d(w_i , x_i)$, $d(z_i , y_i) \to 0$. By Step 2, we have $d(f_iw_i, f_ix_i), $ $d(f_iz_i , f_iy_i) \to 0$, and by Step 3 applied to $h_i$, one gets 
\[     \limsup_{i \to \infty } ( d( w_i , z_i  ) - d(f_i w_i, f_iz_i ) ) \leq 0.         \]
Hence
\[     \limsup_{i \to \infty } ( d( x_i , y_i  ) - d(f_i x_i, f_iy_i ) ) \leq 0.         \]

\noindent \textbf{Step 7: }Lemma \ref{gas:limit} holds.\\

\noindent Let $\varphi_i^j : X_i^j \to X^j \cup \{ \ast \} $ be Gromov--Hausdorff approximations and fix $\mathcal{D} \subset X^1$ a countable dense set. For $x \in \mathcal{D}$, choose $x_i \in U_i^1$ converging to $x$.  By Step 4 we can define (after taking a subsequence)  $f_{\infty }^{\prime} (x) \in X^2$ as 
\[ f _{\infty }^{\prime}(x) : =  \lim_{i \to \infty}  \varphi_i^2 f_i ( x_i ).\]
By a diagonal argument, this can be done simultaneously for all $x \in \mathcal{D}$. It is easy to see from Step 6 that $f_{\infty}^{\prime}: \mathcal{D} \to X^2$ extends to an isometry $f_{\infty} : X^1 \to X^2$ and satisfies \ref{eq:weak-limit}.
\end{proof}

\begin{Pro}\label{pro:large-intersection}
\rm  Under the conditions of Lemma \ref{gas:limit}, sequences of sets $ V_i^j \subset X_i^j$ have asymptotically full measure, then $f_i(V_i^1) \subset X_i^2$, $f_i^{-1}(V_i^2) \subset X_i^1$ have asymptotically full measure as well.
\end{Pro}

\begin{proof}
By replacing $U_i^j$ and $V_i^j$ by $U_i^j \cap V_i^j$, we can assume $U_i^j = V_i^j $ for all $j \in \{ 1, 2\}$, $i \in \mathbb{N}$. Fix $R > \delta > 0$, and consider a sequence $x_i \in B_R(p_i^1)$.  As the sets $U_i^1$ have asymptotically full measure, by Step 5 above, there is a sequence $y_ i^1 \in U_i^1 \cap f_i^{-1}(V_i^2)$ with $d(x_i,y_i^1) \to 0$. Define $y_i^2 : = f_i y_i^1$,  $A_i^j : = U_i^j \cap B_{\delta} (y_i^j)$ for $j \in \{ 1, 2 \}$. By Step 6 above, there is a sequence $\varepsilon _i \to 0$ such that
\[        f_i (A_i^1) \subset B_{\delta + \varepsilon_i} (y_i^2), f_i^{-1}(A_i^2) \subset B_{\delta + \varepsilon _i }(y_i^1) .                  \]
Then
\begin{equation*}
\begin{split}
 \lim\limits_{i \to \infty } \dfrac{\mm _i^1 (f_i^{-1}(A_i^2) \cap B_{\delta } (x_i^1))  }{\mm_i^1 ( B_{\delta } (x_i^1) ) } 
&  \geq \lim_{i \to \infty } \dfrac{\mm _i^1 (f_i^{-1}(A_i^2)  )}{\mm_i^ 1( A_i^1)  } \\
& \geq \lim\limits_{i \to \infty } \dfrac{\mm _i^2 ( A_i^2 )   }{\mm_i^2 ( f_i(A_i^1) ) } \\
& \geq \lim_{i \to \infty } \dfrac{\mm _i^2 (  B_{\delta } (y_i^2))  }{\mm _i^2 (  B_{\delta } (y_i^2)) }\\
& \geq 1.
\end{split}
\end{equation*}
This shows that $U_i^1$ has asymptotically full measure. The result for $U_i^2$ is analogous.
\end{proof}

\begin{Pro}\label{gas:zoom}
\rm Let $(X_i^j,d_i^j, \mm _i^j , p_i^j) $, $j \in \{1, 2 \}$ be a pair of sequences of pointed $\rcd ( K , N)$ spaces and $f_i : [X_i^1, p_i^1] \to [X_i^2, p_i^2] $ is a sequence of GS maps. Then there is a sequence of subsets $W_i^1 \subset X_i^1$ of asymptotically full measure with the property that for all $w_i \in W_i^1$ and $\lambda_i \to \infty$, the sequence $f_i : [ \lambda_i X_i^1 , w_i  ] \to  [\lambda_i X_i^2 , f_i(w_i)] $ is GS.
\end{Pro}

\begin{proof}
Let $ U_i ^j \subset X_i^j$ be given by Definition \ref{def:gas} and consider a sequence $\delta _i \to 0$. Set 
\[\chi_i^j :  = 1 - \chi _{U_i^j} : X_i^j \to \mathbb{R} ,\]
and $V_i ^j : = \{ x \in  U_i^j \vert \mx ( \chi_i^j)(x) \leq \delta _i  \} $. Then by Proposition \ref{pro:max-properties}(\ref{pro:max-properties-1}) and Proposition \ref{pro:large-intersection}, if $\delta_i \to 0$ slowly enough, the sets
\[W_i^1 : =  V_i ^1 \cap f_i^{-1} ( V_i^2 ) , \text{ }  W_i^2 : =  V_i^2 \cap f_i ( V_i^1 ) . \]
have asymptotically full measure. Moreover, by construction, for any sequences $\lambda_i \to \infty$ and $w_i\in W_i^1$, the sets $W_i^1$ and $W_i^2$ also have asymptotically full measure when regarded as subsets of the spaces $( X_ i^1 , \lambda_i  d_i^1, \mm _i^1 , w_i )$ and $(X_i^2, \lambda _i d_i ^2, \mm _i^2 , f_i w_i )$, respectively.

Using the sets $W_i^j$ as a replacement for $U_i^j$, all the properties of Definition \ref{def:gas} for $f_i : [\lambda _i X_i^1 ,w_i] \to [\lambda _i X_i^2, f_i (w_i)]$ follow from the ones of the original sequence, except for condition \ref{def:gas-2}, which follows from Step 1 in the proof of Lemma \ref{gas:limit}.
\end{proof}

\begin{Pro}\label{gas:composition}
\rm Let $(X_i^1j, d_i^j, \mm _i^j,p_i^j)$, $j \in \{1, 2, 3 \}$ be sequences of pointed $\rcd ( K , N)$ spaces. and $f_i : [X_i^1, p_i^1] \to [X_i^2,p_i^2] $ , $h_i : [X_i^2,p_i^2] \to [X_i^3,p_i^3]$ be sequences of GS maps. Then $h_i \circ f_i : [X_i^1 , p_i^1] \to [X_i^3 , p_i^3]$ is GS. Moreover, if $f_i$ converges weakly to $f$ and $h_i$ converges weakly to $h$, then $h_i f_i$ converges weakly to $hf$.
\end{Pro}

\begin{proof}
Let $ U_i ^1 \subset X_i^1$, $ U_i^2 \subset X_i^2, V_i^2 \subset X_i^2$, $V_i^3 \subset X_i^3$ be given by Definition \ref{def:gas} applied to $f_i $ and $h_i$ respectively. Set $W_i^{\prime} : = U_i^1 \cap f_i^{-1}(U_i^2 \cap V_i^2 \cap h_i^{-1}(V_i^3))$,  $\chi _ i : = 1 - \chi _{W_i^{\prime}}$, and for a sequence $\delta_i \to 0$, define 
\begin{equation*}
    W_i : = \{  x \in U_i^1 \vert \mx \left(\chi _ i \right) (x) \leq \delta_i \}.
\end{equation*}
By Proposition \ref{pro:max-properties}(\ref{pro:max-properties-1}), if $\delta _ i \to 0$ slowly enough, the sets $W_i \subset X_i^1$ have asymptotically full measure, so from Lemma \ref{gas:limit}, $S_i^1 : = W_i \cap B_1(p_i^1)$ and $h_i  f_i$ satisfy condition \ref{def:gas-2} from Definition \ref{def:gas}. By Step 2 in the proof of Lemma \ref{gas:limit} (applied to both $f_i$ and $f_i ^{-1}$), there is $\eta_i \to 0$ such that for all $r < 10$, $a,b \in W_i^{\prime }$ with $d(a,b) \leq 2r$, one has
\begin{equation}\label{eq:gas-composition-dist}
    \dr (f_i) (a,b) ,  \dr (h_i) (f_ia , f_i b)  \leq \eta _i \cdot d( a , b ).
\end{equation}   
This implies that $W_i$ consists of essential continuity points of $h_i f_i$ provided $\delta_i , \eta_i \leq 1/2$. Also, for $x \in W_i $, $ r\leq 1/10 $, set $Z = B_r(x) \cap W_i^{\prime}$. Then 
\begin{equation*}
\begin{split}
& \frac{1}{r} \fint _{ B_r(x)^{\times 2} } \dr (h_i  f_i) d( \mm _i^1 \times \mm_i^1 ) \\
& \leq  2 \frac{\mm _i^1 (B_r (x) \backslash Z   ) }{ \mm _i^1 (B_r( x )) }   + \frac{1}{\mm _i^1(B_r(x))^2}  \int_{Z^{\times 2}} \frac{\dr ( 
h_i  f_i )}{r} d( \mm _i^1 \times \mm_i^1 ) \\
& \leq  2 \delta_i +  \fint_{Z^{\times 2}}  \frac{ \dr (h_i)(  
f_i \cdot , f_i \cdot ) + \dr (f_i )(\cdot , \cdot ) }{r} d( \mm _i^1 \times \mm_i^1 )  \\
& \leq   2\delta_i + 4 \eta_i
\end{split}
\end{equation*} 
This shows that $h_i  f_i$ satisfies condition \ref{def:gas-3d} from Definition \ref{def:gas}, with $r \leq 1/10$ instead of $r \leq 1$. Identical arguments show that $f_i^{-1}  h_i^{-1}$ also satisfy the corresponding properties in Definition \ref{def:gas}. Conditions \ref{def:gas-1} and \ref{def:gas-3b} for $h_i  f_i$ follow from the corresponding conditions for $h_i $ and $f_i$.

Notice that the proof of Lemma \ref{gas:limit} still goes through if we replace $1$ by $1/10$ in \ref{def:gas-3d}. In particular, by Step 6, if we replace $W_i$ by $W_i \cap B_{R_i }(p_i^1)$ for some sequence $R_i \to \infty$ diverging slowly enough, then \ref{def:gas-3d} holds for $r \in [1/10, 1]$ as well, and the sequence $h_i f_i$ is GS. 

To verify the last claim, let $(X^j, d_j, \mm ^j, p^j)$  be pointed  $\rcd (K,N)$ spaces and $ \varphi _ i ^j : X_i^j \to X^j \cup \{ \ast \} $ Gromov--Hausdorff approximations for $j \in \{ 1, 2, 3 \} $. Then by hypothesis, there are $\varepsilon_i \to 0$ and sets $A_i^j \subset X_i^j $ with $j \in \{ 1, 2 \} $ having asymptotically full measure such that for $x \in A_i^1$, $y \in A_i^2$ one has
\begin{gather*}
 d(  \varphi _i ^2 f_i (x) , f _{\infty}\varphi _i ^1 (x)  ) ,  d(  \varphi _i ^3 h_i (y) , h _{\infty}\varphi _i ^2 (y)  ) \leq  \varepsilon_i .
\end{gather*}
Then by Proposition \ref{pro:large-intersection}, the sets $A_i : = A_i ^1 \cap f_i^{-1} ( A_i^2)$ have asymptotically full measure, and for all $x \in A_i$,  one has 
\begin{equation*}
\begin{split}
&d( \varphi_i^3 h_i f_i (x) ,  h_{\infty} f_{\infty}  \varphi_i^1 (x) )  \\
& \leq   d( \varphi_i^3 h_i f_i (x)  , h_{\infty} \varphi_i^2 f_i (x) ) +  d( h_{\infty} \varphi_i^2 f_i (x) , h_{\infty} f_{\infty} \varphi_i^1 (x) )  \\
& \leq 2 \varepsilon_i ,
\end{split}
\end{equation*}      
so $ h_i  f_i$ converges weakly to $h_{\infty} f_{\infty}$.
\end{proof}

\section{Construction of GS maps}\label{sec:cgas}

\begin{Lem}\label{lem:euclidean:flow}
\rm Let $(X_i, d_i, \mm _i, p_i)$ be a sequence of $\rcd \left( - \frac{1}{i} ,N \right)$ spaces,   
\[\rho_ i : (\tilde{X}_i, \tilde{d}_i, \tilde{\mm }_i, \tilde{p}_i) \to (X_i, d_i, \mm _i, p_i)\]
their universal covers,   $(Y, y)$, $( \tilde{Y}, \tilde{y} )$ a pair of pointed metric spaces, and a closed group $\Gamma \leq \iso (\mathbb{R}^k \times  \tilde{Y})$ that acts trivially on the $\mathbb{R}^k$ factor with $\tilde{Y} / \Gamma = Y$. Assume the sequences $(X_i, p_i)$ and $(\tilde{X}_i, \tilde{p}_i)$ converge in the Gromov--Hausdorff sense to $(\mathbb{R}^k \times Y , (0, y) )$ and $(\mathbb{R}^k \times \tilde{Y} , (0, \tilde{y}))$, respectively, and the sequence of groups $\pi_1(X_i)$ converges to $\Gamma$. Let $\tilde{\varphi } _i : \tilde{X}_i \to \mathbb{R}^k \times  \tilde{Y} \cup \{ \ast \} $, $\varphi_i : X_i \to \mathbb{R}^k \times  Y \cup \{ \ast \} $ be the Gromov--Hausdorff approximations given by Theorem \ref{thm:equivariant-compactness}. Then for all $s \in \mathbb{R}^k $, there is a sequence $f_i : [\tilde{X}_i , \tilde{p}_i] \to [ \tilde{X}_i , \tilde{p}_i ]$ of deck type GS maps with $(f_i)_{\ast} = \Id _{\pi_1(X_i)}$, and such that $f_i $ converges weakly to the map $\overline{s} : \mathbb{R}^{k} \times \tilde{Y} \to \mathbb{R}^k \times \tilde{Y}$, where $ \overline{s}(x,y ) : = (x+ s , y)$.
\end{Lem}

\begin{proof}
Notice that by replacing $Y$ by $s^{\perp} \times Y$, where $s^{\perp} \leq \mathbb{R} ^k$ denotes the orthogonal complement of $s$, we can assume $k =1$ and $s > 0 $.   By Lemma \ref{lem:convergence-to-delta}, there are $\delta_i \to 0$, $R_i \to \infty$, and a sequence of $L (N)$-Lipschitz functions $h^i \in H^{1,2}(X_i)$ such that
\begin{itemize}
    \item $\nabla h^i $ is divergence free in $B_{R_i}(p_i)$.
    \item For all $r \in [1, R_i]$, one has
    \[   \fint _{B_r(p_i)} \left[  \vert  \vert  \nabla h^i \vert ^2 - 1 \vert +  \vert \nabla \nabla h^i \vert^2   \right] d \mm _i \leq \delta_i^2.          \]
    \item  For all $x \in B_{R_i}(p_i)$,  one has
    \[ d(h^i (x) , \pi ( \varphi  _i (x) ) ) \leq \delta_i,\]
    where $\pi : \mathbb{R} \times Y \to \mathbb{R}$ is the projection.
\end{itemize}
Set $\tilde {h}^i : \tilde{X}_i \to \mathbb{R} $ as $\tilde{h}^i  : = h^i \circ \rho _ i $, and $\tilde{\pi}  : \mathbb{R} \times \tilde{Y} \to \mathbb{R}$ as $ \tilde{\pi} : = \pi \circ \rho$. By \ref{eq:fy} one gets for $x \in B_{R_i }(\tilde{p}_i)$, after possibly updating $\delta_i $ and $R_i$, that 
\begin{equation*} 
    \begin{split}
    & d(\tilde{h}^i (x) , \tilde{\pi} ( \tilde{ \varphi } _i (x )) )   \\
    & \leq d( h^i  ( \rho_i ( x) ), \pi ( \varphi  _i(\rho_i  (x) )) ) + d( \pi ( \varphi  _i(\rho_i  (x) )), \pi ( \rho ( \tilde{ \varphi } _i (x) )) ) \\
    & \leq  \delta_i +  \delta _ i .
    \end{split}
\end{equation*}
Then by Lemma \ref{pro:mx-lift} one gets,  after possibly updating $\delta_i $ and $R_i$, that
\begin{itemize}
    \item $\nabla \tilde{h}^i $ is divergence free in $B_{R_i}(\tilde{p}_i)$.
    \item For all $r \in [1, R_i]$, one has
    \[   \fint _{B_r(\tilde{p}_i)} \left[  \vert  \vert  \nabla \tilde{h}^i  \vert ^2 - 1 \vert +  \vert \nabla \nabla \tilde{h}^i  \vert^2   \right] d \tilde{\mm }_i \leq \delta_i^2.          \]
    \item For all $x \in B_{R_i}(\tilde{p}_i)$, one has 
    \begin{equation}\label{gh-compatible}
        d(\tilde{h}^i  (x) , \tilde{\pi} ( \tilde{ \varphi } _i (x )) ) \leq \delta _ i .
    \end{equation}          
\end{itemize}
Set $V_i : = s \nabla  \tilde{h}^i   $, $\XX^i : [0,1] \times \tilde{X}_i \to \tilde{X}_i$ the corresponding RLF, and $f_i : = \XX_1^i$. For $r \geq 1$, and $i$ large enough, using the Cauchy-Schwarz inequality and Proposition \ref{pro:max-properties}(\ref{pro:max-properties-2}), we have 
\begin{equation}\label{eq:translate-estimate}
\begin{split}
    \fint_{B_{r} (\tilde{p}_i)} & \left[ \int_0^1 \mx _4 (  \vert \nabla V_i  \vert ) (\XX^i_t(x))dt \right] d \tilde{\mm}_i (x)\\
    &  =  \int_0^1 \left[ \fint_{\XX^i_t(B_{r}(\tilde{p}_i))} \mx  _4 (\vert \nabla V_i \vert ) d \tilde{\mm}_i \right] dt\\
    &  \leq  C( N , s,  r  ) \fint_{B_{r + s L  }(\tilde{p}_i)} \mx _4 ( \vert \nabla V_i \vert  ) d \tilde{\mm}_i \\
    & \leq C(N, s, r ) \sqrt{  \fint _{B_{ r + sL}(\tilde{p}_i)} \mx _4  (\vert \nabla V_i \vert  )^2 d \tilde{\mm}_i }\\
    & \leq C(N, s, r ) \delta_i .
\end{split}
\end{equation}
Set 
\[ U_i^{\prime} : = \left\{ x \in \tilde{X}_i \bigg\vert \int_0^1 \mx _4 (\vert \nabla V_i \vert ) ( \XX_t^i(x)) dt \leq \sqrt{\delta_i} \right\} ,\]
and let $U_i \subset U_i^{\prime}$ be the density points of $U_i^{\prime}$. From \ref{eq:translate-estimate}, the sets $U_i$ have asymptotically full measure. By Theorem \ref{thm:essential-stability}, for $i$ large enough, $U_i$ consists of points of essential stability of $\XX^i$, and hence of essential continuity of $f_i$. To verify Part \ref{def:gas-2} of Definition \ref{def:gas}, we notice that  for all $i$ we have 
\[f_i (B_1(\tilde{p}_i)) \subset B_{1+sL} (\tilde{p}_i) . \]
 Applying Proposition \ref{pro:e-s-better-3} to the points in $U_i$, we see that Part \ref{def:gas-3d} of Definition \ref{def:gas} holds. The corresponding properties for $f_i^{-1}$ follow by identical arguments applied to the reverse flow, so we get that the maps $f_i$ are good at all scales and converge, by Lemma \ref{gas:limit}, to a measure preserving isometry $f_{\infty} : \mathbb{R} \times \tilde{Y} \to \mathbb{R} \times \tilde{Y} $.  

For $q \in \mathbb{R} \times \tilde{Y}$ with $d(q, (0,y)) < R $, choose $q_i \in U_i \cap f_i^{-1}(U_i)  $ converging to $q$, and  $\eta <  1/4$. Then, for $i$ large enough we have 
\begin{equation*}
\begin{split}
      \fint_{B_{\eta } (q_i)} & \left\vert \int_0^1   \vert  V_i  \vert  (\XX^i_t(x))dt - s \right\vert  d \tilde{\mm}_i (x)  \\
      &  \leq  s \fint_{B_{\eta} (q_i)} \left[  \int_0^1  \vert \vert  \nabla h_i  \vert  - 1 \vert  (\XX^i_t(x))dt \right]  d \tilde{\mm}_i (x)  \\
    &  \leq  s  \int_0^1 \left[ \fint_{\XX^i_t(B_{\eta} (q_i))}  \vert \vert \nabla h_i \vert - 1 \vert  d \tilde{\mm}_i  \right] dt\\
    &  \leq  C( N , R, s,  \eta )  \fint_{B_{\eta + R +  s L   }(p_i)} \vert \vert \nabla h_i \vert -1 \vert  d \tilde{\mm}_i  \\
    & \leq C( N , R, s,  \eta  )    \sqrt{\fint_{B_{\eta + R + s L  }(p_i)} \vert \vert \nabla h_i \vert ^2 -1 \vert d \tilde{\mm}_i  }\\
    & \leq C(N, R, s,  \eta)  \delta _ i .
\end{split}
\end{equation*}
Hence, from the derivative formula  \cite[Proposition 3.6]{D20},  and using the fact that $U_i$ have asymptotically full measure, we have 
\begin{equation}\label{eq:eta-doing-something}
  \fint _{B_{\eta}(q_i) \cap U_i } \max \left\{   d( f_ i (x), x) - s , 0 \right\} d \tilde{\mm}_i (x)  \leq    C(N, R, s, \eta  ) \delta _ i .  
\end{equation}
 From Step 2 of Lemma \ref{gas:limit}, we know that for $i$ large enough, $d( f_i(x), x)$ varies by at most $5 \eta $ for $x \in B_{\eta}(q_i) \cap U_i$. Since $\eta $ was arbitrary, \ref{eq:eta-doing-something} implies that 
\begin{equation}\label{eq:overline-s-1}
    d( f_{\infty}(q),q)  =   \lim_{i \to \infty} d( f_i (q_i), q_i ) \leq s .
\end{equation}  
Similarly, by the definition of RLF, if $\eta < 1/4$,
\begin{equation*}
\begin{split}
     & \fint_{B_{\eta } (q_i)}  \Big\vert   \left( \tilde{h}_i ( f_i (x) ) - \tilde{h}_i (x) \right)  -s  \Big\vert  d \tilde{\mm}_i (x)     \\
     & \leq  s \fint_{B_{\eta }  (q_i)} \left[ \int_0^1  \Big\vert  \vert \nabla \tilde{h}_i \vert ^2  - 1 \Big\vert  (\XX_t^i(x))  dt \right] d \tilde{\mm}_i (x)  \\
     & \leq C(N, s, R, \eta ) \cdot \delta _ i^2 .
\end{split}
\end{equation*}
Then, as $\eta $ was arbitrary, from \ref{gh-compatible} we get
\begin{equation}\label{eq:overline-s-2}
   \tilde{ \pi} ( f_{\infty } q  ) - \tilde{\pi} ( q ) = \lim_{i \to \infty} \left[ \tilde{h}_i ( f_i (q_i) ) - \tilde{h}_i (q_i) \right]   =   s   . 
\end{equation} 
Since $\overline{s}$ is the only map $\mathbb{R} \times \tilde{Y} \to \mathbb{R} \times \tilde{Y} $ satisfying \ref{eq:overline-s-1} and \ref{eq:overline-s-2}, we get $f_{\infty} = \overline{s}$.
\end{proof}

\begin{Lem}\label{lem:power-flow}
\rm Let $(X_i, d_i, \mm _i)$ be a sequence of $\rcd ( -(N-1)  ,N)$ spaces,  $V_i \in L^1 ( [0,1] ;  H^{1,2}_{C,s}(TX_i))$ a sequence of piece-wise constant on time, divergence-free vector fields, $\XX^i : [0, 1] \times X_i \to X _ i $ their RLFs, and $x_i \in X_i$ a sequence such that $x_i$ is a point of essential stability of $\XX^i$, and
\[   \int_0^1 \mx ( \vert \nabla V_i (t) \vert ^{3/2} )^{2/3} ( \XX^i_t(x_i) ) dt = \varepsilon_i .   \]
If $\varepsilon _ i \to 0$, then for all $\lambda_i \to \infty$, the sequence of maps 
\[ \XX ^i_1 :  [ \lambda _ i X_i , x_i  ] \to [\lambda _ i X_i , 
\XX^i_1 (x_i)]  \]
has the GS property. 
\end{Lem}

\begin{proof}
For $r \leq 1/4$, let 
\[ A_r^i : = \{   y \in B_r(x_i)  \vert \XX^i_t( y) \in B_{2r}(\XX^i_t( x_i )) \text{ for all }t \in [0,1 ] \} \} .\]
By Corollary \ref{cor:e-s-better} and Proposition \ref{pro:e-s-better-2}, we have for $i$ large enough
\begin{gather}
\mm _i (B_r(\XX_t(x_i))) \leq 2 \mm _i (B_r(x_i)) . \label{eq:power-flow-1} \\
\mm _i (A_r^i)  \geq (1 - C (N) \varepsilon_i)  \mm _i (B_r(x_i)) ,    \label{eq:power-flow-1-1}
\end{gather}
Also, using \ref{eq:power-flow-1} and Proposition \ref{pro:max-properties}(\ref{pro:max-properties-3}),
\begin{equation}\label{eq:power-flow-2}
\begin{split}
\fint _{A_r^i} &\int_0^1 \mx _{1/2} (\vert \nabla V_i (t) \vert) (\XX_t^i (y)) dt d \mm _i (y) \\
& = \int_0^1 \fint_{\XX^i_t( A_r^i)}   \mx_{1/2} (\vert \nabla  V_i (t) \vert) ( y) d \mm _ i  (y) dt \\
& \leq  C(N)  \int_0^1 \fint_{B_{2r}(\XX^i_t( x_i ))}   \mx _{1/2} (\vert  \nabla V_i (t) \vert) ( y) d \mm _i (y) dt\\
& \leq C(N) \int_0^1 \mx_{1/2} (  \mx_{1/2} ( \vert \nabla  V_i (t) \vert )) (\XX_t^i (x_i)) dt\\
& \leq C(N) \int_0^1 \mx (   \vert \nabla  V_i (t) \vert ^{3/2} )^{2/3} (\XX_t^i (x_i)) dt \\
& \leq  C (N) \varepsilon _i  .
\end{split}
\end{equation}
Let $U_i(r)$ be the density points of the set 
\[\left\{  y \in A_r^i \bigg\vert \int_0^1 \mx _{1/2} ( \vert  \nabla V_i (t) \vert ) (\XX_t^i(y)) dt \leq \sqrt{ \varepsilon_i }  \right\} . \]
By Theorem \ref{thm:essential-stability}, the set $U_i(r)$ consists of points of essential stability of $\XX^i$ for $i$ large enough. From  \ref{eq:power-flow-1-1} and \ref{eq:power-flow-2}, we have 
\begin{equation} \label{eq:uir-full}
\mm _i(U_i(r)) \geq (1 - C (N) \sqrt{ \varepsilon _ i}) \mm _i (B_r(x_i)) .                
\end{equation}
Given $\lambda_i \to \infty$ and $r_ i \to 0$,  by \ref{eq:uir-full} and Theorem \ref{thm:bg-in},  if $\lambda_i r_i \to \infty $ slowly enough, the sets $U_i(r_i)$ have asymptotically full measure in the spaces $( X_i , \lambda _i d_i , \mm _i, x_i )$. By Proposition  \ref{pro:e-s-better-3}, for  $y \in U_i(r_i)$, $r < 1/\lambda _ i $, we get
\[           \fint _{B_r(y) ^{\times 2} } \dr (\XX_1^i) d(\mm_ i \times \mm _i )   \leq C(N) \sqrt{ \varepsilon_i }   r    ,               \]
verifying part \ref{def:gas-3d} of Definition \ref{def:gas}. The analogue properties for $\XX_{-1}^i : [ \lambda _i X_i, \XX_1^i(x_i)] \to [\lambda _i X_i , x_i]$ follow from an identical argument. Property \ref{def:gas-2} of Definition \ref{def:gas} follows from the definition of essential stability, concluding the proof.
\end{proof}

\begin{Def}
\rm Let $X$ be a geodesic space, $ \rho : Y \to X$ a covering map, and $\varphi : [0,T] \times X \to X$ be a function such that for each $x \in X$, the map $t \mapsto \varphi (t,x)$ is continuous, and $\varphi (0,x) = x$. The \textit{lift} of $\varphi$ is defined to be the unique map $\psi : [0,T] \times Y \to Y$ such that for each $y \in Y$, the map $t \mapsto \psi (t,y)$ is continuous,  $\psi (0,y) = y$, and $\rho ( \psi (t,y) ) = \varphi (t,y)$ for all $t \in [0,T]$.
\end{Def}

Notice that if $Y$ is the universal cover of $X$, then $\psi$ is a deck type map with $\psi_{\ast} = \Id _{\pi_1(X)}$.

\begin{Pro}\label{pro:lift-of-flow}
\rm Let $(X,d,\mm , p)$ be a pointed $\rcd (-(N-1),N)$ space, $(\tilde {X} , \tilde{d}, \tilde{\mm} , \tilde{p})$ its universal cover, $V \in L^1 ( [0,T] ; L^2(TX))$ a vector field satisfying the conditions of Theorem \ref{thm:RLF-existence}, $\XX:[0,T ] \times X \to X$ its RLF, and $\tilde{V} : [0,T] \to  L^{2}_{ \loc}(T \tilde{X}) $ its lift. Then $\tilde{\XX} : [0,T] \times \tilde{X} \to \tilde{X}$, the lift of $\XX$, is the RLF of $\tilde{V}$.  Moreover, if $p$ is a point of essential stability of $\XX$, then $\tilde{p}$ is a point of essential stability of $\tilde{\XX}$.
\end{Pro}

\begin{proof}
Let $\rho : \tilde{X} \to X$ be the projection. \ref{def:rlf-1}  holds by construction. To verify \ref{def:rlf-2}, notice that by linearity, it is enough to check it for $\tilde{f} \in \test (\tilde{X}) $ supported in a ball $ \tilde{B} \subset \tilde{X}$ sent isomorphically as a metric measure space to a ball $B = \rho (\tilde { B } ) $. For such $\tilde{f}$, it induces a function $f \in \test (X)$ supported in $B$ with $\tilde{f} \vert _{\tilde{B}} =  f  \circ \rho \vert _{\tilde{B}}$. Then \ref{def:rlf-2} holds for $\tilde{\XX}$ and $\tilde{f}$ since it holds for $\XX$ and $f$ by locality of \ref{eq:rlf-def}.

To verify \ref{def:rlf-3}, consider a Borel partition $\{ E_ k \} _{k \in \mathbb{N}}$  of $\tilde{X}$ consisting of subsets sent isomorphically 
 by $\rho$ as metric measure spaces to subsets of $X$. For a Borel set $A \subset \tilde{X}$, and $t \in [0,T]$,  setting
\[A_{k,\ell } : = A\cap E_k \cap \tilde{\XX}_t^{-1}(E_{\ell}), \]
and using that $\XX$ satisfies \ref{def:rlf-3}, we get
\begin{equation*}
\begin{split}
 & \tilde{\mm} ( \tilde{\XX} _t (A) )   =  \sum_{k,\ell \in \mathbb{N}} \tilde{\mm} ( \tilde{\XX} _t (A_{k,\ell} ) ) = \sum_{k, \ell \in \mathbb{N}} \mm (\XX_t( \rho (A_{k,\ell})))   \\
 & \leq  \sum_{k,\ell \in \mathbb{N}}  C \mm (\rho ( A_{k,\ell }))    = \sum_{k,\ell} C \tilde{\mm} (  A_{k,\ell} ) = C \tilde{\mm} (A),
\end{split}
\end{equation*}
hence $\tilde{X}$ is the RLF of $\tilde{V}$. 

Now assume $p$ is a point of essential stability. Let $R \geq 1 $ be such that $\XX([0,T] \times \{ p \} ) \subset B_R(p)$. By semi-local-simple-connectedness, there is $r_0 \leq 1/10 $ such that any two curves $\alpha , \beta : [a, b ] \to B_{2R}(p)$ sharing endpoints and at uniform distance $\leq 10 r_0$, are homotopic relative to their endpoints. Then for each $t \in [0,T]$, the ball $B_{2r_0}(\tilde{\XX}_t(\tilde{p}) )$ is isomorphic as a metric measure space to $B_{2r_0}(\XX_t(p))$, so for $r \leq r_0$ small enough one has
\[  \frac{1}{M} \tilde{\mm} (  B_{r}(\tilde{p} )) \leq \tilde{\mm} (B_{r}(\tilde{\XX}_t(\tilde{p}))) \leq  M  \tilde{\mm} ( B_{r}(\tilde{p} ))  \text{ for all }t \in [0,T].          \]
 By hypothesis,  for $r \leq r_0$ small enough, there is $A_r \subset B_r(p)$ with 
\[ \mm (A_r) \geq \frac{1}{2} \mm (B_r(p)), \text{ and }
\XX_t(A_r) \subset B_{2r}(\XX_t(p)) \text{ for all }t \in [0,T]. \]
Then if $\tilde{A}_r \subset \tilde{X} $ denotes the intersection of the preimage of $A_r$ with $B_r(\tilde{p})$, one has 
\[        \tilde{ \mm} (\tilde{A}_r) \geq \frac{1}{2} \tilde{\mm} (B_r(\tilde{p})), \text{ and }
\tilde{\XX}_t(\tilde{A}_r) \subset B_{2r}(\tilde{\XX}_t(\tilde{p})) \text{ for all }t \in [0,T]    . \]
\end{proof}

\section{Re-scaling Theorem}\label{sec:res}

In this section we prove the following result, following the lines of \cite[Section 5]{KW11}.

\begin{Thm}\label{thm:re-scaling}
\rm For each $i$, let $(X_i, d_i , \mm _i, p_i) $ be an $ \rcd \left( -\frac{1}{i}, N \right) $ space of rectifiable dimension $n$, $( \tilde{X}_i, \tilde{d}_i, \tilde{\mm}_i , \tilde{p}_i  )$ be its universal cover, and define $\Gamma _ i : = \mathcal{G}(\pi_1(X_i), \tilde{X}_i , \tilde{p}_i , 1 ) $. If $(X_i, p_i)$ converges in the Gromov--Hausdorff sense to  $(\mathbb{R}^k, 0)$ with $k < n$, then, after taking a subsequence, there are sets $ \Theta _i \subset B_{1/2}(p_i) $ s.t.  $\mm _i ( \Theta _i) / \mm _i (B_{1/2}(p_i) ) \to 1$ as $i \to \infty$, a sequence $\lambda _i \to \infty$, and a compact   space $(Y, y)$ such that $Y \neq \{ \ast \}$, diam$(Y) \leq 1/10$, and 
\begin{enumerate}
\item  For all $x_i \in  \Theta _i$, after taking a subsequence, $(\lambda_i X_i , x_i) $ converges to $ (\mathbb{R}^k \times Y , (0,y_1))$ in the Gromov--Hausdorff sense ($y_1$ may depend on the $x_i$, but $Y$ doesn't), and for any lift $\tilde{x}_i \in B_{1/2}(\tilde{p}_i)$,
\[     \Gamma _ i = \mathcal{ G } ( \pi_1(X_i) , \lambda _ i X_i , x_i , 1  ) .                           \]\label{eq:re-scaling-1}
\item  For all $a_i ,b_i \in  \Theta _i$ and lifts $\tilde{a}_i, \tilde{b}_i \in B_{1/2}( \tilde{p}_i ) $, there are sequences 
\begin{gather*}
h_i : [\lambda _i X_i, a_i] \to [\lambda _i X_i , b_i]\\
f_i : [ \lambda_i \tilde{X}_i , \tilde{a}_i] \to [ \lambda_i \tilde{X}_i , \tilde{b}_i]      
\end{gather*}
of maps with the GS property such that the $f_i$ are deck maps with $(f_i)_{\ast} \in ( \Gamma_i )_{\ast}$ for all $i$, where $(\Gamma_i )_{\ast} : = \{ g_{\ast} : \pi_1(X_i) \to \pi_1(X_i) \vert g \in \Gamma_i \}$. \label{eq:adjustment}
\end{enumerate}
\end{Thm}

\begin{Lem}\label{lem:flow-in}
\rm For each $N \geq 1, L \geq 1 $, there are $R > 1$, $\delta \leq   1/100  $, such that the following holds. If $r \leq 1$,  $(X, d , \mm , p)$ is a pointed $\rcd (- \delta ,N)$ space with $d_{GH}((r^{-1}X, p), (\mathbb{R}^k, 0))\leq \delta$, and $f \in H^{1,2}( X ; \mathbb{R}^k)$ is an $L$-Lipschitz function with $f(p) = 0$ such that  $\nabla f_j$ is divergence free in $B_R(p)$ for each $j \in \{ 1, \ldots , k \} $, and for all $s \in [r, R]$, one has 
    \[   \fint_{B_s(p)} \left[ \text{ } \sum_{j_1,j_2 =1} ^k \vert \langle \nabla  f_{j_1} , \nabla f_{j_2} \rangle - \delta _{j_1, j_2 } \vert + \sum_{j=1}^k \vert \nabla \nabla (f_j) \vert ^2 \right] d \mm  \leq \delta.          \]
Then:
\begin{enumerate}
    \item For all $x_1,x_2 \in B_{10^kr}(p)$, and $r_1 , r_2 \in \left[ \frac{r}{4}, (10^k + 1 )r \right] $, one has
    \[  \mm (  B_{r_1}(x_1))  \leq 2 \cdot \dfrac{r_1^k}{r_2^k} \cdot \mm ( B_{r_2}(x_2) ).     \] \label{eq:mm-convergence}
    \item For all $x \in B_{10^kr }(p)$, if $\XX:[0,1]\times X \to X$ denotes the RLF of the vector field
    \begin{equation}\label{eq:v_x}
    V_x : = - \sum_{j=1}^k  f_j(x)  \nabla f_j  ,    
    \end{equation}
    then there is a set $A \subset B_{r/10}(x)$ of points of essential stability of $\mathbf{X}$ with 
\begin{equation}\label{eq:a-properties}
    \mm  (A) \geq \frac{1}{2} \mm (B_{r/10}(x)) \text{ and }     \mathbf{X}_1(A) \subset B_{r/5}(p) .    
\end{equation}    \label{eq:flow-right}
\end{enumerate}
\end{Lem}
\begin{proof}
  Notice that by replacing $X$ by $r^{-1}X$ and $f$ by $r^{-1}f$, we can assume $r=1$, and without loss of generality we can also assume $(X, d, \mm , p)$ is normalized. Arguing by contradiction, we get sequences $R_i \to \infty$, $\delta_i \to 0$, a sequence $(X_i, d_i, \mm _i, p_i )$  of normalized $\rcd (-\delta_i , N)$ spaces for which $(X_i,p_i)$  converges to $(\mathbb{R}^k, 0)$ in the Gromov--Hausdorff sense and $L$-Lipschitz functions $f ^i \in H^{1,2}(X_i; \mathbb{R}^k)$ with $f^i(p_i) = 0$ such that:
    \begin{itemize}
        \item $\nabla f^i_i$ is divergence free in $B_{R_i}(p_i)$ for each $j \in \{1, \ldots, k \}$, $i \in \mathbb{N}$.
        \item  For all $s \in [1,R_i]$, one has 
        \[       \fint_{B_s(p_i)} \left[ \text{ } \sum_{j_1,j_2 =1} ^k \vert \langle \nabla  f_{j_1}^i , \nabla f_{j_2}^i \rangle - \delta _{j_1, j_2 } \vert + \sum_{j=1}^k \vert \nabla \nabla (f_j^i) \vert ^2 \right] d \mm _i \leq \delta_i.                      \]
    \end{itemize}
And for each $i$, at least one of the conditions \ref{eq:mm-convergence} or \ref{eq:flow-right} fails. Notice however, that \ref{eq:mm-convergence} holds as by Corollary \ref{thm:gigli-corollary-2},  $(X_i, d_i, \mm _i, p_i)$ converges to $(\mathbb{R}^k, d^{ \mathbb{R}^k}, \mathcal{H}^k,0)$ in the measured Gromov--Hausdorff sense. For a sequence $x_i \in B_{10^k}(p_i)$, let 
\[     V ^i :  = -  \sum_{j=1}^k  f_j^i(x_i)  \nabla f_j^i    ,      \]
and $\XX^i : [0,1]\times X_i \to X_i$ its RLF. Then for $s = (k+1) \cdot 10^kL^2$ and $i$ large enough,
\begin{equation*}
\begin{split}
   & \fint_{B_{1/10}(x_i)} \int_0^1 \sum_{j_1,j_2=1}^k  \vert \langle \nabla f_{j_1}^i , \nabla f_{j_2}^i \rangle  - \delta_{j_1, j_2}   \vert ( \XX_t^i(y) )  dt d \mm_i (y)       \\
    & =  \int_0^1 \fint_{\XX_t^i(B_{1/10}(x_i))} \sum_{j_1,j_2=1}^k   \vert \langle \nabla f_{j_1}^i , \nabla f_{j_2}^i \rangle  - \delta_{j_1, j_2}   \vert ( y )  d \mm_i (y) dt \\
     & \leq  \dfrac{ \mm _i(B_{ s}(p_i)) }{\mathfrak{m_i}(B_{1/10}(x_i))}  \fint_{B_{ s}(p_i)} \sum_{j_1,j_2=1}^k   \vert \langle \nabla f_{j_1}^i , \nabla f_{j_2}^i \rangle  - \delta_{j_1, j_2}   \vert ( y )  d \mm_i (y)  \\
     &  \leq C(N,L)   \delta_i.
\end{split}
\end{equation*}
Then, from the definition of RLF, 
\begin{equation*}
\begin{split}
&\fint_{B_{1/10}(x_i)}  \vert \left( f^i( \XX_1^i (y)) - f^i(y) \right) +  f^i (x_i)   \vert    d \mm _i(y)    \\
& \leq  \sum_{j=1}^k \fint _{B_{1/10}(x_i)}  \vert \left( f^i_j( \XX_1^i (y)) - f^i_j(y) \right) +  f^i _j(x_i)   \vert    d \mm_i (y)    \\
& \leq  \sum_{j=1}^k  \fint _{B_{1/10}(x_i)}    \vert  f^i_j(x_i) \vert      \int_0^1    \vert  \vert \nabla f^i_j  \vert  ^2  - 1       \vert ( \XX^i_t(y) ) dt d \mm_i (y)                \\
& +   \sum_{\substack{ j_1, j_2=1\\ j_1 \neq j_2 }}^k  \fint _{B_{1/10}(x_i)}       \vert  f^i_{j_1}(x_i) \vert  \int_0^1  \langle \nabla f^i_{j_1}, \nabla f^i _{j_2} \rangle    ( \XX^i_t(y) ) dt d \mm_i (y)     \\
& \leq C(N,L) \delta _ i .
\end{split}
\end{equation*}
From the last assertion of Lemma \ref{lem:delta-to-convergence},  $  \vert f^i(y) - f^i (x_i) \vert \leq 3/20 $ for all $y\in B_{1/10}(x_i)$ if $i$ is large enough, so the set 
\[ A_i^{\prime  } : = \left\{ y \in B_{1/10}(x_i) \bigg\vert   \vert \XX^i_1(y) \vert  < \frac{1}{5} \right\}  \] 
satisfies $ \mm _i(A_i^{\prime}) /\mm _i(B_{1/10}(x_i)) \to 1 $. Then by Corollary \ref{cor:essential-stability}, if we define $A_i$ to be the points of $A_i^{\prime}$ that are of essential stability of $\XX^i$, we get that $\mm _i (A_i) \geq \frac{1}{2} \mm _i(B_{1/10}(x_i))$ for $i$ large enough, implying condition \ref{eq:flow-right}; a contradiction. 
\end{proof}

\begin{Lem}\label{lem:half-induction-blown}
\rm For $N \geq 1, L \geq  1$, let $\delta > 0 $ be given by Lemma \ref{lem:flow-in}. Then there are $R>1$, $C_0 > 1$, $\varepsilon_0 > 0 $ such that the  following holds. Let $r \leq 1$,  $(X, d , \mm , p)$ is a pointed $\rcd (- \delta ,N)$ space with $d_{GH}((r^{-1}X,p), (\mathbb{R}^k ,0)) \leq \delta$, and  $f \in H^{1,2}( X ; \mathbb{R}^k)$ an $L$-Lipschitz function with $f(p) = 0$ such that $\nabla f_j$ is divergence free in $B_{R}(p)$ for each $j \in \{1, \ldots, k \}$, and 
    \[   \mx _R \left(   \sum_{j_1,j_2 =1} ^k \vert \langle \nabla  f_{j_1} , \nabla f_{j_2} \rangle - \delta _{j_1, j_2 } \vert   +  \sum_{j=1}^k \vert \nabla \nabla (f_j) \vert ^2  \right) ( p )  \leq \varepsilon ^2  \leq \varepsilon_0^2.          \]  
For $x \in B_{10^kr}(p)$, let $V_x$ be given by \ref{eq:v_x}, and let $\XX:[0,1]\times X \to X$ be the RLF of $V_x$. Then there is a subset $B_{r/2}'(x) \subset B_{r/2}(x)$ of points of essential stability of $\XX$ satisfying  
\begin{gather}
\XX_1(B_{r/2}'(x) ) \subset B_r(p), \label{eq:flow-in-part-1} \\
      \mm (B_{r/2}'(x) ) \geq (1 - C _0 \varepsilon r ) \mm (B_{r/2}(x)) ,       \label{eq:flow-in-part-2}   
\end{gather}
and
\begin{equation}\label{eq:flow-in-part-3}
    \dfrac{1}{\mm (B_{r/2}(x))}    \int_{B_{r/2}'(x) } \int_0^1 \mx (  \vert \nabla V_x \vert ^{3/2} )^{2/3} (\XX_t(y)) dt d \mm (y) \leq C _0 \varepsilon r . 
\end{equation}              
\end{Lem}

\begin{proof}
Set $s = (k+1) \cdot 10^kL^2 r $ and compute, provided $R \geq 2s + 8$,
\begin{equation}\label{eq:hardest-induction}
\begin{split}
     &\fint_{B_{r/2}(x)} \int_0^1 \mx _4 (  \vert \nabla V_x \vert ^{3/2} )^{2/3} (\XX_t(y)) dt d \mm (y)       \\
    & =  \int_0^1 \fint_{\XX_t(B_{r/2}(x))} \mx _4 (  \vert \nabla V_x \vert ^{3/2} )^{2/3} d \mm  dt \\
    & \leq \dfrac{1}{\mm (B_{r/2}(x))}  \int_{B_{s}(p)}   \mx _4 (  \vert \nabla V_x \vert ^{3/2} )^{2/3}   d \mm    \\
    & = \dfrac{ \mm (B_{ s}(p)) }{\mm (B_{r/2}(x))}  \fint_{B_{s}(p)} \mx _4 (  \vert \nabla V_x \vert ^{3/2} )^{2/3}  d \mm   \\
    &  \leq C(N,L) \cdot \mx _s \left( \mx _4 ( \vert \nabla V_x \vert ^{3/2})^{2/3}  \right) (p) \\
    &  \leq C(N,L) \cdot \mx _{R} \left(  \vert \nabla V_x \vert ^{2}  \right)^{1/2} (p) \\
     &  \leq C(N,L)  \sum _ {j=1}^k \vert  f_j (x) \vert ^2 \mx _{R} \left(    \vert \nabla  \nabla f_j \vert ^{2}  \right)^{1/2} (p) \\
    &  \leq C(N,L)  \cdot  \varepsilon \cdot r  .
\end{split}
\end{equation}
Combining this with Theorem \ref{thm:essential-stability}, if 
\[A_0 : = \{ y \in B_{r/2}(x) \vert y\text{ is of essential stability of }\XX \},\]
then 
\begin{equation}\label{eq:ess-st-half}
    \mm (A_0) \geq (1 - C (N,L)  \varepsilon r) \mm (B_{r/2}(x)).
\end{equation}
From Lemma \ref{lem:flow-in}, if $\varepsilon$ is small enough, there is a set $A \subset B_{r/10}(x)  \cap A_0 $ satisfying \ref{eq:a-properties}. By \ref{eq:hardest-induction}, there is $q \in A $ with $\int_0^1 \mx ( \vert \nabla V_x \vert ) ( \XX_t (q) )dt \leq C (N,L) \varepsilon r $, and by Proposition \ref{pro:e-s-better-3}, this implies 
\begin{equation*}
 \fint_{B_{r}(q) ^{\times 2}}  \text{dist}_{r}(1) d ( \mm \times \mm )  \leq C (N,L) \varepsilon r^2 ,   
 \end{equation*}  
so 
\[ \fint_{A \times A_0}  \text{dist}_{r}(1) d ( \mm \times \mm )   \leq C (N,L) \varepsilon r^2 .      \]
Hence there is $y \in A$ such that 
\begin{equation}\label{eq:flow-in-intermediate-2}
    \fint_{ A_0}  \dr (1) (y,z)  d \mm (z)  \leq C (N,L) \varepsilon r^2 ,                     
\end{equation}  
so we define  
\[B_{r/2}'(x) : = \{ z \in  A_0  \vert \dr (1) (y,z) < r/10 \} .\]
Then for all $z \in B_{r/2}'(x) $ we have 
\begin{equation*}
\begin{split}
d(   \XX_1 (z), p ) & \leq d ( \XX_1(z), \XX_1 (y) ) + d(\XX_1(y), p ) \\
& \leq  d(z,y) + r/10 + r/5 \\
& \leq r/2 + r/10 + r/10 + r/5 < r ,
\end{split}
\end{equation*}  
so \ref{eq:flow-in-part-1} holds.  \ref{eq:ess-st-half} and \ref{eq:flow-in-intermediate-2} imply \ref{eq:flow-in-part-2}, and  \ref{eq:hardest-induction} implies \ref{eq:flow-in-part-3}.
\end{proof}

\begin{Lem}\label{lem:iterate-flows}
\rm For $N \geq 1, L \geq  1$, let $\delta > 0 $ be given by Lemma \ref{lem:flow-in}. Then there are $R \geq 1$, $C_0  \geq 1$, $\varepsilon_0  > 0$ such that the following holds. Assume $(X,d, \mm , p)$ is a pointed $\rcd (- \delta , N)$ space with $d_{GH}((X,p), (\mathbb{R}^k, 0))\leq  \delta $, and $f \in H^{1,2}( X ; \mathbb{R}^k) $ is an $L$-Lipschitz function with $f(p) = 0$ such that $ \nabla f_j$ is divergence free in $B_R(p)$ for each $j \in \{ 1, \ldots , k \}$,  and 
    \[   \mx _R \left(   \sum_{j_1,j_2 =1} ^k \vert \langle \nabla  f_{j_1} , \nabla f_{j_2} \rangle - \delta _{j_1, j_2 } \vert   +  \sum_{j=1}^k \vert \nabla \nabla (f_j) \vert ^2  \right) ( p )   \leq  \varepsilon^2 \leq \varepsilon _ 0 ^2 .         \]
Assume $p$ is an $n$-regular point with $n > k$ and set
\[ \rho  \geq \{ r \in ( 0, 1 ] \vert d_{GH} ( ( r^{-1} X, p  ) , (\mathbb{R}^k , 0) ) = \delta \} .\]
Then there is a set $G \subset B_1(p)$ with 
\[\mm (G) \geq (1 - C_0 \varepsilon ) \mm (B_1(p)),\]
a finite number of divergence-free on $B_{100 C_0 }(p)$ vector fields 
\[V_1, \ldots , V_m \in L^1([0,1] ;  H^{1,2}_{C,s}(TX))\]
with $\Vert V_j (t) \Vert_{\infty} \leq C_0 $ for all $t \in [0,1]$, $j \in \{ 1 , \ldots , m \}$, with RLFs 
\[\XX^1 , \ldots , \XX^m : [0,1] \times X \to X ,\]
 and a measurable map $\theta : G \to \{ 1, \ldots , m \} $ such that for all $y \in G$, $ y$ is a point of weak essential stability of $\XX^{\theta (y)},  \XX_1^{\theta (y)}(y) \in B_{\rho} (p), $ and   
\[    \fint_G \int_0^1 \mx \left( \vert \nabla V_{\theta (y)} (t) \vert ^{3/2}  \right) ^{2/3} ( \XX^{\theta (y)}_t(y)) dtd \mm  (y)  \leq C_0 \varepsilon  .                    \]
\end{Lem}
\begin{proof}
We will show that for each $r \leq 1$, there is $G_r \subset B_r(p)$ with 
\[  \mm (G_r) \geq (1 - C_0 \varepsilon r) \mm  (B_r(p)),         \]
a finite number of divergence-free  on $B_{ 100 C_0  }(p)$ vector fields 
\[ W_1, \ldots , W_m  \in L^1([0,1] ;  H^{1,2}_{C,s}(TX))  \]
with $\Vert W_j (t) \Vert_{\infty} \leq  C_0 r $  for all $t \in [0,1]$, $j \in \{ 1 , \ldots , m \}$, with RLFs 
\[\Phi ^1, \ldots , \Phi ^m : [0,1] \times X \to X, \]
and a measurable map $\theta _r : G_r \to \{ 1, \ldots , m \}$ satisfying that for all $y \in G_r$, $y$ is a point of weak essential stability of $\Phi  ^{\theta_r(y)}$, $\Phi  _1 ^{\theta_r(y)}(y) \in B_{\rho }(p)$, and 
\[   \fint_{G_r} \int_0^1 \left(  \mx ( \vert \nabla W_{\theta_r(y)} (t) \vert  )^{3/2}  \right)^{2/3} ( \Phi  _t ^{\theta_r (y)})(y) ) dtd \mm  (y) \leq C_0 \varepsilon r .                \]
Clearly, the claim holds for $r \leq \rho$ with $G_r = B_r (p)$ and the zero vector field. Now we check that if the claim holds for some $r \leq  10^{-k}$, then it also holds for $10^k r$. 

Choose $\{  q_1, \ldots , q_{\ell} \}$, a maximal $r/2$-separated set in $B_{10^kr}(p)$. By Lemma \ref{lem:flow-in}, one has
\[  \sum_{j=1}^{\ell} \mm (B_{r/2}(q_j)) \leq 2^{k+1} \sum_{j=1}^{\ell} \mm (B_{r/4}(q_j)) \leq 2^{k+3} \mm  ( B_{10^kr}(p) ).               \]
By Lemma \ref{lem:half-induction-blown}, if $\varepsilon $ is small enough and $R$ is large enough, for each $j \in \{ 1, \ldots , \ell \} $ there is a divergence free vector field $\overline{W}_j \in H^{1,2}_{C,s}(TX)$ such that  $\big\Vert \overline{W}_j \big \Vert_{\infty} \leq C(N,L) r$,  with RLF
\[\overline{\Phi }^j : [0,1] \times X \to X, \]
and a set $B_{r/2}'(q_j)$ of points of essential stability of $\overline{\Phi }^j$ such that 
$\overline{\Phi }^j_1(B_{r/2}'(q_j)) \subset B_r(p)$, and
\begin{gather*}
\mm (B_{r/2}'(q_j)) \geq (1 - C (N,L) \varepsilon r) \mm (B_{r/2}(q_j)) , \\
\frac{1}{\mm (B_{r/2}(q_j))} \int_{B_{r/2}'(q_j)}\int_0^1  \mx ( \vert \nabla \overline{W}_j \vert ^{3/2} )^{2/3} ( \overline{\Phi  }^j_t (y) ) dt d \mm  (y) \leq C  (N,L)\varepsilon r .
\end{gather*} 
Set 
\[  G_{10^kr} : = B_{10^kr}(p) \cap \bigcup_{j=1}^{\ell} \left(   B_{r/2}'(q_j) \cap \left( ( \overline{\Phi }_1^j )^{-1} (G_r)    \right)    \right) .      \]
Then
\begin{equation*}
    \begin{split}
        \mm (G_{10^kr}) &\geq     \mm  (B_{10^kr}(p)) - \sum_{j=1}^{\ell} \left(  \mm (B_{r/2}(q_j) ) - \mm (B_{r/2}'(q_j)) + \mm (B_r(p)) - \mm  ( G_r)      \right)\\
        & \geq  \mm  (B_{10^kr}(p)) - \sum_{j=1}^{\ell} \left(  C (N,L) \varepsilon r \mm (B_{r/2}(q_j)) + 2^{k+3} C_0 \varepsilon r    \mm (B_{r/2}(q_j))  \right)\\
        & \geq \left(  1 - 2^{k+2} ( C (N,L) \varepsilon r + 2^{k+3} C_0 \varepsilon r  )   \right) \mm (B_{10^kr}(p))\\
        & \geq (1 - C_0 \varepsilon 10^k r ) \mm (B_{10^kr}(p)),
        \end{split}
\end{equation*}
provided $C_0$ was chosen large enough, depending on $N$ and $L$. For each $y \in G_{10^kr}$, set $V_y$ as follows: let $j \in \{ 1, \ldots , \ell \} $ be the smallest index for which $y \in B_{r/2}'(q_j) \cap \left(  ( \overline{\Phi }_1 ^j )^{-1} (G_r)  \right)$. Then define 
\[    V_y (t) := \begin{cases}
    2 \overline{W}_j & \text{ if }t \in [0, \frac{1}{2} ) \\
    2 W_{\theta_r( \overline{\Phi }_1^j (y)) }(2t-1) & \text{ if }t \in [\frac{1}{2}, 1].
\end{cases}          \]
Notice that $\Vert V_y ( t) \Vert_{\infty} \leq  \max \{  C(N,L) r , 2 C_0 r    \} \leq C_0 10^kr$ for all $t \in [0,1]$, provided $C_0 $ was chosen large enough. Set $\Psi ^y$ be the RLF of $V_y$, and set 
\[B_{r/2}^{\prime \prime }(q_j) : = G_{10^k r} \cap B_{r/2}'(q_j) \backslash \cup _{\alpha = 1}^{j-1} B_{r/2}'(q_{\alpha}).\]
Then 
\begin{equation*}
    \begin{split}
        \fint_{G_{10^kr}} &\int_0^1 \mx \left( \vert \nabla V_y  \vert ^{3/2} \right) ^{2/3} (\Psi ^y_t (y) ) dtd \mm (y) \\
        &   \leq \dfrac{1}{\mm (B_{10^kr}(p))} \sum_{j=1}^{\ell} \int_{B_{r/2}^{\prime \prime }(q_j)} \int_0^1 \mx \left( \vert \nabla V_y \vert ^{3/2}   \right)^{2/3} ( \Psi^y_t(t) ) dt d \mm  (y)  \\
        & \leq \dfrac{1}{\mm (B_{10^kr}(p))} \sum_{j=1}^{\ell} \left(  C (N,L) \varepsilon r \mm (B_{r/2}(q_j))  + 2^{k+2} C_0 \varepsilon r \mm (B_{r/2}(q_j))  \right) \\
        & \leq C_0 \varepsilon 10^k r ,
    \end{split}
\end{equation*}
again provided $C_0$ was chosen large enough, depending on $N$ and $L$, concluding the proof of the lemma. 
\end{proof}
\begin{proof}[Proof of Theorem \ref{thm:re-scaling}:]
  By Lemma \ref{lem:convergence-to-delta}, there are sequences $\delta_i \to 0$, $R_i \to \infty$, and a sequence of $L(N)$-Lipschitz maps $h^i \in H^{1,2}(X_i; \mathbb{R}^k)$ with $h^i (p_i)=0$ for all $i$, $\nabla h^i_j $ divergence free in $B_{R_i}(p_i)$ for each $j \in \{ 1, \ldots , k \}$, and such that if 
  \[u_i : = \sum_{j_1, j_2 = 1}^k \vert \langle \nabla h^i_{j_1}, \nabla h^i_{j_2} \rangle -\delta _{j_1, j_2} \vert  + \sum_{j=1}^k \vert \nabla \nabla h_j^i \vert ^2, 
  \]
then for all $r \in [1, R_i ]$, one has 
\[ \fint_{B_r(p_i)} u_i d \mm _i \leq \delta_i^3.\] 
Set $ \Theta _i : = \{  x\in B_{1/2}(p_i) \vert x \text{ is }n\text{-regular, } \mx (u_i)\leq \delta_i^2 \} $. By Proposition \ref{pro:max-properties}(\ref{pro:max-properties-1}), we have $\mm _i( \Theta _i) / \mm _i(B_{1/2}(p_i)) \to 1$. Notice that, possibly after modifying $\delta_i$ and $R_i$, we may assume 
\[   \mx _{R_i} (u_i) (x)  \leq \delta_i^2 \text{ for all }x \in  \Theta _i  .  \]
For $\delta \leq 1/100 $ given by Lemma \ref{lem:flow-in}, set 
\[\lambda_i : =  \inf_{x \in  \Theta _i} \inf \{ \lambda \geq 1 \vert d_{GH}((\lambda X_i , x), (\mathbb{R}^k, 0)) \geq \delta \} .       \]
As $ \Theta _i$ consists of regular points, $\lambda_i < \infty$ for all $i$, and as $(X_i, p_i) $ converges to $(\mathbb{R}^k, 0)$ in the Gromov--Hausdorff sense, we also have $\lambda_i \to \infty$.

We claim there is a sequence $\mu _ i \to \infty$ with the property that for all $x _i \in  \Theta _i$, and any lift $\tilde{x}_i \in B_{1/2}(\tilde{p}_i)$, we have 
\begin{equation}\label{eq:huge-spectral-gap}
    \Gamma _ i  = \mathcal{G} ( \pi_1(X_i), \tilde{X}_i , \tilde{x}_i , 1/\lambda_i )  =  \mathcal {G} (\pi_1(X_i) ,\tilde{X}_i , \tilde{x}_i , \mu _i ) . 
\end{equation}
Otherwise,  by Corollary \ref{cor:empty-spec-1}, after taking a subsequence, there would be a sequence $r_i \in \sigma (  \pi_1(X_i) , \tilde{X}_i , \tilde{x}_i  ) \cap [ 1/\lambda _ i , M ] $ for some $M > 0$. After again taking a subsequence, by Lemma \ref{lem:delta-to-convergence} and Proposition \ref{pro:gigli-corollary}, we can assume $(r_i^{-1}X_i, x_i)$, $(r_i^{-1}\tilde{X}_i, \tilde{x}_i)$ converge to $(\mathbb{R}^k \times Z , (0,z) ) $, $(\mathbb{R}^k \times \tilde{Z}, (0, \tilde{Z}))$, respectively, for some spaces $Z, \tilde{Z}$ with diam$(Z) \leq 1/10$ (see Remark \ref{rem:100k}) in such a way that the sequence $\Gamma_i$ converges equivariantly to some group $\Gamma \leq \iso (\mathbb{R}^k\times \tilde{Z})$ that acts trivially on the first factor and such that $\tilde{Z} / \Gamma = Z$. By Lemma \ref{lem:co-compact-spectrum}, $r \leq 1/2 $ for all $r \in \sigma (\Gamma)$, but by construction $1 \in \sigma (\Gamma_i , r_i^{-1} \tilde{X}_i, \tilde{x}_i ) $ for all $i$, contradicting Proposition \ref{pro:spec-cont} and proving \ref{eq:huge-spectral-gap}.

Let $x_i \in  \Theta _i$ be such that 
\[\inf \{ \lambda \geq 1 \vert d_{GH}((\lambda X_i , x_i), (\mathbb{R}^k, 0)) \geq \delta \} \leq \lambda_i +1.\]
After passing to a subsequence, by Lemma \ref{lem:delta-to-convergence}, we can assume $(\lambda_i  X_i, x_i)$ converges to $(\mathbb{R}^k \times Y , (0,y))$ for some compact space $(Y,y)$ with diam$(Y) \in ( 0, 1/10 ] $.  

For $a_i, b_i \in  \Theta _i$, let $G(a_i) \subset B_1(a_i)$, $G(b_i)\subset B_1(b_i) $ be the sets given by Lemma \ref{lem:iterate-flows}, and let $U_i : = G(a_i) \cap G(b_i)$. Notice that for $i$ large enough we have 
\[ \max \{ \mm _i(B_1(a_i)), \mm _i(B_1(b_i)) \} \leq  C \cdot \mm _i (U_i ) \text{ for some }C(N). \]
For each $y \in U_i$, let $V^{a_i}_{y}$, $V^{b_i}_y \in L^1([0,1]; H^{1,2}_{C,s}(TX_i))$ denote the vector fields given by Lemma \ref{lem:iterate-flows}, define $V_y \in L^1([0,1];H^{1,2}(TX_i)) $ as 
\[    V_y (t) := \begin{cases}
    - 2 V_y^{a_i} (1-2t) & \text{ if }t \in [0, \frac{1}{2} ] \\
    \hspace{0.35cm} 2 V_y ^{b_i} (2t -1 ) & \text{ if }t \in [\frac{1}{2}, 1], 
\end{cases}          \]
and let $\XX^y : [0,1] \times X_i \to X_i$ be its RLF. Then there are measurable maps
\[ V : U_i \to L^{\infty} ( [0,1]; H^{1,2}_{C,s}(TX_i))\]
with finite image such that for all $y \in U_i$, $V_y$ is a divergence-free   on $B_{ 100 C (N) }(p_i)$  vector field with $\Vert V_y (t) \Vert_{\infty}\leq C(N)$, there is a point $y^{\prime} \in B_{1/\lambda_i}(a_i)$ of weak essential stability of $\XX^y$ with $\XX^{y}_1(y^{\prime}) \in B_{1/\lambda_i}(b_i)$, and 
\[       \fint_{U_i} \int_0^1 \mx ( \vert \nabla V_y (t) \vert ^{3/2} )^{2/3} ( \XX_t ^y (y^{\prime }) ) dtd \mm _i  (y)  \leq C(N) \delta _i  .    \]
This implies there is a sequence $y_i \in B_{1/\lambda_i}(a_i)$ and vector fields $W_i \in L^1([0,1]; H^{1,2}_{C,s}(TX_i))$ with RLFs $\XX^i : [0,1]\times X_i \to X_i$ such that $y_i$ is a point of weak essential stability of $\XX^i$, $\XX_1^i (y_i) \in B_{1/\lambda_i }(b_i)$, and 
\[    \int_0^1 \mx ( \vert \nabla W_i (t) \vert ^{3/2} )^{2/3} ( \XX_t ^i (y_i) ) dt  \leq C(N) \delta _i  .       \]
By Lemma \ref{pro:connect}, $y_i$ is a point of essential stability of $\XX^i$ for $i$ large enough, so by Lemma \ref{lem:power-flow}, we get a sequence of  GS maps $ h_i : [\lambda_i X_i, a_i] \to [\lambda _i X_i , b_i ] $. This implies, by Lemma \ref{gas:limit}, that the measured Gromov--Hausdorff limit of the sequence $(\lambda _i X_i, a_i)$ does not depend on the choice of $a_i \in  \Theta _i$.

By Propositions \ref{pro:mx-lift}, \ref{pro:lift-of-flow}, and Lemma \ref{lem:power-flow}, for lifts $\tilde{a}_i, \tilde{b}_i \in B_{1/2} (\tilde{p}_i)$, we also get a sequence of deck type GS maps $f_i^{\prime} : [ \lambda _i  \tilde {X} _i, \tilde{a}_i ] \to [\lambda _i \tilde{X}_i , b_i^{\prime}] $ with $\left( f_i^{\prime } \right)_{\ast} = \Id _{\pi_1(X_i)}$ for some  $b_i^{\prime}$ in the preimage of $b_i$ with 
\begin{equation}\label{eq:no-escape-lift}
    \tilde{d} (\tilde{a}_i, b_i^{\prime}) \leq C (N)
\end{equation}
and such that $(f_i^{\prime})_{\ast} = \Id _{\pi_1(X_i)} $.  From \ref{eq:huge-spectral-gap} and \ref{eq:no-escape-lift}, there are $g_i \in \Gamma_i $ with $g_i (b_i^{\prime}) = \tilde{b}_i$. Composing $f_i^{\prime }$ with  $g_i $,  we get a sequence of deck type GS maps $f_i : [ \lambda_i \tilde{X}_i , \tilde{a}_i  ] \to [\lambda _i \tilde{X}_i , \tilde{b}_i] $ with $(f_i)_{\ast} = (g_i )_{\ast} \in (\Gamma_i)_{\ast}$.
\end{proof}

\section{Proof of Main Theorems}\label{sec:pmt}

We now prove Theorem \ref{thm:induction-main-theorem} by reverse induction on $k$. This is done by contradiction; after passing to a subsequence, we assume the following.

\begin{Asu}\label{asu:contradiction}
\rm There is a sequence of integers $\xi _ i \to \infty$ such that no subgroup of $ \Gamma_i$ of index $\leq \xi _i$ admits a nilpotent basis of length $\leq n-k$ respected by $(f_{j,i})_{\ast}^{\xi _i !}$ for each $j \in \{ 1, \ldots , \ell \}$.
\end{Asu}

The base of induction consists of Proposition \ref{pro:base-main}. The induction step is first proved assuming that the sequences $f_{j,i}$ converge to the identity and $Y \neq \{ \ast \}$ (Proposition \ref{pro:step-easy}). Then we drop the assumption that the maps $f_{j,i}$ converge to the identity (Proposition \ref{pro:step-medium}), and the last step consists on dropping the assumption $Y \neq \{ \ast \}$ (Proposition \ref{pro:step-hard}).

 After taking a subsequence we may assume $( \tilde{X}_i , \tilde{p}_i ) $ converges to a space $ ( \mathbb{R} ^k \times \tilde{Y} , (0,\tilde{y}) ) $ and   $\Gamma_i  $ converges equivariantly to some closed group $\Gamma \leq \iso ( \mathbb{R}^k \times \tilde{Y} ) $, which by Proposition \ref{pro:gigli-corollary}, acts trivially on the $\mathbb{R}^k$ factor. By Corollary \ref{cor:split-cover}, $\tilde{Y}$ splits as a product $\mathbb{R}^m \times Z$ for some compact space $Z$, and by Corollary \ref{cor:fybc}, $\Gamma / \Gamma_0$ has an abelian subgroup of finite index generated by at most $m$ elements. After passing to a subsequence, by Lemma \ref{gas:limit} we can also assume that for each $j \in \{ 1, \ldots , \ell \}$,  $f_{j,i}$ converges weakly  to an isometry $f_{j, \infty} : \mathbb{R}^k \times \tilde{Y} \to \mathbb{R}^k \times \tilde{Y} $.

\begin{Pro}\label{pro:base-main}
\rm Theorem \ref{thm:induction-main-theorem} holds if $k = n$.
\end{Pro}

\begin{proof} By dimensionality, $\tilde{Y}$ is trivial and so is $\Gamma$. By Lemma \ref{lem:no-collapse-small-subgroup}, the sequence $\Gamma_i$ is trivial as well.
\end{proof}

\begin{Pro}\label{pro:step-easy}
\rm In the induction step, \ref{asu:contradiction} leads to a contradiction if  $Y \neq \{ \ast \}$ and $f_{j, \infty} = \Id_{\mathbb{R}^k \times \tilde{Y}}$ for all $j$.
\end{Pro}

\begin{proof}
Let $v_1, \ldots , v_{m} \in \Gamma$ be such that $\{  v_1 \Gamma_0 , \ldots , v_{m}\Gamma_0  \}$ generates a finite index abelian subgroup of $\Gamma / \Gamma_0$.  For each $j \in \{ 1, \ldots , m \}$ pick $w_j^i \in \Gamma _i$ with $w_j^i \to v_j$, and define $\Upsilon_i \triangleleft \Gamma_i$  to be the subgroups given by Theorem \ref{thm:upsilon}. Then from the proof of Theorem \ref{thm:upsilon} one has for $i$ large enough:
\begin{itemize}
\item  $\langle  \Upsilon_i , w_1^i , \ldots , w_{m}^i \rangle$ is a finite index subgroup of $\Gamma_i$.
\item $ [ w_{j_1}^i , w_{j_2}^i  ] \in \Upsilon_i $ for $j_1, j_2 \in \{  1, \ldots , m \}$.
\end{itemize}
Furthermore, as $f_{j,i} \to \Id_{\mathbb{R}^k \times \tilde{Y}}$ for each $j \in \{ 1, \ldots , \ell \}$, we also have
\begin{itemize}
    \item $[ f_{j,i} , w_{j_1}^i  ] \in \Upsilon_i $ for all $ j_1 \in \{ 1, \ldots , m \}$ and large enough $i$.
\end{itemize}
\begin{center}
\textbf{Case 1:} The sequence $ [ \Gamma_i : \Upsilon _ i ]$ is bounded. 
\end{center}
By \ref{lem:very-finite-generated} and \ref{pro:pass-to-subgroup}, there are characteristic subgroups $H_i \triangleleft  \Gamma _i $ contained in $\Upsilon_i$  such that the sequence $[\Upsilon _i : H_i ]$ is bounded. After slightly shifting the basepoints $p_i$, we may assume $y$ is an $\alpha$-regular point of $Y$.  Let $\lambda_i \to \infty$ so slowly that:
\begin{itemize}
\item $(\lambda_i X_i , p_i ) \to (\mathbb{R}^{k + \alpha}, 0)$ in the Gromov--Hausdorff sense.
\item $f_{j,i} : [ \lambda_i \tilde{X}_i  , \tilde{p}_i ] \to [\lambda _i \tilde{X}_i , \tilde{p}_i]$ still is GS and converges to $\Id_{\mathbb{R}^{m + \alpha }}$.
\item $ \lim _{i \to \infty } \sup \sigma ( H_i , \lambda_i \tilde{X}_i , \tilde{p}_i ) < \infty $. 
\end{itemize}
By Proposition \ref{pro:gigli-corollary}, any Gromov--Hausdorff limit of $(\lambda_i \tilde{X}_i / H_i , [\tilde{p}_i])$ splits off $\mathbb{R}^{k + \alpha }$, and as $H_ i \triangleleft \Gamma_i$ is characteristic, it is preserved by  $(f_{j,i})_{\ast}$ for each $j$, so the induction hypothesis applies to the spaces $ ( \tilde{X}_i / H_i, \lambda _ i d_i , \mm _i, [\tilde{p}_i] )$, contradicting \ref{asu:contradiction}.
\begin{center}
\textbf{Case 2:} After passing to a subsequence,  $ [ \Gamma_i  : \Upsilon _ i ] \to \infty$. 
\end{center}
For each $j \in \{ 1 , \ldots , m  \}$, set $\Gamma_{i,j} : = \langle \Upsilon_i , w_{j}^i , \ldots , w_m^i  \rangle $ and $\Gamma_{i , m +1 } : = \Upsilon_i$. Let $j_0 \in \{ 1 , \ldots , m \} $ be the smallest number such that, after passing to a subsequence, we get $  [ \Gamma_{i,j_0} : \Gamma_{i, j_0+1} ] \to \infty $, and let $\Gamma_i^{\prime} : = \Gamma_{i, j_0}$, $\Upsilon_i^{\prime} : = \Gamma_{i, j_0 + 1 }$. Notice that by our choice of $w_j^i$'s, $\Upsilon_i^{\prime}$ is normal in $\Gamma_{i}^{\prime}$. 

Let $X_i^{\prime} : = \tilde{X}_i / \Upsilon_i^{\prime} $, $p_i^{\prime} \in X_i^{\prime }$ the image of $\tilde{p}_i$, and $  f_{\ell + 1 , i}  : = w_{j_0}^i \in \iso ( \tilde{X}_i )$. After taking a subsequence, we can assume $(X_i^{\prime} , p_i ^{\prime}) $ converges to a space $( \mathbb{R}^k \times Y^{\prime} , (0, y^{\prime}) )$, $\Upsilon_i^{\prime}$ converges to a closed group $\Upsilon^{\prime } \leq \Gamma$, and $\Gamma_{i}^{\prime}$ converges to a closed group $ \Gamma^{\prime} \leq \Gamma$ with $[\Gamma : \Gamma^{\prime} ] < \infty$. By Theorem \ref{thm:upsilon}, $[\Gamma^{\prime} : \Upsilon ^{\prime} ] = \infty$, so the group $ \Gamma^{\prime} /  \Upsilon ^{\prime } $ is non-compact.

Since $\Gamma^{\prime } / \Upsilon^{\prime}$ acts on $Y^{\prime}$ with compact quotient $\tilde{Y} / \Gamma^{\prime} $, Corollary \ref{cor:split-cover} applies, and since $\Gamma^{\prime}/ \Upsilon ^{\prime}$ is non-compact, $Y^{\prime}$ contains a non-trivial Euclidean factor. Therefore the induction hypothesis applies to the sequence of spaces  $ ( X_i^{\prime}, p_i^{\prime} ) $,  the groups $\Upsilon _{i }^{\prime }$, and the maps $f_{j,i}$ for $j \in \{ 1, \ldots , \ell + 1 \}$ (as $f_{j,i} \to \Id_{\mathbb{R}^k \times \tilde{Y}}$, $(f_{j,i})_{\ast}$ preserves $\Upsilon _i ^{\prime} $). This means there is $C>0$ and subgroups $G_i^{\prime} \leq \Upsilon_i^{\prime} $ such that for $i$ large enough:
\begin{itemize}
\item  $[ \Upsilon_i ^{\prime} : G_i^{\prime} ] \leq C$.
\item  $G_i^{\prime} $ admits a nilpotent basis $\beta^{\prime}_i$ of length $\leq n-k-1$.
\item  $(f_{j,i})^{C!}_{\ast}$ respects $\beta ^{\prime }_i$ for $j \in \{ 1, \ldots , \ell + 1 \} $.
\end{itemize}
By \ref{lem:very-finite-generated} and \ref{pro:pass-to-subgroup}, we can assume $G_i^{\prime}$ is characteristic in $\Upsilon_i^{\prime}$. Then we define $G_i : = \langle G_i^{\prime} , f_{\ell+1, i }^{(2C)!} \rangle $. Notice that $G_i$ admits the nilpotent chain $\beta_i$ obtained by appending $f_{\ell+1, i }^{(2C)!}$ to $\beta_i^{\prime}$. From the fact that $[f_{j,i}, f_{\ell + 1, i}] \in \Upsilon _i^{\prime} $ for $j \in \{ 1, \ldots , \ell \}$ and   Proposition \ref{pro:powers}, we have that $( f_{j,i} )^{C!}_{\ast}$ respects $\beta $ for $j \in \{ 1 , \ldots , \ell \}$. Finally, the sequence  
\[ [\Gamma _ i :  G_i ] = [ \Gamma_i : \Gamma_i^{\prime} ] [ \Gamma _i^{\prime} : G_i ] \leq  [ \Gamma_i : \Gamma_i^{\prime} ] (2C)!C \]
is bounded, contradicting \ref{asu:contradiction}.
\end{proof}

\begin{Pro}\label{pro:step-medium}
\rm In the induction step, \ref{asu:contradiction} leads to a contradiction if  $Y \neq \{ \ast \}$. 
\end{Pro}

\begin{proof}
Fix $j \in \{ 1, \ldots , \ell \}$. Then $f_{j, \infty} (0, \tilde{y}) = (s , y^{\prime})$ for some $s \in \mathbb{R}^k$, $y^{\prime} \in \tilde{Y}$. By Lemma \ref{gas:composition}, after composing $f_{j,i}$ with maps given by Theorem \ref{lem:euclidean:flow}, we can assume $f_{j, \infty} (0,\tilde{y}) = ( 0 , y^{\prime})   $. Since $\Gamma $ acts co-compactly on $\tilde{Y}$, there is a sequence $\gamma_{\nu} \in \Gamma$ such that the sequence $f_{j, \infty}^{\nu} ( \gamma_{\nu} (0,\tilde{y} ))$ is bounded. As $\iso ( \mathbb{R}^k \times \tilde{Y} )$ is proper, there is a sequence $\nu _{\alpha} \to \infty $ such that $f_{j, \infty}^{\nu_{\alpha}} ( \gamma _{\nu _{\alpha}} ) $ is a Cauchy sequence in $\iso (\mathbb{R}^k \times \tilde{Y})$. This implies that the sequence 
\[( f_{j, \infty} ^{\nu_{\alpha }} \gamma_{\nu_{\alpha }} ) ^{-1} (f_{j, \infty}^{\nu_{\alpha + 1}} \gamma_{\nu_{\alpha +1}} ) = (f_{j, \infty}^{\nu_{\alpha + 1} - \nu_{\alpha}}) \left[ ( f_{j, \infty}^{  \nu_{\alpha + 1} - \nu_{ \alpha   } }  ) _{\ast}^{-1} (\gamma_{\nu_{\alpha} }^{-1} ) \right] (\gamma_{\nu_{\alpha + 1 }} )\]
converges to $\Id_{\mathbb{R}^k\times \tilde{Y}}$.

Set $\mu_{\alpha} : = \nu_{\alpha +1  } - \nu_{\alpha}$, $g_{\alpha} : = \left[ ( f_{j, \infty}^{ \mu_{ \alpha   } 
 }  ) _{\ast}^{-1} (\gamma_{\nu_{\alpha}} ^{-1} ) \right] (\gamma_{\nu_{\alpha + 1 }} )$, and choose $g_{\alpha , i} \in \Gamma_i$ such that $g_{\alpha , i}$ converges to $g_{\alpha}$ as $i \to \infty$. By Proposition \ref{gas:composition} and a diagonal argument, if a function $i \mapsto \alpha (i) $ diverges to infinity slowly enough, the  maps $ f_{j,i}^{\mu_{\alpha}} g_{\alpha , i} : [\tilde{X}_i , \tilde{p}_i] \to [\tilde{X}_i , \tilde{p}_i]  $ are GS and converge to $\Id_{\mathbb{R}^k \times \tilde{Y}}$.  By Proposition \ref{pro:replace}, if a function $i \mapsto \alpha (i)$ diverges slowly enough, we can replace $f_{j,i}$ by $f_{j,i}^{\mu_{\alpha}} g_{\alpha , i}$ and still have \ref{asu:contradiction}.  By doing this independently for each $j$, we can assume $f_{j, \infty} = \Id _{\mathbb{R}^k \times \tilde{Y}}$ for all $j \in \{ 1, \ldots , \ell \}$ and  Proposition \ref{pro:step-easy} applies.
\end{proof}

\begin{Pro}\label{pro:step-hard}
\rm In the induction step, \ref{asu:contradiction} leads to a contradiction. 
\end{Pro}

\begin{proof}
After re-scaling down each $X_i$ by a fixed factor, we can assume $f_{j,\infty }$ displaces $(0, \tilde{y})$ at most $1/10$. If $Y = \{ \ast \}$, let $\lambda_i $ and  $ \Theta _i \subset B_{1/2}(p_i)$ be given by Theorem \ref{thm:re-scaling}, and $\tilde{ \Theta }_i \subset B_{1/2}(\tilde{p}_i)$ their lifts. For each $j \in \{ 1, \ldots , \ell \}$, let $W_{j,i}^1 \subset \tilde{X}_i$ be the sets obtained by applying Proposition \ref{gas:zoom} to each $f_{j,i}$, and set 
\[W_i :  =  \tilde{ \Theta }_i \cap \bigcap_j W_{j,i}^1.\]
Then for large enough $i$, we can take $a_i \in W_i $ such that $f_{j,i} (a_i) \in W_i $ for each $j$. Let $\varphi_{j,i} : [\lambda_i \tilde{X}_i , f_{j,i}(a_i) ] \to [ \lambda _i  \tilde{X}_i , a_i]$ be the maps given by Part \ref{eq:adjustment} of Theorem \ref{thm:re-scaling}.  By Remark \ref{rem:main-characteristic} and Proposition \ref{pro:replace}, if we replace $f_{j,i}$ by $\varphi_{j,i} f_{j,i}$, we still have \ref{asu:contradiction}. Then by Part \ref{eq:re-scaling-1} of Theorem \ref{thm:re-scaling}, Proposition \ref{pro:step-medium} applies to the spaces $( \lambda_i X_i, a_i )$ and the GS maps $\varphi_{j,i} f_{j,i} : [\lambda \tilde{X}_i, a_i ] \to [\lambda_i \tilde{X}_i, a_i]$.
\end{proof}

\begin{proof}[Proof of \ref{thm:margulis}:]
Assuming the result fails, there is a sequence $(X_i, d_i, \mathfrak{m}_i, p_i)$ of pointed $\rcd (K,N )$ spaces, $\varepsilon _ i \to 0$, and integers $\xi_i \to \infty$, such that if $H_i \leq \pi_1(X_i, p_i)$ denotes the image of the map $\pi_1(B_{\varepsilon_i}(p_i), p_i) \to \pi_1(X_i, p_i)$ induced by the inclusion, then no subgroup of $H_i$ of index $\leq \xi_i$ admits a nilpotent basis of length $\leq n$. 

Taking the pointed universal covers $(\tilde{X}_i, \tilde{d}_i, \tilde{\mm}_i, \tilde{p}_i)$, for each $i$ one can identify $H_i$ with a subgroup of $\mathcal{G}(\pi_1(X_i), \tilde{X}_i, \tilde{p}_i, 2 \varepsilon_i)$. After taking a subsequence, we can assume $(X_i, p_i)$
and $(\tilde{X}_i, \tilde{p}_i)$ converge in the Gromov--Hausdorff sense to spaces $(X,p)$ and $(\tilde{X}, \tilde{p})$, respectively, and the sequence $\pi_1(X_i)$ converges equivariantly to a closed group of isometries $G \leq Iso (\tilde{X})$. 

Let $K \leq  G$ be the stabilizer of $\tilde{p}$, and let $m$ be the number of connected components of $G$ it intersects. Fix $\varepsilon > 0 $ such that the set $\{ g \in G \vert d(d \tilde{p}, \tilde{p}) \leq 2 \varepsilon \}$ intersects the same $m$ connected components of $G$ as $K$, and define 
\[H_i ^{\prime } : = \langle \{ g \in \pi_1(X_i) \vert d(g \tilde{p}_i, \tilde{p}_i) \leq \varepsilon \} \rangle.\]
After taking a subsequence, we can assume $H_i^{\prime}$ converges equivariantly to a closed group $H^{\prime} \leq G$, and let $\Upsilon_i \triangleleft H_i^{\prime}  $ be given by Theorem \ref{thm:upsilon}. Then for $i$ large enough, $H_i \leq H_i^{\prime}$ and by Theorem \ref{thm:upsilon},  $[H_i ^{\prime} : \Upsilon_i ] \leq m $. Hence no subgroup of $\Upsilon_i$ of index $\leq \xi_i / m$ admits a nilpotent basis of length $\leq n$. 

Pick $q \in B_1(p)$ a $k$-regular point, $\tilde{q} \in B_1(\tilde{p})$ a lift, and $\tilde{q}_i \in  B_1 (\tilde{p}_i)$ converging to $\tilde{q}$. If we equip $\pi_1(X_i)$ with the metric $d_0^{\tilde{p}_i}$ from \ref{d0},  then for any $ g \in B_{\delta} (\Id _{\tilde{X}_i}) $ with $\delta < 1$ one has $d(g \tilde{q}_i, \tilde{q}_i) < \delta$. Hence for all $\delta < 1 $ we have
\begin{equation}\label{eq:margulis-delta}
    B_{\delta } (\Id_{\tilde{X}_i})  \subset \{ g \in \pi_1(X_i) \vert d(g\tilde{q}_i, \tilde{q}_i ) < \delta \}  .
\end{equation}      
For a sequence $\delta_i \to 0$, define $\Gamma_ i : = \mathcal{G} ( \pi_1 (X_i), \tilde{X}_i, \tilde{q}_i , \delta_i  )$. By \ref{eq:margulis-delta} and Theorem \ref{thm:upsilon}, if $\delta_i \to 0 $ slowly enough, for all $i$ large we have $ \Upsilon_i \leq \Gamma_i $. 
Finally, consider a sequence $\lambda_i \to \infty$  diverging so slowly that $\lambda_i \delta_i \to 0$, and such that $(\lambda_i X_i, q_i)$ converges to $(\mathbb{R}^k,0)$ in the Gromov--Hausdorff sense. Then $\Gamma_i = \mathcal{G}( \Gamma_i , \lambda_i \tilde{X}_i, \tilde{q}_i, 1 )$, contradicting  Theorem \ref{thm:induction-main-theorem} with $\ell = 0$.
\end{proof}


\begin{thebibliography}{ABCDE}





\bibitem[AT14]{AT14}
L. Ambrosio and D. Trevisan, \emph{Well posedness of Lagrangian flows and continuity equations in metric measure spaces}, Analysis and PDE, \textbf{7} (2014), 1179-1234. 

\bibitem[BS10]{BS10} K. Bacher and K.T. Sturm, \emph{Localization and tensorization properties of the curvature--dimension condition for metric measure spaces.} J. Funct. Anal. \textbf{259} (2010), no. 1, 28--56. 


\bibitem[BGT12]{BGT12} E. Breuillard, B. Green and T. Tao, \emph{The structure of approximate groups.} Publ. Math. Inst. Hautes \'Etudes Sci. \textbf{116} (2012), 115--221.

\bibitem[BDS22]{BDS22}
E. Bru\'e, Q. Deng and D. Semola, \emph{Improved regularity estimates for Lagrangian flows on RCD(K,N) spaces}, Nonlinear Anal. \textbf{214} (2022), Paper No. 112609, 26 pp.


\bibitem[BPS23]{BPS23} E. Bru{\`e}, E. Pasqualetto and D. Semola,
\emph{Rectifiability of the reduced boundary for sets of finite perimeter over RCD(K,N) spaces}， 
J. Eur. Math. Soc. (JEMS) \textbf{25} (2023), no. 2, 413-465.

\bibitem[BS18]{BS18}
E. Bru\'e and D. Semola, \emph{{Regularity of Lagrangian flows over $RCD^{*}(K,N)$ spaces}}, J. Reine Angew. Math. \textbf{73} (2020), 171-203.

\bibitem[BS20]{BS20}
E. Bru\'e and D. Semola, \emph{{Constancy of the dimension for $RCD(K,N)$ spaces via
  regularity of Lagrangian flows}}, Communications on Pure and Applied Mathematics \textbf{73} (2020), no.~6, 1141-1204.
  
\bibitem[CC96]{CC96}
J. Cheeger and  T.H. Colding, \emph{{Lower bounds on Ricci curvature and the
  almost rigidity of warped products}}, Ann. of Math. \textbf{144} (1996),
  no.~1, 189-237.




\bibitem[CN12]{CN12}
T.H. Colding and A. Naber, \emph{{Sharp H{\"o}lder continuity of tangent cones
  for spaces with a lower Ricci curvature bound and applications}}, Ann. of
  Math. \textbf{176} (2012).
  
\bibitem[D20]{D20} Q. Deng,
\emph{H{\"o}lder continuity of tangent cones in RCD(K,N) spaces and applications to non-branching}, preprint, https://arxiv.org/abs/2009.07956 (2020).


\bibitem[EKS15]{EKS15} M. Erbar, K. Kuwada and K.T. Sturm, \emph{On the equivalence of the entropic curvature--dimension condition and Bochner's inequality on metric measure spaces.} Invent. Math. \textbf{201} (2015), no. 3, 993--1071

\bibitem[FY92]{FY92} K. Fukaya and T. Yamaguchi, \emph{The fundamental groups of almost non--negatively curved manifolds.} Ann. of Math. (2) \textbf{136} (1992), no. 2, 253--333.



\bibitem[G14]{G14} N. Gigli, \emph{An overview of the proof of the splitting theorem in spaces with non-negative Ricci curvature.} Anal. Geom. Metr. Spaces \textbf{2} (2014), no. 1, 169--213.



\bibitem[GT21]{GT21}N. Gigli and L. Tamanini, \emph{Second order differentiation formula on $RCD^{*}(K,N)$ spaces}, Journal of the European Mathematical Society. \textbf{23(5)} (2021), 1727-1795.




\bibitem[G81]{G81} M. Gromov, \emph{Metric structures for Riemannian and non-Riemannian spaces.} Based on the 1981 French original. With appendices by M. Katz, P. Pansu and S. Semmes. Translated from the French by Sean Michael Bates. Reprint of the 2001 English edition. Modern Birkh\"auser Classics. Birkh\"auser Boston, Inc., Boston, MA, 2007. xx+585 pp. ISBN: 978--0--8176--4582--3; 0--8176--4582--9


\bibitem[GS19]{GS19}  L. Guijarro and J. Santos-Rodr\'iguez, \emph{On the isometry group of $\rcd (K,N)$--spaces.} Manuscripta Math. \textbf{158} (2019), no. 3--4, 441--461.

  
\bibitem[KL18]{KL18}
V.~Kapovitch and N.~Li, \emph{{On the dimension of tangent cones in limit
  spaces with lower Ricci curvature bounds}}, J. Reine Angew. Math.
  \textbf{742} (2018), 263-280.

\bibitem[KW11]{KW11}
V.~Kapovitch and B.~Wilking, \emph{{Structure of fundamental groups of
  manifolds with Ricci curvature bounded below}},
  http://arxiv.org/abs/1105.5955 (2011).
  
\bibitem[K19]{K19} Y. Kitabeppu, \emph{A sufficient condition to a regular set being of positive measure on $RCD$ spaces.} Potential Anal. \textbf{51} (2019), no. 2, 179--196.


\bibitem[MN19]{MN19} A. Mondino and A. Naber, \emph{Structure theory of metric measure spaces with lower Ricci curvature bounds.} J. Eur. Math. Soc. (JEMS) \textbf{21} (2019), no. 6, 1809--1854.

\bibitem[MW19]{MW19} A. Mondino and G. Wei, \emph{On the universal cover and the fundamental group of an $\rcd (K,N)$--space.} J. Reine Angew. Math. \textbf{753} (2019), 211--237.





\bibitem[SZ23]{SZ23} J. Santos-Rodr\'iguez and S. Zamora, \emph{On fundamental groups of RCD spaces.} J. reine angew. Math. \textbf{799} (2023), 249--286.


\bibitem[SW04]{SW04} C. Sormani and G. Wei, \emph{The covering spectrum of a compact length space.} J. Differential Geom. \textbf{67} (2004), no. 1, 35--77. 

\bibitem[S18]{S18} G. Sosa, \emph{The isometry group of an $\rcd $ space is Lie.} Potential Anal. \textbf{49} (2018), no. 2, 267--286.
  
\bibitem[S93]{S93}
E.~Stein, \emph{Harmonic Analysis}, Princeton University Press, 1993




\bibitem[W22]{W22} J. Wang, \emph{RCD spaces are semi-locally simply connected}, preprint, https://arxiv.org/abs/2211.07087 (2022).


\end{thebibliography}
\end{document}